\documentclass[psamsfonts]{amsart}
\usepackage{amsmath,upref,amssymb}
\usepackage{pdfsync}
\RequirePackage{calrsfs}
\usepackage[all]{xy}
\CompileMatrices
\numberwithin{equation}{subsection} 

\theoremstyle{plain}
\newtheorem{thm}{Theorem}

\newtheorem{theorem}{Theorem}[section]

\newtheorem{lemma}[theorem]{Lemma}

\newtheorem{corollary}[theorem]{Corollary}
   
\newtheorem{definition}[theorem]{Definition}

\newtheorem{example}[theorem]{Example}

\theoremstyle{definition}

\theoremstyle{remark}

\begin{document}

\title[Highest Weight Categories For Number Rings] {Highest Weight Categories For Number Rings}

\author {Annette Pilkington}
 \address{Department of Mathematics\\ University of Notre
Dame\\ Room 255 Hurley Building \\ Notre Dame, Indiana, 46556}
\email{Pilkington.4@nd.edu}


\maketitle    

\begin{abstract}
This paper examines the concept of a stratified exact category in the context of number rings and corresponding Galois groups. 
 BGG reciprocity and duality are proven for these categories making them highest weight categories. 
The strong connections  between the structure of the category and  ramification in the ring are explored. 
\end{abstract}

\section{Introduction}

In \cite{Dyer}, Dyer introduced the concept of a stratified exact category. Such a category is constructed from a set of data $E = \{k, \mathcal{B}, \{N_x\}_{x \in \Omega}\}$, satisfying the properties listed in Section 4 of this paper.  Here  $k$ is a commutative,
Noetherian ring, $\mathcal{B}$ is an Abelian k-category and the $N_x$ are objects of $\mathcal{B}$ indexed by a finite poset $\Omega$. The category is equipped with a family of ``sheaf exact " sequences, making it an exact category in the sense of Quillen \cite{Quillen}. 
The category $\mathcal{C}$ constructed from the set of data given has,  under mild finiteness assumptions, a projective generator $P$. Using a projective generator for $\mathcal{C}$, one constructs an associated 
$k$- algebra $\mathcal{A}(P) = End(P)$, which is unique up to Morita equivalence. The functor 
$Hom(P, -)$ from $\mathcal{C}$ to $\mathcal{A}(P)$-mod  is fully faithful and exact, with the property that the  image of a sequence from  $\mathcal{C}$ is exact in $\mathcal{A}(P)$-mod if and and only if the sequence is sheaf exact in $\mathcal{C}$.  A more detailed treatment of these categories is provided in \cite{Dyer}, where  
proofs  of their fundamental properties are supplied. In \cite{Dyer2}, Dyer shows how such sets of data 
arise naturally in the cases of real reflection groups, crystollographic reflection groups, Kac-Moody Lie algebras and Quantum groups. In \cite{Brown}, Brown gives an example of how these categories can be applied to the Virasora algebra.  In this paper we will show how such a set of data arises naturally in a number theoretic situation. We will examine the resulting category and show that some of its categorical properties reflect the number theoretical properties of the underlying ring. 

Let $R$ be the ring of integers of a number field $K$ and let $S$ be its integral closure in a Galois extension $L$. Let $G = \{\sigma_{g_1}, \sigma_{g_2}, \cdots , \sigma_{g_n}\}$  be the Galois group 
of $L/K$. We let $\Omega$ be a poset giving a total ordering on the indices of $G$. We assume here, for simplicity,   that $\Omega$ gives $g_1 < g_2 < \cdots < g_n$. In section 5 we  extract a set of data using these  assumptions,  and construct a stratified exact category $\mathcal{C}_{(S\otimes_RS, \Omega, G)}$. It is the full subcategory of $S\otimes_RS$ modules, $M$,  with a filtration
\begin{equation} 
\label{1point1}
M = M^{g_0} \supseteq M^{g_1} \supseteq M^{g_2} \supseteq \cdots \supseteq M^{g_n} = 0,
\end{equation}
where $M^{g_{i - 1}}/M^{g_i}$ is  finitely generated and  projective as a right 
$S$ module and 
$$(s_1\otimes s_2)\cdot u =  us_2 \sigma_{g_i}(s_1), \ \ \hbox{for any} \ \ 
u \in M^{g_{i - 1}}/M^{g_i}.$$

 Such a filtration  is unique, hence we can refer to $M^{g_i}$ without ambiguity.
 Given a sequence $0 \to M \to N \to P \to 0$ of $S\otimes_R S$
modules in $\mathcal{C}_{(S\otimes_RS, \Omega, G)}$, we say it is sheaf exact if each 
$$0 \to M^{g_i} \to N^{g_i} \to P^{g_i} \to 0$$
is exact as a sequence of $S\otimes_R S$ modules. The category $\mathcal{C}_{(S\otimes_RS, \Omega, G)}$ equipped  with 
these sheaf exact sequences   is an exact category in the sense of 
Quillen. 

In section 6, we give an explicit construction of a projective generator $P$ for the category 
$\mathcal{C}_{(S\otimes_RS, \Omega, G)}$ and in Section 7, we determine the structure of 
the associated $S$-algebra $\mathcal{A}_{S\otimes_RS}(P)$, by representing it as a quotient of a matrix ring.  We prove a number of  results
 relating  these algebras to the structure of the ring $S$. 
 In section 8 we prove the following result:
 
 \begin{thm} Let $\frak{Q}$ be a prime ideal in $S$ and let $F_{\frak{Q}} = S/\frak{Q}$ be the residue class field of $\frak{Q}$ in $S$. Then the algebra $\mathcal{A}_{\frak{Q}} = \mathcal{A}_{S\otimes_RS}(P)\otimes_SF_{\frak{Q}}$ is a semisimple $F_{\frak{Q}}$ algebra if and only if the ideal $\frak{Q}$ is unramified in $S$.
 \end{thm}

We also prove the following duality theorem for the projective modules in
the category  $\mathcal{C}_{(S\otimes_RS, \Omega, G)}$:

\begin{thm} \label{Theorem2} Let $Q$ be a projective module in the category $\mathcal{C}_{(S\otimes_RS, \Omega, G)}$. Then 
$$ Q^{*} = Hom_{(S\otimes_R S)}(Q, S\otimes_RS)$$
is also projective in the category $\mathcal{C}_{(S\otimes_RS, \Omega, G)}$. In fact $Q^{*}$ is isomorphic to  $Q$ locally at every prime of $S$.
\end{thm}

 Further, we show that the category   $\mathcal{A}_{\frak{Q}}$-mod has a reciprocity analogous to BGG reciprocity as outlined below. 
 This reciprocity is discussed in more detail in Section 10, see 
Irving \cite{Irving} and Cline, Parshall and Scott \cite{CPS} for further details.  We find   a  set of simple modules $\{L({g_i})\}_{i = 1}^n$, which constitute a complete set of representatives of the isomorphism classes of the simple modules
in $\mathcal{A}_{\frak{Q}}$-mod.  We also find   a set of projective covers $\{P(g_i)\}_{i = 1}^n$ for 
the simple modules $\{L({g_i})\}_{i = 1}^n$ in $\mathcal{A}_{\frak{Q}}$-mod. 
If $M$ is a module in $\mathcal{A}_{\frak{Q}}$-mod, we let $[M: L(g_i)]$ denote the number of times 
that $L(g_i)$ appears as a factor in a composition series for $M$. We will show in section 10 that we can choose a set of Verma modules for the category $\mathcal{A}_{\frak{Q}}$-mod. In particular,  we 
will choose a set of modules $M(g_i)$, $1 \leq i \leq n$, universal n the algebra $\mathcal{A}_{\frak{Q}}$-mod
with respect to the following properties:
$$[M(g_i): L(g_i)] = 1, [M(g_i): L(g_j)] = 0 \ \mbox{if} \  j > i \ \mbox{and} \ \ M(g_i)/Rad(M(g_i)) = L(g_i).$$
for $1 \leq i, j \leq n$.
Each projective indecomposable $P(g_i)$ has a filtration (a Verma flag)
$$P(g_i)_0 = \{0\} \subset P(g_i)_1 \subset \cdots \subset P(g_i)_K,$$
where each quotient $P(g_i)_k/P(g_i)_{k - 1}$ is equal to $M(g_j)$ for some $j, 1\leq j \leq n$. 
Now we let  $(P(g_i) : M(g_j))$ denote the number of quotients of the form $M(g_j)$ in a Verma flag 
for $P(g_i)$. This turns out to be well defined and in Section 10, we show that we have the following reciprocity in $\mathcal{A}_{\frak{Q}}$-mod:

\begin{thm} Let $\frak{Q}$ be a prime ideal of $S$ and let $E$ denote the inertia group of $\frak{Q}$ in the Galois group $G$.  Let $L(g_i), P(g_i), M(g_i), 1\leq i \leq n$ be the modules in $\mathcal{A}_{\frak{Q}}$-mod discussed above. Then
$$(P(g_i) : M(g_j) ) =  [M(g_j): L(g_i)] = 
 \Big\{\  \begin{matrix}
& 1 & \hbox{if}  & j \geq i \ \ \hbox{and} \ \ E\sigma_{g_i} = E\sigma_{g_j} \\
& 0 &   &\mbox{otherwise} 
\end{matrix}.
$$
 \end{thm}

The duality from Theorem \ref{Theorem2} above, the existence of Verma flags for the projective indecomposables and the universal property of Verma modules ensure that the category $\mathcal{A}_{\frak{Q}}$-mod is a highest weight category as discussed in 
Cline, Parshall and Scott \cite{CPS}. 

This paper is computational in nature and many computations rely heavily on   the structure of the ring
$S\otimes_RS_{\frak{Q}}$, where $\frak{Q}$ is a prime ideal of $S$. In section 2, we derive the following  result on the structure of this ring:

\begin{thm} Let $\frak{Q}$ be a prime ideal of $S$ and let $E$ be the inertia group of $\frak{Q}$ in $G$. Then we have a set of orthogonal idempotents $\{x_i\}_{i = 1}^m$ in one to one correspondence with the cosets of $E$ in $G$ such that $S\otimes_RS_{\frak{Q}}$ is a direct product of subrings 
$(S\otimes_RS_{\frak{Q}})x_i$. Each subring $(S\otimes_RS_{\frak{Q}})x_i$ is isomorphic to the ring 
$S\otimes_{S_E}S_{\frak{Q}}$ where $S_E$ is the subring of $S$ fixed by $E$. 
\end{thm}
\noindent
\thanks {\bf Acknowledgement}  \ \ The author wishes to thank Matthew Dyer for suggesting this problem and for helpful conversations. 

\section{Notation and Definitions}
We will use the following conventional notation.  If $X$ is a set, $|X|$ will denote the number of elements in the set $X$, and  $I\subseteq X$ will indicate that $I$ is a subset of $X$, possibly equal to $X$. The empty set will be denoted by $\emptyset$.  The notation $I^c$ or $X\backslash I$  will be used to denote the complement of $I$ in $X$. 
The symbol $\cong$ will denote isomorphisms between, rings, groups and modules. We will denote which type of isomorphism is intended only when it is unclear from the context.     If $A$ is a commutative ring and $\frak{M}$  a prime ideal of $A$, then the localization of $A$ at the multiplicative subset $A\backslash  \frak{M} = \{a \in A | a \not\in \frak{M}\}$,  will be denoted by
$A_{\frak{M}}$. We ask the reader to bear in mind that the notation used for $S_E$ above is similar and hope that the meaning  will be clear from the context. The set of units of a commutative ring, $A$, will be denoted by $A^*$.  If A is a subring of the commutative ring B and 
$\beta$ is an element of $B$, then $A[\beta]$ denotes the subring of $B$ generated by $A$ and $\beta$.   If $T$ and $U$ are Dedekind domains, with $U \subseteq T$,  then we have unique factorization of ideals in both rings. If $\frak{Q}$ is an ideal of $T$ and $\frak{P}$ is an ideal of $U$, we say that $\frak{Q}$ lies above 
$\frak{P}$ if the ideal $\frak{Q}$ appears in the factorization of the ideal $\frak{P}T$ in $T$ (or equivalently if $\frak{Q} \cap U = \frak{P}$). If $\frak{P}T
= \frak{Q}^e\frak{R}$, where $\frak{R}$ is an ideal of $T$ relatively prime to $\frak{Q}$, we say $e$ is the ramification index of $\frak{Q}$ with respect to $\frak{P}$. If $e > 1$ for some  $\frak{Q}$ lying over $\frak{P}$, we say that $\frak{P}$ ramifies in $T$. If $t \in T$, and $tT = \frak{Q}^o\frak{R}$, where $\frak{R}$ is an ideal of $T$ relatively prime to $\frak{Q}$, we say that $o$ is the order of $t$ at $\frak{Q}$.  For a commutative ring $A$, $M_n(A)$ will denote the ring of $n\times n$ matrices over $A$. 

Below, we will define $L, K, G = \{\sigma_{g_1}, \sigma_{g_2}, \cdots , \sigma_{g_n}\}, R, S, \frak{Q}$
$
\frak{P}, \phi, \phi_{g_k}, \phi', \phi'_{e_k}, I_{g_i}$, $ (S\otimes_RS)^{g_i}, (S\otimes_RS)_I, L_{g_i}$
$
S_{g_i}, S_{\frak{Q}, g_i}$ and $ S[\sigma_{g_i}]$. 
These symbols will retain this meaning throughout the paper.

Let  $L$ and $K$ be  subfields of the complex numbers, $\Bbb{C}$, having finite degree as a vector space over $\Bbb{Q}$. We assume further that    $L$ is   a Galois extension of $K$ of degree $n$, with Galois  group $G = \{\sigma_{g_1}, \sigma_{g_2}, \sigma_{g_3}, \cdots , \sigma_{g_n}\}$. Let $S$ and $R$  denote the integral closures of $\Bbb{Z}$ in $L$ and $K$ respectively. Let $\frak{P}$ be a prime ideal of $R$ and let $\frak{Q}$ be a prime ideal of $S$ lying above $\frak{P}$.     We let $E$ denote the inertia group with respect to $\frak{Q}$;
$$E = E(\frak{Q}|\frak{P}) = \{\sigma \in G | \sigma(\alpha) \equiv \alpha (\hbox{mod} \  \frak{Q}) \ \hbox{for all} \ \alpha \in S \}.$$
Following the notation of \cite[Chapter 4]{Marcus}, we let $L_{E}$ denote the fixed subfield of $E$ in $L$, $S_E$, the ring of algebraic integers in $L_E$, and $\frak{Q}_E$ the unique prime of $S_E$ lying under $\frak{Q}$. We know,  from Galois theory, that $L$ is a Galois extension of $L_E$ with Galois group, $E$. We also know, by \cite[Theorem 28, Chapter 4]{Marcus}, that $\frak{Q}_E$ is totally ramified in $S$, that is  $\frak{Q}_ES = \frak{Q}^{|E|}S$ and $[L : L_E] = |E|$.

We will make frequent use of  the  homomorphism $\phi$ from the tensor product $L\otimes_KL$ to the direct sum of $n$ copies of $L$  defined as follows: 
$$\phi(l_1\otimes l_2) = (\sigma_{g_1}(l_1)l_2, \sigma_{g_2}(l_1)l_2, \sigma_{g_3}(l_1)l_2, 
\cdots , \sigma_{g_n}(l_1)l_2).$$ 
We define the homorphism  $\phi_{g_k} : L \otimes_KL \to L$ as the map $\phi$ followed by projection onto the component corresponding to $\sigma_{g_k}$, that is
$$\phi_{g_k}(l_1\otimes l_2) = \sigma_{g_k}(l_1)l_2.$$
In fact we can show using Dedkind's lemma  that  $\phi$ is an isomorphism:

\begin{lemma}
\label{phi} The map $\phi: L\otimes_KL \to L_{g_1} \oplus L_{g_2} \oplus L_{g_3} \oplus \cdots \oplus L_{g_n}$, where $L_{g_i}$ is a copy of $L$,  given above is a ring isomorphism. This isomorphism  restricts to an imbedding of the ring $S\otimes_RS$( respectively  $S\otimes_RS_\frak{Q}$) into the ring $\oplus_{i = 1}^{n}S_{g_i}$ (respetively  $\oplus_{i = 1}^{n}S_{\frak{Q}, {g_i}}$), where $S_{g_i}$ is the ring of integers of $L_{g_i}$ and $S_{\frak{Q}, {g_i}}$ is its localization at $\phi_{g_i}(1\otimes \frak{Q})$. 
\end{lemma}

\begin{proof} We can view $L\otimes_K L$ as a right vector space over $L$ with basis $\{e_j\otimes 1 , 1\leq i, j \leq n\}$,  where $e_1, e_2, \dots , e_n$ is a basis for $L$ as a vector space over $K$. On the other hand $L_{g_1} \oplus L_{g_2} \oplus L_{g_3} \oplus \cdots \oplus L_{g_n}$ is also a right vector space over $L$ in the obvious way. 	Since $\phi$ is clearly a homomorphism of vector spaces over $L$, by  comparing dimensions  we see that it suffices to show that $\phi$ is a monomorphism, in order to show that it is an isomorphism. 

Let $\sum_i e_i \otimes l_i $ be an element of $ \ker \phi$.
 Applying $\phi$ to this element  we see that
$\sum_{i = 1}^{n}(\sigma_{g_1}(e_i)l_i, \sigma_{g_2}(e_i)l_i, \cdots , \sigma_{g_n}(e_i)l_i) = (0, 0, \cdots , 0)$. 
Equivalently $(l_1, l_2, \dots , l_n)D = (0, 0, \dots , 0)$, where $D$ is the matrix given by 
$$D = \begin{pmatrix}
		\sigma_{g_1}(e_1) & \sigma_{g_2}(e_1) & \cdots & \sigma_{g_n}(e_1)\\
		\sigma_{g_1}(e_2) & \sigma_{g_2}(e_2) & \cdots & \sigma_{g_n}(e_2)\\
		\vdots & \vdots &&\vdots \\
		\sigma_{g_1}(e_n) & \sigma_{g_2}(e_n) & \cdots & \sigma_{g_n}(e_n)\\
	\end{pmatrix}.
$$
By Dedekind's lemma \cite{Jac1}[section 4.14, p.291], we have $\sigma_{g_1}, \sigma_{g_2}, \dots ,\sigma_{g_n}$ are linearly independent over $L$ and hence the matrix $D$ is invertible. Hence $(l_1, l_2, \dots , l_n)  =  (0, 0, \dots , 0)$ and $\phi$ is a monomorphism as required. 

  The second statement of the lemma is now obvious since each $\phi_{g_k}(S\otimes_RS) = S$ and 
$\phi_{g_k}(S\otimes_RS_\frak{Q}) = S_\frak{Q}$ for each $k$. 
\end{proof}

Since $E = \{\sigma_{e_1}, \sigma_{e_2}, \cdots , \sigma_{e_{|E|}}\}$ is the Galois group of $L$ over $L_E$, we also have an isomorphism
$$\phi' : L\otimes_{L_E}L \to L_{e_1} \oplus L_{e_2} \oplus \cdots \oplus L_{e_{|E| }},$$
where each $L_{e_i}$ is a copy of $L$, given by 
$$\phi'(l_1\otimes l_2) = 
 (\sigma_{e_1}(l_1)l_2, \sigma_{e_2}(l_1)l_2, \sigma_{e_3}(l_1)l_2, 
\cdots , \sigma_{e_{|E|}}(l_1)l_2).$$
We also define $\phi'_{e_k}: L_1\otimes_{L_E}L_2 \to L_{e_k}$, by $\phi'_{e_k}(l_1 \otimes l_2 )
= \sigma_{e_k}(l_1)l_2$.

We have that $S\otimes_RS$ is an $S-S$ bimodule in the usual way, with $s(s_1\otimes s_2) = ss_1\otimes s_2$ and $(s_1\otimes s_2)s = s_1\otimes s_2s$.  We can also give $S_{g_1} \oplus S_{g_2} \oplus S_{g_3} \oplus \dots \oplus S_{g_n}$ an $S-S$ bimodule structure as follows:
$(s_1, s_2, s_3,  \dots , s_n)s = (s_1s, s_2s, s_3s, \dots , s_ns)$
and  \ $s(s_1, s_2, s_3, \dots , s_n) = (\sigma_{g_1}(s)s_1, \sigma_{g_2}(s)s_2, \sigma_{g_3}(s)s_3, \dots , \sigma_{g_n}(s)s_n)$.  
The  bimodule structures above give both modules the structure of an $S\otimes_RS$ module and 
$\phi: S\otimes_RS \to S_{g_1} \oplus S_{g_2} \oplus S_{g_3} \oplus \dots \oplus S_{g_n}$ is  homomorphism of 
$S\otimes_RS$ modules. 
In fact, from the lemma above, we see that we have a monomorphism 
 $\phi  :  S \otimes_RS_{\frak{Q}}   \to
(S_{\frak{Q}})_{g_1} \oplus (S_{\frak{Q}})_{g_2} \oplus (S_{\frak{Q}})_ {g_3} \oplus \dots \oplus (S_{\frak{Q}})_{g_n}$
between  the localizations of these modules at the prime ideal  $\frak{Q}$, when we regard them as right $S$-modules. 
 We have similar results for $\phi'$. 

 We can also endow a single  copy of $S$ with an $S-S$ bimodule structure as follows: for a given $\sigma_{g_i} \in
G$, we let $S[\sigma_{g_i}]$ be the $S-S$ bimodule which is $S$ as a right $S$-module and  $su = u\sigma_{g_i}(s)$ for all $u \in S[\sigma_{g_i}]$ and $s \in S$. Again this gives $S[\sigma_{g_i}]$ the structure of an $S\otimes_RS$ module and 
we see that $\phi_{g_i}: S\otimes_RS \to S$ is in fact an  $S\otimes_RS$ module     homomorphism from $S\otimes_R S$ into $S[\sigma_{g_i}]$.  Similarly we see that $\phi'_{e_k} : S\otimes_{S_E}S \to S[\sigma_{e_k}]$
is a homomorphism of $S\otimes_RS$ modules.

Let $S, R, L, K,  G = \{\sigma_{g_1}, \sigma_{g_2}, \dots , \sigma_{g_n}\}, \phi$ and $\phi_{g_k}$  be as defined above. Let $I$ be a subset of $G$, we define
$$(S\otimes_RS)_I = \{x \in S\otimes_RS | \phi_{g_i}(x) = 0 \ \hbox{for all} \  i \ \hbox{such that} \ \sigma_{g_i} \not\in I\}$$
 and 
$(S\otimes_RS_\frak{Q})_I = \{x \in S\otimes_RS_\frak{Q} | \phi_{g_i}(x) = 0 \ \hbox{for all} \  i  \ \hbox{such that} \ \sigma_{g_i} \not\in I\}$ 
it is not difficult to see, that $(S\otimes_RS)_I  \ \{(S\otimes_RS_\frak{Q})_I\}$ is an ideal  of $(S\otimes_RS) \  \{\mbox{resp} \ (S\otimes_RS_\frak{Q})\}$, by examining the images when $\phi$ is applied.

Since $(S\otimes_RS_{\frak{Q}})_I  \cap S\otimes_RS = (S\otimes_RS)_I$, we see that 
$((S\otimes_RS)_I)_{\frak{Q}} = (S\otimes_RS_\frak{Q})_I$ as follows:   for $x \in S\otimes_RS$, given $k$ such that 
 $1 \leq k \leq n$, we have 
$x \in \ker \phi_{g_k}$ if and only if $x(1 \otimes s^{-1}) \in \ker \phi_{g_k}$ for every $s \in S\backslash \frak{Q}$.  We see that
$((S\otimes_RS)_\frak{Q})_I  = \{x(1\otimes s^{-1}) \  |  \ x \in S\otimes_RS,  \ s \in S\backslash \frak{Q},  \  \phi_{g_k}(x(1\otimes s^{-1}) )  = 0 \ \mbox{for all} \ k \ \mbox{such that} \ \sigma_{g_k} \not\in I \} =  \{x(1\otimes s^{-1}) \  |  \ x \in S\otimes_RS,  \ s \in S\backslash \frak{Q}, \  \phi_{g_k}(x) = 0 \ \mbox{for all} \ k \ \mbox{such that} \ \sigma_{g_k} \not\in I \}  = ((S\otimes_RS)_I)_{\frak{Q}}$

We will sometimes  use a different  notation for some special cases of the $(S\otimes_RS)_I  \ \{(S\otimes_RS_\frak{Q})_I\}$ constructed above, when they appear in filtrations of modules. We let 
$$(S\otimes_RS)^{g_i} = \{x \in S\otimes_RS | \phi_{g_k}(x) = 0 \ \hbox{for all} \  k \leq i\} = (S\otimes_RS)_{I_{g_{i+1}}},$$ 
where $I_{g_i} = \{\sigma_{g_{i }}, \sigma_{g_{i+1}}, \dots , \sigma_{g_n}\}, i  \leq n$ and $I_{g_{n+1}} = \emptyset$.  We use a similar definition for $(S\otimes_RS_{\frak{Q}})^{g_i}$.   This notation is consistent with that used by Dyer in \cite{Dyer} and \cite{Dyer2} and is itself a special case of the notation used for filtrations of modules.  In general if $X = \{\sigma_{x_1}, \sigma_{x_2}, \cdots \sigma_{x_m}\}$ is a set of ring endomorphisms of $S$, indexed by the poset $\{x_1, x_2, \dots , x_n\}$ with ordering $x_1 < x_2 < \dots < x_n$,  we 
let $I_{x_i} = \{\sigma_{x_i}, \sigma_{x_{i + 1}}, \cdots , \sigma_{x_n}\}$. 

We also use the following slight abuse of notation throughout the paper. Let $\theta \in S$ such that $L = K(\theta)$. We say that $S\otimes_RS$    = $R[\theta] \otimes_RS$ (respectively $S\otimes_RS_{\frak{P}}$ =  $R[\theta]\otimes_RS_{\frak{P}}$) if 
each $x \in S\otimes_RS$ (resp.  $S\otimes_RS_{\frak{P}}$)  has the form $x = \sum f_i(\theta)\otimes s_i$ where $f_i(\theta) \in R[\theta]$ and 
$s_i \in S$ (resp.  $S_{\frak{P}}$).  Note that $R[\theta] \otimes_RS$ (resp.  $R[\theta]\otimes_RS_{\frak{P}}$) is isomorphic to a subring of $S\otimes_RS$ (resp.  $S\otimes_RS_{\frak{P}}$)  since $S$ is projective as an $R$ module, hence flat, and localized rings are also flat. 
Here we are identifying $R[\theta] \otimes_RS$  and  $R[\theta]\otimes_RS_{\frak{P}}$ with their isomorphic images. 

\section{The $S_\frak{Q}$-module structure of $S\otimes_{R} S_\frak{Q}$.}

Let  $L, K , S, R, G = \{ \sigma_{g_1}, \sigma_{g_2}, \cdots , \sigma_{g_n}\}$  and $I_{g_i}, 1\leq i \leq n$ be as defined in  Section 2.  We will see later that 
the $S\otimes_RS$ module $\frak{S} = S_{g_1} \oplus S_{g_2} \oplus \dots \oplus S_{g_n}$ has a natural filtration of the type  given in \eqref{1point1}, namely with $\frak{S}^{g_i} = 0 \oplus 0 \oplus \dots \oplus 0 \oplus S_{g_{i+1}} \oplus S_{g_{i+2}} \oplus \dots S_{g_n}$. 
We wish to take advantage of this natural filtration of the $S\otimes_RS$ module  $\frak{S} $ and the imbedding  $\phi$ to provide a similar filtration  for $S\otimes_R S$. However since 
$\phi: S\otimes_RS \to \frak{S}$ is not always onto, even after localization, in order to gain some insight into the structure of the modules in this filtration,  we will  analyze the structure of the  localization, $S\otimes_{R} S_\frak{Q}$,  in detail in this section. 
We will, in fact,  show that the ring $S\otimes_RS_\frak{Q}$ splits into a direct product of  isomorphic subrings, which have a relatively simple structure.

 Our first lemma demonstrates the simplicity of the structure  of $S\otimes_RS$ in the case where $S\otimes_RS = R[\theta] \otimes_RS$ for some $\theta \in S$, where $L = K(\theta)$.  
In particular, in this case, the ideals  $(S\otimes_RS)_{I_{g_i}}$,  turn out to be principal. We will see later that these ideals will in fact
give us  a filtration of $S\otimes_RS$. 
 
 \begin{definition} \label{Atheta}  Let T, U and R be commutative rings, where $R$ is a subring of both $T$ and $U$. Let $G = \{\sigma_{g_1}, \sigma_{g_2}, \cdots,
 \sigma_{g_n}\}$ be  such that each $\sigma \in G$ is a map from  $T$ to $U$ fixing $R$. Let $\theta \in U$. We define $A_{g_i}(\theta)$, $1 \leq i \leq n$  to be the following element of $T\otimes_RU$:
 $$A_{g_i}(\theta) = \theta \otimes 1 - 1 \otimes \sigma_{g_i}(\theta).$$
 We define $A_{g_0}(\theta) = 1\otimes 1$. 
 \end{definition}

 \begin{lemma}\label{alphan}  Let $L, K, S, R$ and $G = \{\sigma_{g_1}, \sigma_{g_2}, \dots , \sigma_{g_n}\}$   be as defined in Section 2.  Let  $\theta \in S$ such that $L = K(\theta)$ and $S\otimes_RS = 
R[\theta]\otimes_RS$. Then the product $A_{g_1}(\theta)A_{g_2}(\theta)\cdots A_{g_n}(\theta) = 0\otimes 0$ in $S\otimes_RS$.
\end{lemma}
\begin{proof}  Letting $f(x) = k_0 + k_1x + \dots + k_nx^n$ denote the minimum polynomial of $\theta$ over $K$, we see that 
$f(x) = \prod_{i = 1}^n(x - \sigma_{g_i}(\theta))$ and has coefficients in K. Hence  $\prod_{i = 1}^nA_{g_i}(\theta) = \sum_{i = 0}^n(1\otimes k_i)(\theta \otimes 1)^i = (\sum_{i = 0}^nk_i \theta^i)\otimes 1 = 0 \otimes 1$. 
\end{proof}

\begin{lemma} \label{basis}  Let $L, K, S, R, G = \{\sigma_{g_1}, \sigma_{g_2}, \dots , \sigma_{g_n}\}, \phi$,  $\phi_{g_k}$ and $I_{g_i}$  be as defined in Section 2. Let $\theta \in S$ such that $L = K(\theta)$ and $S\otimes_RS = 
R[\theta]\otimes_RS$. Then the elements  
$$A_{g_0}(\theta), \ A_{g_1}(\theta) , \ A_{g_1}(\theta)A_{g_2}(\theta), \ A_{g_1}(\theta)A_{g_2}(\theta)A_{g_3}(\theta), \dots , A_{g_1}(\theta)A_{g_2}(\theta)\cdots A_{g_{n- 1}}(\theta)$$
form a basis for $S\otimes_RS$ as a right $S-module$. 

In addition, for   $1 \leq i \leq n+1$,  the ideal $(S\otimes_RS)_{I_{ g_i}}$ of $S\otimes_RS$ is principal
and $(S\otimes_RS)_{I_{g_i}} = (\prod_{j \leq i - 1}A_{g_j}(\theta))(S\otimes_RS)$. 
\end{lemma}

\begin{proof} Since $1, \theta, \theta^2, \dots , \theta^{n - 1}$ are linearly independent over $K$, we have $1\otimes 1, \theta\otimes 1, \theta^2\otimes 1, \dots , \theta^{n - 1} \otimes 1$ form a basis for  $S\otimes_RS$ as a right $S$-module.  Hence the monic polynomials in $\theta\otimes 1$, of degrees $0$ through $n - 1$ respectively:
$$A_{g_0}(\theta), \ A_{g_1}(\theta) , \ A_{g_1}(\theta)A_{g_2}(\theta), \  A_{g_1}(\theta)A_{g_2}(\theta)A_{g_3}(\theta), \dots , A_{g_1}(\theta)A_{g_2}(\theta)\cdots A_{g_{n- 1}}(\theta),$$
must also form a basis for  $S\otimes_RS$ as a right $S$-module. Since $\phi_{g_k}(\prod_{j \leq i - 1}A_{g_j}(\theta)) =0$, for $k \leq i - 1$, it is obvious that $(\prod_{j \leq i - 1}A_{g_j}(\theta))(S\otimes_RS) \subseteq 
(S\otimes_RS)_{I_{ g_i}}$.  Since the field automorphisms, $\{\sigma_{g_1}, \sigma_{g_2}, \dots , \sigma_{g_n}\}$ are distinct, we have $\sigma_{g_i}(\theta) - \sigma_{g_j}(\theta) \not= 0$ if $i \not= j$.  Hence $\phi_{g_k}(\prod_{j\leq i - 1}A_{g_j}(\theta)) \not= 0$ if $k \geq  i$.  If $x \in (S\otimes_RS)_{I_{g_i}}$, then $x$ has the form 
{\small $$ x = (A_{g_0}(\theta))s_0 + (A_{g_1}(\theta))s_1 + (A_{g_1}(\theta)A_{g_2}(\theta))s_2 +   \dots  + (A_{g_1}(\theta)A_{g_2}(\theta)\cdots A_{g_{n-1}}(\theta))s_{n-1},$$}
where $s_i \in S$. Since $\phi_{g_k}(x) = 0$ for $k \leq i - 1$, we see that $\phi_{g_1}(x) = s_0 = 0$. Now   $\phi_{g_2}(x) = 
\phi_{g_2}(A_{g_1}(\theta)s_1) = (\sigma_{g_2}(\theta) - \sigma_{g_1}(\theta))s_1$, giving that $s_1 = 0$, since $S$ is a domain.   By an inductive argument, we get that $s_k = 0$ for $k < i - 1$ and hence $x \in 
(\prod_{j\leq i}A_{g_j}(\theta))(S\otimes_RS)$. This shows that $(S\otimes_RS)_{I_{ g_i}}$ is a principal ideal of $S\otimes_RS$.
\end{proof}

The same argument gives us the result when  $S\otimes_RS_{\frak{Q}} = R[\theta]\otimes_RS_{\frak{Q}}$.

\begin{lemma} \label{basisP} Let $L, K, S, R, \frak{Q},  G = \{\sigma_{g_1}, \sigma_{g_2}, \dots , \sigma_{g_n}\}, \phi$,   $\phi_{g_k}$ and $I_{g_i}$  be as defined in Section 2. Let $\theta \in S$ such that $L = K(\theta)$ and $S\otimes_RS_{\frak{Q}} = 
R[\theta]\otimes_RS_{\frak{Q}}$. Then the elements  
$$A_{g_0}(\theta), \ A_{g_1}(\theta) , \ A_{g_1}(\theta)A_{g_2}(\theta), \ A_{g_1}(\theta)A_{g_2}(\theta)A_{g_3}(\theta), \dots , A_{g_1}(\theta)A_{g_2}(\theta)\cdots A_{g_{n- 1}}(\theta)$$
form a basis for $S\otimes_RS_{\frak{Q}}$ as a right $S_{\frak{Q}}-module$. 

In addition, for   $1 \leq i \leq n+1$,  the ideal  $(S\otimes_RS_{\frak{Q}})_{I_{ g_i}}$ of $S\otimes_RS_{\frak{Q}}$ is principal
and $(S\otimes_RS_{\frak{Q}})_{I_{ g_i}} = (\prod_{j \leq i - 1}A_{g_j}(\theta))(S\otimes_RS_{\frak{Q}})$. 
\end{lemma}

Such a structure is not guaranteed for $S\otimes_RS$, in fact not even for $S\otimes_RS_{\frak{Q}}$. However we can show that $S\otimes_RS_{\frak{Q}}$ splits into isomorphic subrings of this type. We start with a basic lemma about the structure of $S$.

\begin{lemma} \label{Pie}  Let $S, R, \frak{Q}, \frak{P}, E = E(\frak{Q}|\frak{P})$, $S_E$,  and $\frak{Q}_E$  be as defined in Section 2. Then there exists an element  $\Pi$ of order $1$ in $\frak{Q}$, with $\Pi^{|E| } \in \frak{Q}_E$,  such that 
$$S \subseteq  (S_E)_{\frak{Q}_E}[\Pi].$$
\end{lemma}

\begin{proof} We will  apply \cite{Lang}[Chapter I,  Section 7,  Proposition 23].  First we  show that the rings examined below satisfy the conditions of the proposition. 
  The localization of $S_E$ at $\frak{Q}_E$, $(S_E)_{\frak{Q}_E}$, is a discrete valuation ring with quotient field, $L_E$.  It is not difficult to see, by manipulating monic polynomials,  that the integral closure of $(S_E)_{\frak{Q}_E}$ in $L$  is $S_{\frak{Q}_E}$, where $S_{\frak{Q}_E}$ denotes the localization of $S$ regarded as an $S_E$-module at the ideal $\frak{Q}_E$. 
 If $M$ is a maximal ideal of $S_{\frak{Q}_E}$, then we must have $M \cap (S_E)_{\frak{Q}_E} = \frak{Q}_E(S_E)_{\frak{Q}_E}$, since $(S_E)_{\frak{Q}_E}$ is local and has a unique maximal ideal. 
 Hence if $M$ is a maximal ideal of  $S_{\frak{Q}_E}$, we have  $M \cap S$ is a maximal ideal of $S$ lying above $\frak{Q}_E$, by the theory of localization, see for example \cite{Reiner}[Section 3c]. Since $\frak{Q}$ is the unique ideal of $S$ lying above $\frak{Q}_E$, by \cite{Marcus}[Chapter 4, Theorem 28],  we have  $M\cap S = \frak{Q}$ and 
$M = (M\cap S)S_{\frak{Q}_E} = \frak{Q}S_{\frak{Q}_E}$. Hence $S_{\frak{Q}_E}$ has a unique maximal ideal, $\frak{Q}S_{\frak{Q}_E}$ which lies above $\frak{Q}_E$. 

The imbedding of the residue class field $S/\frak{Q}$ into the  residue class field of  $S_{\frak{Q}_E}$ 
modulo $\frak{Q}S_{\frak{Q}_E}$ is an isomorphism, by the theory of localizations, see for example \cite{Reiner}[Theorem 3.13].  Likewise the imbedding of the residue class field $S_E/\frak{Q}_E$ into $S/\frak{Q}$ is an isomorphism,
see  \cite{Marcus}[Chapter 4, Theorem 28]. Hence the residue class  field $S_{\frak{Q}_E}/\frak{Q}S_{\frak{Q}_E}$ is a trivial extension of  $S_E/\frak{Q}_E$ and $S_{\frak{Q}_E}/\frak{Q}S_{\frak{Q}_E} =  S_E/\frak{Q}_E[1]$. Let $\Pi \in S$ be an element of order 1 at $\frak{Q}$, 
then  by \cite{Lang}[Chapter I,  Section 7,  Proposition 23] we have $S_{\frak{Q}_E} = (S_E)_{\frak{Q}_E}[1, \Pi] =  (S_E)_{\frak{Q}_E}[\Pi]$.  One can see that $\Pi^{|E| } \in \frak{Q}_E$  from \cite{Marcus}[Chapter 4, Theorem 28], since $|E| = [L : L_E]$, where $\frak{Q}_E = \frak{Q}^e$. 
This proves our Lemma.
\end{proof}

Our construction of  idempotents in $S\otimes_RS_\frak{Q}$ uses the following Lemma:

\begin{lemma}  Let $S, R, \frak{Q}, \frak{P}, E = E(\frak{Q}|\frak{P})$, $S_E$,  $\frak{Q}_E$ and $G = \{\sigma_{g_1}, \sigma_{g_2}, \cdots , \sigma_{g_n}\}$ be as defined in Section 2. Let $\Pi$ be an element of order $1$ in $\frak{Q}$, with $\Pi^{|E| } \in \frak{Q}_E$,  such that 
$$S \subseteq  (S_E)_{\frak{Q}_E}[\Pi].$$
  If $\sigma_{g_i} \not\in E$, then there exists $s_i \in S_E$ such that $\sigma_{g_i}(s_i) - s_i \not\in \frak{Q}$. 
\end{lemma}

\begin{proof} Let us assume that $\sigma_{g_i}(\alpha) - \alpha \in \frak{Q}$ for all $\alpha$ is $S_E$ and then prove the Lemma by contradiction.  By definition of $E$, if $\sigma_{g_i} \not\in E$, there exists $s \in S$ such that $\sigma_{g_i}(s) - s \not\in \frak{Q}$.  By the previous Lemma, 
$$s = \frac{a_0}{b_0} +  \frac{a_1}{b_1}\Pi + \cdots +  \frac{a_k}{b_k}\Pi^k = \frac{a_0' + a_1'\Pi + \cdots   + a_k'\Pi^k}{b_0b_1\cdots b_k}$$
where $a_i , a_i' \in S_E$, $b_i \in S_E \backslash \frak{Q}_E$, $i = 1, 2, \cdots , k$ and  $k = |E| -1$. 
Let $b = b_0b_1\cdots b_k$, and $s' = a_0' + a_1'\Pi + \cdots  + a_k'\Pi^k$. \\
\underline{Claim:}  $\sigma_{g_i}(s') - s' \not\in \frak{Q}$. \\
\underline{Proof of Claim} We know that 
$$\sigma_{g_i}(s) - s = \sigma_{g_i}(\frac{s'}{b}) - \frac{s'}{b}  = \frac{b\sigma_{g_i}(s') - \sigma_{g_i}(b)s'}{\sigma_{g_i}(b)b}  \not\in \frak{Q}.$$
If $\sigma_{g_i}(b) \in \frak{Q}$, then $\sigma_{g_i}(b) - b \not\in \frak{Q}$, because  \ $b \in S\backslash \frak{Q}$.  Thus from our initial assumption, we have $\sigma_{g_i}(b) \not\in \frak{Q}$.
Since $\frak{Q}$ is a prime ideal, this gives that 
$\sigma_{g_i}(b)b \not\in \frak{Q}$. Hence $\sigma_{g_i}(b)b(\sigma_{g_i}(s) - s) \not\in \frak{Q}$ and therefore $\sigma_{g_i}(b)b(\sigma_{g_i}(s) - s) = b\sigma_{g_i}(s') - \sigma_{g_i}(b)s' = b\sigma_{g_i}(s') - bs' + bs' - \sigma_{g_i}(b)s' 
= b(\sigma_{g_i}(s') - s') + (b - \sigma_{g_i}(b))s'   \not\in \frak{Q}$.  Since $b \in S_E$, by assumption, $\sigma_{g_i}(b) - b \in \frak{Q}$. Hence 
$b(\sigma_{g_i}(s') - s') + (b - \sigma_{g_i}(b))s'  \equiv b(\sigma_{g_i}(s') - s') \mod \frak{Q}$. Therefore $b(\sigma_{g_i}(s') - s') \not\in \frak{Q}$ and since 
$b \not\in \frak{Q}$, we have $(\sigma_{g_i}(s') - s') \not\in \frak{Q}$, \underline{thus proving the claim}. \\
Now if $(\sigma_{g_i}(s') - s') \not\in \frak{Q}$, we must have that $\sigma_{g_i}(a'_j\Pi^j) - a_j\Pi^j \not\in \frak{Q}$ for some $j$, $0 \leq j \leq k$. By our assumption at the beginning of the proof,
$\sigma_{g_i}(a'_j) \equiv a'_j \mod \frak{Q}$, because $a'_j \in S_E$. Therefore $\sigma_{g_i}(a'_j\Pi^j)  - a_j\Pi^j \equiv a'_j(\sigma_{g_i}(\Pi^j) - \Pi^j)
\mod \frak{Q}$.  Hence $\sigma_{g_i}(\Pi^j) - \Pi^j \not\in \frak{Q}$.  However $\Pi^j \in \frak{Q}$, giving that $\sigma_{g_i}(\Pi^j) \not\in \frak{Q}$. Now $(\Pi^j)^{|E|} \in \frak{Q}_E$
and $\sigma_{g_i}((\Pi^j)^{|E|}) = (\sigma_{g_i}(\Pi^j))^{|E|}$. Since $\sigma_{g_i}(\Pi^j) \not\in \frak{Q}$, we must have that  $(\sigma_{g_i}(\Pi^j))^{|E|} \not\in \frak{Q}$.  Hence $\sigma_{g_i}(\Pi^{j|E|}) - \Pi^{j|E|} \not\in \frak{Q}$  contradicting  our initial assumption and  thus proving the lemma.
\end{proof}

\begin{lemma} Let $S, R, \frak{Q}, \frak{P}, G = \{\sigma_{g_1}, \sigma_{g_2}, \dots , \sigma_{g_n}\},  E = E( \frak{Q}|\frak{P}), \phi$  and  $\phi_{g_k}$ be as defined in Section 2. Let  $\{\sigma_{e_1}, \sigma_{e_2}, \cdots , \sigma_{e_{|E|}}\}$ be the elements of the group $E$. 
We can find  an $x_1  \in S\otimes_RS_\frak{Q}$ such that 
$$\phi_{g_k}(x_1)  =  \Big\{\  \begin{matrix}
& 1 & \hbox{if}  & \sigma_{g_k} \in E; & \mbox{that is} \ {g_k} \in \{e_1, e_2, \cdots e_{|E|} \} \\
& 0 &   &\mbox{otherwise} &
\end{matrix}
$$
\end{lemma}

\begin{proof} If $\sigma_{g_l} \not\in E$, by the previous lemma, we have $s_l \in S_E$ such that $s_l - \sigma_{g_l}(s_l) \in S_\frak{Q}^*$. Now for each such $l$, let $y_l = s_l\otimes 1 - 1 \otimes \sigma_{g_l}(s_l) \in S\otimes_RS_\frak{Q}$.  So for $\sigma_{g_l} \not\in E$, we have $\phi_{e_i}(y_l)
= \sigma_{e_i}(s_l) - \sigma_{g_l}(s_l) = s_l - \sigma_{g_l}(s_l)$, for $i = 1, 2, \cdots , |E|$, since $s_l \in S_E$ by the assumption above.  Now let $y_l' = y_l(1\otimes (s_l - \sigma_{g_l}(s_l))^{-1}) \in S\otimes_RS_\frak{Q}$.  Then $y_l'$ has the property:
$$\phi_{g_k}(y_l')  =  \Big\{\  \begin{matrix}
& 1 & \hbox{if}  & \sigma_{g_k} \in E; & \mbox{that is} \ {g_k} \in \{e_1, e_2, \cdots e_{|E|} \} \\
& 0 &   \hbox{if}  & \sigma_{g_k} = \sigma_{g_l}   &
\end{matrix}
$$
Letting $x_1 = \prod_{\sigma_{g_l} \not\in E}y_l'$, gives us the required element of $S\otimes_RS_\frak{Q}$. 
\end{proof}

Each $\sigma_{g_j} \in G$, gives  an isomorphism of rings  $\sigma_{g_j}\otimes 1:  S\otimes_RS_\frak{Q} \to S\otimes_RS_\frak{Q}$. We use these maps to construct idempotents in $S\otimes_RS_\frak{Q}$ 
corresponding to  the right cosets of $E$ in $G$.

\begin{lemma} \label{componentiso} Let $S, R, \frak{Q}, \frak{P}, G = \{\sigma_{g_1}, \sigma_{g_2}, \dots , \sigma_{g_n}\},  E = E( \frak{Q}|\frak{P}), \phi$  and  $\phi_{g_k}$ be as defined in Section 2.  Let $E\sigma_{g_j} =  \{\sigma_{j_1}, \sigma_{j_2}, \cdots  , \sigma_{j_{|E|}}\} $ be a  right coset of $E$ in G. Using the chosen  coset representative $\sigma_{g_j}$, and the idempotent $x_1 \in S\otimes_RS_\frak{Q}$
constructed in the previous lemma, let $x_j = (\sigma_{g_j}^{-1}\otimes 1)(x_1)$.  Then
$$\phi_{g_k}(x_j)  =  \Big\{\  \begin{matrix}
& 1 & \hbox{if}  & {g_k} \in \{j_1, j_2, \cdots , j_{|E|} \}   \\
& 0 &   &\mbox{otherwise} 
\end{matrix}
$$
The  ring  isomorphism, $\sigma_{g_j}^{-1}\otimes 1:  S\otimes_RS_\frak{Q} \to S\otimes_RS_\frak{Q}$, restricts to a ring isomrphism  between the components of the ring $S\otimes_R S_\frak{Q}$ ;  $\sigma_{g_j}^{-1} \otimes 1 : (S\otimes_R S_\frak{Q})x_1 \to 
(S\otimes_R S_\frak{Q})x_j$.
\end{lemma}

\begin{proof} Let $\sigma_{g_j}$ be the chosen coset representative of $E\sigma_{g_j}$. Let $\sigma_{g_l} \in G\backslash E$. Consider the element 
$y_l$ corresponding to $\sigma_{g_l}$ constructed in the proof of the previous lemma. Let $\bar{y_l} = (\sigma_{g_j}^{-1} \otimes 1)(y_l) = 
\sigma_{g_j}^{-1}(s_l) \otimes 1 - 1 \otimes \sigma_{g_l}(s_l)$.  For any $\sigma_{g_t} \in G$, we have $\phi_{g_t}(\bar{y_l}) = \sigma_{g_t}\sigma_{g_j}^{-1} (s_l) - \sigma_{g_l}(s_l)$. Using the fact that $s_l \in S_E$, it is easy to see that 
$$\phi_{g_t}(\bar{y_l})  =  \Big\{\  \begin{matrix}
& s_l - \sigma_{g_l}(s_l)  & \hbox{if}  & \sigma_{g_t} \in E\sigma_{g_j} \\
& 0 &   \hbox{if}  & \sigma_{g_t} = \sigma_{g_l}\sigma_{g_j}  &
\end{matrix}
$$
Now  in the proof of the previous lemma, we set  $y_l' = y_l(1\otimes (s_l - \sigma_{g_l}(s_l))^{-1})$, and we set 
$x_1 = \prod_{\sigma_{g_l} \not\in E}y_l'$.  Note that both homomorphisms, $\phi_{g_k} : S\otimes_RS_\frak{Q} \to S_\frak{Q}$ and $\sigma_{g_j}^{-1} \otimes 1:
 S\otimes_RS_\frak{Q}  \to  S\otimes_RS_\frak{Q} $ commute with the right action of $S_\frak{Q}$, given by $xs = x(1\otimes s)$ for all $x \in S\otimes_R S_\frak{Q}$ and $s \in S_\frak{Q}$.  With $\sigma_{g_j}$ as above,  we also  note that as $\sigma_{g_l}$ runs through $G\backslash E$, $\sigma_{g_l}\sigma_{g_j}$ runs through 
 $G\backslash E\sigma_{g_j}$. Hence letting $x_j = (\sigma_{g_j}^{-1} \otimes 1)(x_1) = (\sigma_{g_j}^{-1} \otimes 1)\prod_{\sigma_{g_l} \not\in E}y_l( 1\otimes (s_l - \sigma_{g_l}(s_l))^{-1}) =  \prod_{\sigma_{g_l} \not\in E}\bar{y_l}(1\otimes (s_l - \sigma_{g_l}(s_l))^{-1})$
 we see that 
$$\phi_{g_t}( x_j)  =  \Big\{\  \begin{matrix}
& 1 & \hbox{if}  & \sigma_{g_t} \in E\sigma_{g_j} \\
& 0 &  & \mbox{otherwise} &
\end{matrix}
$$
The isomorphism between ring components  is obvious. 
\end{proof}

Note That $x_j$ depends only on the coset $E\sigma_{g_j}$ and is independent  of the coset representative, since it is completely determined by 
its image $\phi(x_j)$. 
We now see that $S\otimes_R S_\frak{Q}$ is isomorphic to a product of  $m = |G|/|E|$ copies of the subring $(S\otimes_R S_\frak{Q})x_1$.  In the next lemma we  investigate the nature of this subring.  

\begin{lemma} \label{isomap}  Let $S, R, \frak{Q}, G =  \{\sigma_{g_1}, \sigma_{g_2}, \cdots , \sigma_{g_n}\}, E, S_E, S_\frak{Q}, \phi_{g_k},$ be as defined in Section 2.   Let  $\{\sigma_{e_1}, \sigma_{e_2}, \cdots , \sigma_{e_{|E|}}\}$ be the elements of the group $E$ and let 
$\{E\sigma_{x_1}, E\sigma_{x_2}, \cdots , E\sigma_{x_m}\}$ be the right cosets of $E$ in $G$, where 
$m = |G|/|E|$ and $\sigma_{x_1}$ is the identity element of $G$. 
 Let  $x_i$ be the  element of  $S\otimes_RS_\frak{Q}$ such that 
$$\phi_{g_k}(x_i)  =  \Big\{\  \begin{matrix}
& 1 & \hbox{if}  & \sigma_{g_k} \in E\sigma_{x_i} &  \\
& 0 &   &\mbox{otherwise} &
\end{matrix}.
$$
 Then 
$$(S\otimes_R S_\frak{Q})x_1 \cong  S\otimes_{S_E}S_\frak{Q}.$$
The map $\gamma : S\otimes_{S_E}S_\frak{Q} \to (S\otimes_R S_\frak{Q})x_1$ given by 
$\gamma(s_1\otimes s_2) = (s_1\otimes s_2)x_1$, where $s_1 \in S, s_2 \in S_{\frak{Q}}$, is an isomorphism of rings.
Regarding  $(S\otimes_RS_{\frak{Q}})x_1$ as a subset of $S\otimes_RS_{\frak{Q}}$, the  inverse of $\gamma$ is  $\psi : (S\otimes_RS_{\frak{Q}})x_1 \to S\otimes_{S_E}S_{\frak{Q}}$ given by the restriction of the map $\psi : S\otimes_RS_{\frak{Q}} \to
S\otimes_{S_E}S_{\frak{Q}}$, defined by 
 $\psi(s_1'\otimes s_2') = s_1' \otimes s_2'$, for $s_1' \in S$ and $s_2' \in S_{\frak{Q}}$. 
\end{lemma}

\begin{proof} As in the previous notation,  let $L_{g_k}$, a copy of $L$,  denote the image of $\phi_{g_k}$ for $\sigma_{g_k} \in G$.  Recall that  the Galois group of $L$ over $L_E$ is $E$, and  we have an isomorphism 
$\phi' : L\otimes_{L_E}L \to L_{e_1} \oplus L_{e_2} \oplus \cdots \oplus L_{e_{|E|}}$ as defined in Section 2. Let $\psi : L\otimes_KL \to L\otimes_{L_E}L$ be the surjective  ring homomorphism  sending $l_1\otimes l_2$ to $l_1\otimes l_2$.  The diagram
below is obviously commutative:

\[\xymatrix{
{L\otimes_KL}\ar[drr]^{\phi_{e_1} \oplus  \phi_{e_2} \oplus \cdots \oplus \phi_{e_{|E|}} }\ar[d]^{\psi}
&&\\
{L\otimes_{L_E} L}\ar[rr]^{\phi'}&& L_{e_1} \oplus L_{e_2} \oplus \cdots \oplus L_{e_{|E|}}
}\]
This diagram  restricts to a commutative diagram:
\[\xymatrix{
{S\otimes_RS_\frak{Q}}\ar[drr]^{\phi_{e_1} \oplus  \phi_{e_2} \oplus \cdots \oplus \phi_{e_{|E|}} }\ar[d]^{\psi}
&&\\
{S\otimes_{S_E} {S_\frak{Q}}}\ar[rr]^{\phi'}&&{S_\frak{Q}}_{e_1} \oplus {S_\frak{Q}}_{e_2} \oplus \cdots \oplus {S_\frak{Q}}_{e_{|E|}}
}\]
Clearly the images of the maps $\phi'$ and ${\phi_{e_1} \oplus  \phi_{e_2} \oplus \cdots \oplus \phi_{e_{|E|}} }$ in the second diagram  are the same and  the map $\phi'$ is one to one. I claim that the kernel of ${\phi_{e_1} \oplus  \phi_{e_2} \oplus \cdots \oplus \phi_{e_{|E|}} }$ is 
$\oplus_{i \not= 1}(S\otimes_RS_\frak{Q})x_i$.
It is not difficult to see that  $\oplus_{i \not= 1}(S\otimes_RS_\frak{Q})x_i \subseteq \ker {\phi_{e_1} \oplus  \phi_{e_2} \oplus \cdots \oplus \phi_{e_{|E|}} }$.  On the other hand if $x \in S\otimes_{R}S_{\frak{Q}}$ is in $\ker {\phi_{e_1} \oplus  \phi_{e_2} \oplus \cdots \oplus \phi_{e_{|E|}} }$, then $\phi_{g_k}(xx_1) =0$ for $k = 1, 2, \dots , n$ and hence $ xx_1 = 0$. Hence $x = x(\sum_ix_i) = x(\sum_{i \not= 1}x_i) \in \oplus_{i \not= 1}(S\otimes_RS_\frak{Q})x_i $.  This proves the claim.

Hence in the following diagram, ${\phi_{e_1} \oplus  \phi_{e_2} \oplus \cdots \oplus \phi_{e_{|E|}} }$ and $\phi'$ are isomorphisms, giving us that  $\psi$ is an isomorphism. 
\[\xymatrix{
({S\otimes_RS_\frak{Q}})x_1\ar[drr]^{\phi_{e_1} \oplus  \phi_{e_2} \oplus \cdots \oplus \phi_{e_{|E|}} }\ar[d]^{\psi}
&&\\
{S\otimes_{S_E} {S_\frak{Q}}}\ar[rr]^{\phi'}&&\phi'(S\otimes_{S_E}S_\frak{Q})
}\]
Now if $s \in S_E$, we see that $\phi((s\otimes 1)x_1) = \phi(1 \otimes s)x_1$ and hence 
$(s \otimes 1)x_1 = (1\otimes s)x_1$. Thus we have a homomorphism $\gamma : S\otimes_{S_E} {S_\frak{Q}} \to ({S\otimes_RS_\frak{Q}})x_1$ such that $\gamma(s_1 \otimes s_2) = (s_1 \otimes s_2)x_1$, for $s_1  \in S$ and $s_2 \in S_{\frak{Q}}$. It is obvious that $\gamma = \psi^{-1}$ and that $\gamma$ is an isomorphism.

\end{proof}
Now we can use Lemma \ref{Pie}  to determine the structure of $S\otimes_{S_E}S_\frak{Q}$.  

\begin{lemma} \label{structure}  Let $S, R, \frak{Q}, \frak{P},  E = E(\frak{Q}|\frak{P}), S_E, $ and $\frak{Q}_E$ be as defined in section 2.  There exists an element   $\Pi$ of order 1 in  $\frak{Q}$ such that $S \subseteq (S_E)_{\frak{Q}_E}[\Pi]$ 
and 
$$S\otimes_{S_E}S_\frak{Q} = S_E[\Pi]\otimes_{S_E}S_\frak{Q},$$
where this identification is as described in Section 2. 
\end{lemma}

\begin{proof}
By Lemma \ref{Pie}  there exists an element   $\Pi$ of order 1 in  $\frak{Q}$ such that $S \subseteq (S_E)_{\frak{Q}_E}[\Pi]$.  By the discussion at the end of section 2, we can identify  $S_E[\Pi]\otimes_{S_E}S_\frak{Q}$ with an isomorphic subring of 
$S\otimes_{S_E}S_\frak{Q}$. 
 If $s \in S$, we have $s = s_1/t$, where $s_1 \in {S_E}[\Pi]$ and $t \in S_E\backslash \frak{Q}_E$. Now $\frac{s_1}{t}\otimes 1 = s_1\otimes {\frac{1}{t}}$ in $S\otimes_{S_E}S_\frak{Q}$, since $\frac{s_1}{t}\otimes 1 = (\frac{s_1}{t}\otimes 1)(1\otimes t)(1 \otimes \frac{1}{t}) =  (\frac{s_1}{t}\otimes 1)(t\otimes 1)(1 \otimes \frac{1}{t})
= s_1\otimes \frac{1}{t}$ Hence $S \otimes 1 \subseteq  S_E[\Pi]\otimes_{S_E}S_\frak{Q}$ and the result follows. 
\end{proof}

\begin{corollary} \label{corstar} Let $S, R, \frak{Q}, \frak{P}, G = \{\sigma_{g_1}, \sigma_{g_2}, \dots , \sigma_{g_n}\},  E = E( \frak{Q}|\frak{P}), S_E$  and  $S_{\frak{Q}}$ be as defined in Section 2.
Let $\{E\sigma_{x_1}, E\sigma_{x_2}, 
\cdots , E\sigma_{x_m}\}$  be the right cosets of $E$ in $G$, with  coset representatives
$\{\sigma_{x_1}, \sigma_{x_2}, \cdots , \sigma_{x_m}\}$.   Then there exist unique  orthogonal idempotents 
$\{x_1, x_2, \dots , x_m\}$ in $S\otimes_RS_{\frak{Q}}$, where $m = n/|E|$ such that 
$$\phi_{g_k}(x_i)  =  \Big\{\  \begin{matrix}
& 1 & \hbox{if}  & \sigma_{g_k} \in E\sigma_{x_i} \\
& 0 &   &\mbox{otherwise} &
\end{matrix}.
$$
and 
$$S\otimes_RS_{\frak{Q}} = \oplus_{i = 1}^m(S\otimes_RS_{\frak{Q}})x_i.$$
Furthermore we have ring isomorphisms  $(S\otimes_RS_{\frak{Q}})x_1  \cong S\otimes_{S_E} S_{\frak{Q}} \cong S_E[\Pi]\otimes_{S_E}S_{\frak{Q}}$ and 
$\sigma_{x_j}  \otimes 1 : (S\otimes_RS_{\frak{Q}})x_j  \to (S\otimes_RS_{\frak{Q}})x_1$ for $1 \leq j \leq m$, where $\Pi$ is an element of order 1 in $\frak{Q}$. 

Note that the uniqueness of the orthogonal idempotents follows from the fact that $\phi: L\otimes_K L \to L_{g_1} \oplus L_{g_2} \oplus \dots \oplus L_{g_n}$ is an isomorphism. 
\end{corollary}

We will also need to use the following lemma about  the structure of $S\otimes_{S_E}S_\frak{Q}\otimes_{S_{\frak{Q}}} F_{\frak{Q}}$, where $F_{\frak{Q}}$ is the quotient field $S/\frak{Q}$, later. 

\begin{lemma} \label{FQ}  Let $S, R, \frak{Q}, \frak{P},  E = E(\frak{Q}|\frak{P}) = \{\sigma_{e_1}, \sigma_{e_2}, \cdots
, \sigma_{e_{|E|}}\}, S_E, $ and $\frak{Q}_E$ be as defined in section 2.  Let $F_{\frak{Q}}$ denote the quotient field $S_{\frak{Q}}/\frak{Q}S_{\frak{Q}} \cong S/\frak{Q}$. Let $\Pi$ be an element of order 1 in 
$S$ such that 
$S\otimes_{S_E}S_\frak{Q} = S_E[\Pi]\otimes_{S_E}S_\frak{Q}$ and let $A_{e_i}(\Pi) = \Pi\otimes 1 -
1 \otimes \sigma_{e_i}(\Pi) \in S\otimes_{S_E}S_\frak{Q}, 1\leq i \leq |E|$. Let 
$\Psi': S\otimes_{S_E}S_{\frak{Q}} \to S\otimes_{S_E}S_{\frak{Q}} \otimes_{S_{\frak{Q}}} F_{\frak{Q}}$
be the homomorphism such that $\Psi'(s_1 \otimes s_2) = s_1 \otimes s_2 \otimes 1$. Then
$$\Psi'(A_{e_i}(\Pi)) = \Psi'(A_{e_j}(\Pi)), 1 \leq i, j, \leq |E|.$$
Let $\alpha = \Psi'(A_{e_1}(\Pi))$, then
$$ S\otimes_{S_E}S_{\frak{Q}} \otimes_{S_{\frak{Q}}} F_{\frak{Q}} =
F_{\frak{Q}}[\alpha]$$
is a local ring with maximal ideal $(\alpha)$. Let  $h: F_{\frak{Q}}[x] \to F_{\frak{Q}}[\alpha]$ be the homomorphism with  $h(x) = \alpha$, then $\ker h$ is the ideal generated by the polynomial $x^{|E|} $.  
\end{lemma}

\begin{proof} By the definition of $E$, we have $\sigma_{e_i}(\Pi) - \Pi \in \frak{Q}$ for $1\leq i \leq |E|$. 
Hence $\Psi'((\Pi\otimes1 - 1\otimes\Pi) - (\Pi \otimes 1 - 1 \otimes \sigma_{e_i}(\Pi))) = 0$ for $1 \leq i \leq |E|$. Thus $\Psi'(\Pi\otimes1 - 1\otimes\Pi) = \Psi'(A_{e_i}(\Pi) = \alpha$ for $1\leq i \leq |E|$. 
Now by applying Lemma \ref{basisP}  and Lemma \ref{alphan}, with $R$ replaced by $S_E$ and $G$ replaced by $E$, we see that $S\otimes_{S_E}S_{\frak{Q}} \otimes_{S_{\frak{Q}}} F_{\frak{Q}} =
F_{\frak{Q}}[\alpha]$ and $\alpha^{|E|} = 0$ in $S\otimes_{S_E}S_{\frak{Q}} \otimes_{S_{\frak{Q}}} F_{\frak{Q}} $.  Since $\alpha$ is nilpotent, we have that $S\otimes_{S_E}S_{\frak{Q}} \otimes_{S_{\frak{Q}}} F_{\frak{Q}} $ is local with maximal ideal $(\alpha)$.  Now by 
Lemma \ref{basisP}, we have that $1, \alpha, \cdots , \alpha^{|E| - 1}$ is a basis for 
$S\otimes_{S_E}S_{\frak{Q}} \otimes_{S_{\frak{Q}}} F_{\frak{Q}} $ over $F_{\frak{Q}}$, hence
the map $h$ has kernel $(x^{|E|} )$. 
\end{proof}

\begin{lemma} \label{unram}
 Let $S, R, \frak{Q}, \frak{P}, S_{\frak{Q}},  \phi, \phi_{g_k},  G = \{\sigma_{g_1}, \sigma_{g_2}, \dots , \sigma_{g_n}\},  E = E( \frak{Q}|\frak{P})$ and $S_E$    be as defined in Section 2.  If $|E| = 1$ then there exist orthogonal idempotents, 
 $\{x_1, x_2, \cdots , x_n\}$,  in $S\otimes_RS$, such that 
$$\phi_{g_k}(x_i)  =  \Big\{\  \begin{matrix}
& 1 & \hbox{if}  & k = i \\
& 0 &   &\mbox{otherwise} &
\end{matrix}.
$$
Furthermore  $\phi: S\otimes_RS_\frak{Q} \to S_{\frak{Q}, {g_1}} \oplus  S_{\frak{Q}, {g_2}} \oplus \cdots \oplus S_{\frak{Q}, {g_n}}$ is an isomorphism.  
\end{lemma}

\begin{proof} The existence of the idempotents  $\{x_1, x_2, \cdots , x_n\}$ follows from Corollary  \ref{corstar}. 
 Since $\phi : L\otimes_KL \to L_{g_1} \oplus L_{g_2} \oplus \cdots \oplus L_{g_n}$ is an isomorphism, by Lemma \ref{phi}, we need only show that $\phi: S\otimes_RS_\frak{Q} \to S_{\frak{Q}, {g_1}} \oplus  S_{\frak{Q}, {g_2}} \oplus \cdots \oplus S_{\frak{Q}, {g_n}}$ is onto. Let $(s_1, s_2, \cdots , s_n)$ be 
an arbitrary element of $S_{\frak{Q}, {g_1}} \oplus  S_{\frak{Q}, {g_2}} \oplus \cdots \oplus S_{\frak{Q}, {g_n}}$, with $s_i \in S_{\frak{Q}}$ for each $i$.  
Then $\phi((1\otimes s_1)x_1 + (1 \otimes s_2)x_2 + \cdots  + (1\otimes s_n)x_n) = (s_1, s_2, \cdots ,
s_n)$ and hence   the map is onto as required. 
\end{proof}

\section{Stratified Exact Categories}

In this section we sketch the general definition of stratified exact categories from \cite{Dyer}. Proofs of 
many fundamental results about these categories are supplied in \cite{Dyer}.  We will present the definition in full within the context of the framework considered by Dyer in \cite{Dyer2}, since this is adequate for our needs. We will quote the results that we will use  from \cite{Dyer}, in this context.

We will start with the general definition from \cite{Dyer}. 
Consider a set of data $E = (k, \mathcal{B}, \{N_x\}_{x \in \Omega})$ where $k$ is a commutative, Noetherian 
ring, ${\mathcal{B}}$ is an abelian $k$-category, and the $N_x$ are objets of $\mathcal{B}$ indexed by a finite poset $\Omega$. We assume this set of data has the following properties:

(a) $Hom_{\mathcal{B}}(N_x, N_y)$ is zero unless $x \leq y$

(b) $Ext_{\mathcal{B}}^1(N_x, N_y)$ is zero unless $x < y$ or $y \leq x$. (Note that in the case of a total ordering on $\Omega$, this is not a restriction.)

(c) For $x \in \Omega$, and any $N, N'$ in $Add\  N_x$, any surjection $N \to N' \to 0$ in $\mathcal{B}$ splits, 
where $Add \  N_x$ denotes the class of objects in $\mathcal{B}$, which are isomorphic to a direct summand of a finite direct sum of copies of $N_x$. 

(d) For $x, y \in \Omega$, $Hom_{\mathcal{B}}(N_x, N_y)$ is a finitely generated $k$-module, and if $x < y$, then $Ext_{\mathcal{B}}^1(N_x, N_y)$ is a finitely generated $k$- module. 

A coideal, $\Gamma$,  of $\Omega$ is a subset with the property that if $y \in \Gamma$, then $x \in \Gamma$ for each  $x \in \Omega$ with the property that $x \geq y$.
We say an object $N$ of $\mathcal{B}$ has an $N$-filtration if it has subobjects $N(\Gamma)$, for $\Gamma$ a coideal of $\Omega$ such that:

(i) $N(\emptyset) = 0, \ N(\Omega) = N$

(ii) if $\Gamma \subseteq \Gamma'$ are coideals, then $N(\Gamma)$ is a subobject of $N(\Gamma')$. 

(iii) if $x$ is a minimal element of a coideal $\Gamma$ of $\Omega$, then $N' = N(\Gamma)/N(\Gamma \backslash \{x\})$ is in Add $N_x$.

The following lemma, when applied to the identity morphism $N \to N$, ensures that such a filtration is unique, Dyer  \cite{Dyer} [Lemma 1.4]:

\begin{lemma} \label{maps} Let   $N$ and $N'$ be  objects of $\mathcal{B}$ , with  N-filtrations as given above.   Any morphism 
$N \to N'$ in $\mathcal{B}$ maps $N(\Gamma)$ into  $N'(\Gamma)$ where $\Gamma$ is a coideal of $\Omega$.
\end{lemma}

Let $\mathcal{C}$ be the full additive subcategory of objects of $\mathcal{B}$ consisting of objects having an $N$-filtration.  We say a sequence $0 \to N \to P \to Q \to 0$  of objects in $\mathcal{C}$ is sheaf exact if it is exact in $\mathcal{C}$ and  for each $x \in \Omega$, the sequence   $0 \to N(\geq x) \to P(\geq x) \to 
Q(\geq x) \to 0$ is exact, where $\geq x$ is  the obvious coideal.  Note in particular  that sheaf exact sequences are short exact in $\mathcal{B}$. 

An exact category $C$ in the sense of Quillen is an additive category $C$ equipped with a family of ``short exact sequences", $0 \to M \to N \to P \to 0$ satisfying certain axioms
\cite{Quillen}. The stratified category $\mathcal{C}$ defined above is exact in the sense of Quillen, see Dyer  \cite{Dyer} [Proposition 1.7], with the sheaf exact sequences as exact sequences.  Stratified categories thus inherit the concept of an exact functor and a projective object from exact categories.
A functor between exact categories is called exact if it is additive and preserves short exact sequences. An object $P$ of an exact category $C$ is called projective if $\mbox{Hom}(P,  - )$ is exact as a functor from $C$ to the category of Abelian groups.  An exact category is said to have sufficiently many projectives if for any object $M$ in $C$, there is a short exact sequence $0 \to N \to P \to M \to 0$ in $C$ with $P$ projective in $C$.

Stratified categories   can be constructed in a wide range of situations, see Dyer  \cite{Dyer2} and Brown \cite{Brown}. Since our attention is limited  to categories  related to number rings and total orderings on Galois groups, we will now restrict our attention to  the specific framework considered by Dyer in \cite{Dyer2}[Section 1.10]. 

Let $T$ be a commutative ring and let $U$ and $A$ be commutative $T$ algebras.  We let $A$ be a Noetherian domain (in our special case, $T, U$ and $A$ will be Dedekind domains).  Let $\Sigma  = 
\{\sigma_{x_i}\}_{x_i \in \Omega}$ be a family of pairwise distinct $T$-algebra homomorphisms
$$\sigma_{x_i}: U \to A,$$
with indices in a poset $\Omega$ with a total ordering $\{x_1 <  x_2 < \dots < x_n\}$. (Note that the additional condition $(*)$ given in section 1.10 of \cite{Dyer2} is irrelevant when the poset has a total ordering). We associate to $(U, T, A, \Omega, \Sigma)$ a collection of data,
$D = (k, \mathcal{B}, \{N_{x_i}\}_{{x_i} \in \Omega})$ as follows:
$k = A,  \mathcal{B}$ is the category of all $U\otimes_TA$ modules and  $N_{x_i} = A[\sigma_{x_i}]$, which is a free right A-module of rank one, with basis element $1_{x_i}$, where the left action by $U$ is given by $ua = a\sigma_{x_i}(u)$ for all $u \in U$ and $a \in N_{x_i} = A[\sigma_{x_i}]$.  (We will use the notation $A[\sigma_{x_i}]$ for these $U\otimes_TA$ modules from now on in this paper : this is compatible with the definition of $S[\sigma_{g_i}]$ in section 2.). 

\begin{definition}\label{Cat}
The category, $\mathcal{C}_{(U\otimes_TA, \Omega, \Sigma)}$,  for  a set of data, $D = \{A, U\otimes_TA-\hbox{mod}, \{A[\sigma_{x_i}]\}_{x_i \in \Omega}\})$, where, $A, T, U, \Sigma$ and $\Omega$ are as above,  is the full subcategory of $\mathcal{B}$, consisting of all $U\otimes_TA$ modules, $M$, 
which have a filtration
$$M = M^{x_0} \supseteq M^{x_1} \supseteq M^{x_2} \supseteq M^{x_3} \supseteq \dots \supseteq M^{x_n} = 0$$
by $U \otimes_TA$ modules, such that $M^{x_{i - 1}}/M^{x_{i}}$ is finitely generated and projective as a right 
$A$ module with $um = m\sigma_{x_i}(u)$ for $m \in M^{x_{i - 1}}/M^{x_{i}}$ for all $u \in U$. 
\end{definition}

By Lemma  \ref{maps} applied to the identity map, such a filtration is unique, so for a module $M$ in category $\mathcal{C}_{(U\otimes_TA, \Omega, \Sigma)}$, we let $M^{x_i}$ denote the corresponding module in such a filtration. By the same Lemma,  every short exact sequence, $0 \to M \to N \to P \to 0$  in $\mathcal{B}$ gives a sequence
$0 \to M^{x_i} \to N^{x_i} \to P^{x_i} \to 0$ for each $x_i \in \Omega$.  
Given a sequence $0 \to M \to N \to P \to 0$ of $U \otimes_TA$ modules in the category $\mathcal{C}_{(U\otimes_TA, \Omega, \Sigma)}$,  it is \underline{sheaf exact} if each restricted sequence
$$0 \to M^{x_i} \to N^{x_i} \to P^{x_i} \to 0$$
is exact as a sequence of $U\otimes_TA$ modules.  An exact sequence $0 \to M \to N \to P \to 0$ in the category $\mathcal{B}$ need not be sheaf exact (we give an example below), however by definition, every sheaf exact sequence in 
$\mathcal{C}_{(U\otimes_TA, \Omega, \Sigma)}$ is exact when viewed as a sequence in the category $\mathcal{B}$.

As discussed above, the category $\mathcal{C}_{(U\otimes_TA, \Omega, \Sigma)}$ with sheaf exact sequences as short exact ones is an exact category in the sense of Quillen \cite{Quillen}.
As for exact categories, an object $P$ of the category $\mathcal{C}_{(U\otimes_TA, \Omega, \Sigma)}$ is called projective if $Hom(P, -)$ is exact as a functor from $\mathcal{C}_{(U\otimes_TA, \Omega, \Sigma)}$ to the category of Abelian groups. From Dyer \cite{Dyer}[Theorem 1.18] we get another characterization of projectives in $\mathcal{C}_{(U\otimes_TA, \Omega, \Sigma)}$. 

\begin{lemma} \label{extproj}  An object $N$ in $\mathcal{C}_{(U\otimes_TA, \Omega, \Sigma)}$ is projective if and only if for all $x_i \in \Omega$, the short exact sequence 
$$0 \to N^{x_{i - 1}}/ N^{x_i} \to N/N^{x_i} \to N/N^{x_{i - 1}} \to 0, $$
gives rise to an exact sequence 
\vskip .1in
\noindent
$0 \to Hom(N/N^{x_{i - 1}}, A[\sigma_{x_i}]) \to Hom(N/N^{x_i}, A[\sigma_{x_i}]) \to Hom(N^{x_{i - 1}}/N^{x_i}, A[\sigma_{x_i}])  \to Ext(N/N^{x_{i - 1}}, A[\sigma_{x_i}]) \to 0 $, 
for each $i, 1 \leq i \leq n$. 
\end{lemma}

We also  have the following theorem concerning projectives according to \cite{Dyer}[Theorem 1.18]. 

\begin{theorem}  \label{progen}(1) There exists a projective generator P in $\mathcal{C}_{(U\otimes_TA, \Omega, \Sigma)}$, i.e. 
P is projective and for any M in $\mathcal{C}_{(U\otimes_TA, \Omega, \Sigma)}$, there is a sheaf exact 
sequence $0 \to N \to P^n \to M \to 0$ for some $n \geq 0$.
Let $\mathcal{A}_{U\otimes_T A}(P) = End_{U\otimes_T A}(P)^{op}$, then  $_{U\otimes_T A}P_{\mathcal{A}}$ is a bimodule.

(2) The functor $F = Hom_{U\otimes_T A}(P, \ ):\mathcal{C}_{(U\otimes_TA, \Omega, \Sigma)} \to \mathcal{A}$-mod is fully
faithful and 
\[0 \to M \to N \to Q \to 0\]
is sheaf exact in $\mathcal{C}_{(U\otimes_TA, \Omega, \Sigma)}$ if and only if 
\[0\to FM \to FN \to FQ \to 0\] is exact as a sequence of $\mathcal{A}_{U\otimes_T A}(P)$-modules.  

\end{theorem}

The algebra $\mathcal{A}_{U\otimes_T A}(P)$ is not unique, a different choice of  projective generator $P'$ will give a different 
algebra $\mathcal{A}_{U\otimes_T A}(P') = End_{U\otimes_TA}(P')$.  However \\
$Hom(P, P')\otimes_{\mathcal{A}_{U\otimes_T A}(P)}-$ gives a Morita equivalence from $\mathcal{A}_{U\otimes_T A}(P)-$mod to $\mathcal{A}_{U\otimes_T A}(P')-$ mod.  We can see this from Jacobson \cite{Jac2} [Theorem 3.20], Theorem \ref{progen} above  and the following Lemma:

\begin{lemma}\label{Newprogen} Let $P$ and $P'$ be projective generators in the category  $\mathcal{C}_{(U\otimes_TA, \Omega, \Sigma)}$, where $U, T, A, \Omega$ and $\Sigma$ are as above. Let $\bar{P} = Hom_{U\otimes_TA}(P, P')$. Let $\mathcal{A}_{U\otimes_T A}(P) = End_{U\otimes_T A}(P)$. Then 
$\bar{P}$ is a  projective generator  in the category $\mathcal{A}_{U\otimes_T A}(P)$-mod.
\end{lemma}
\begin{proof} Since both $P$ and $P'$ are projective generators, we have sheaf  exact sequences 
$$0 \to M_1 \to P^n \to P' \to 0 \ \hbox{and} \ \ \ 0 \to M_2 \to (P')^m \to P \to 0.$$
These sequences split to give us that $P$ is a direct summand of $(P')^m$ and $P'$ is a direct summand of $P^n$.  Now applying the exact functor $Hom_{U\otimes_T A}(P - )$, we see 
that $\bar{P}$ is a direct summand of $\mathcal{A}_{U\otimes_T A}(P)^n$ and hence $\bar{P}$ is projective in 
$\mathcal{A}_{U\otimes_T A}(P)$-mod. We also have that $\mathcal{A}_{U\otimes_T A}(P)$ is a direct summand of $\bar{P}^m$. Let 
$\bar{P}^m = \mathcal{A}_{U\otimes_T A}(P) \oplus L$. 
Now for any finitely generated $\mathcal{A}_{U\otimes_T A}(P)$ module, M,  we have an exact sequence:
$$0 \to K \to \mathcal{A}_{U\otimes_T A}(P)^l \to M \to 0,$$
for some $l$. Then we also  have an exact sequence 
$$0 \to K \oplus L^l  \to \mathcal{A}_{U\otimes_T A}(P)^l \oplus L^l  \to M \to 0,$$
and hence we have the exact sequence of $\mathcal{A}_{U\otimes_T A}(P)$ mosules:
$$0 \to K \to \bar{P}^{ml} \to M \to 0.$$
Hence $\bar{P}$ is a projective generator for $\mathcal{A}_{U\otimes_T A}(P)$ modules. 
\end{proof}
Let $\frak{M}$ be a prime ideal of $A$,  and let $A_{\frak{M}}$ be the localization of $A$ at $\frak{M}$. We can regard the elements of $\Sigma$ as homomorphisms from $U$ to $A_{\frak{M}}$. 
 Localization of modules, $M \to M \otimes_AA_{\frak{M}} $ is an exact functor from  $U\otimes_TA$-mod  to 
$U\otimes_TA_{\frak{M}}$-mod  which restricts to a functor of exact categories  $\mathcal{C}_{(U\otimes_TA, \Omega, \Sigma)}$ to 
 $\mathcal{C}_{(U\otimes_TA_{\frak{M}}, \Omega, \Sigma)}$. Clearly it takes a projective generator for $\mathcal{C}_{(U\otimes_TA, \Omega, \Sigma)}$ to a projective generator for $\mathcal{C}_{(U\otimes_TA_{\frak{M}}, \Omega, \Sigma)}$. Since $Hom_{U\otimes_TA}(M, N)_{\frak{M}} = Hom_{U\otimes_TA_{\frak{M}}}(M_{\frak{M}}, N_{\frak{M}})$, we have 
 $(\mathcal{A}_{U\otimes_T A}(P))_{\frak{M}} = (End_{U\otimes_TA}(P)^{op})_{\frak{M}}  = End_{U\otimes_TA_{\frak{M}}}(P_{\frak{M}})^{op}$ and $\mathcal{C}_{(U\otimes_TA_{\frak{M}}, \Omega, \Sigma)}$ is isomorphic to a subcategory of 
 $(\mathcal{A}_{U\otimes_T A}(P))_{\frak{M}}$ modules.

 \section{The Category $\mathcal{C}_{(S\otimes_RS, \Omega, G)}$}
 
 In this section we introduce the data from which we construct  a stratified exact category for number rings.  We  look at some examples of modules with filtrations and proceed to demonstrate some of the subtleties associated with the definition of 
 sheaf  exact sequences. 
 
 Let $S, R, L, K, $  and $G = \{\sigma_{g_1}, \sigma_{g_2}, \dots  , \sigma_{g_n}\}$ be as defined in Section 2. Let us assume  initially that the indices of $G$; $\{{g_1}, {g_2}, \dots ,{g_n}\}$, form a poset $\Omega$ equipped with the  ordering $g_1 <  g_2 < \dots < g_n$.  Using the notation developed in the previous section, 
 we let $\mathcal{C}_{(S\otimes_RS,  \Omega, G)}$ be the stratified exact category corresponding to 
 the data $D = \{S, S\otimes_RS-\mbox{mod}, S[\sigma_{g_i}]_{i \in \Omega}\}$.  
 
   We have that  $\mathcal{C}_{(S\otimes_RS,  \Omega, G)}$ is  the full subcategory
 of $S\otimes_RS$ modules, $M$,  with filtration
 $$M =  M^{g_0} \supseteq M^{g_1} \supseteq  M^{g_2} \supseteq \dots \supseteq M^{g_n} = 0,$$
 by $S\otimes_RS$ modules  such that $M^{g_{i - 1}}/M^{g_i}$ is finitely generated and projective as a right $S$-module  with $um = m\sigma_{g_i}(u)$ for 
 $m \in M^{g_{i - 1}}/M^{g_i}$ for all $u \in S$.  We let $M^{(g_i)}$ denote the quotient $M^{g_{i - 1}}/M^{g_i}$.  
 Since this filtration is unique,  we can  let $M^{g_i}$ denote the corresponding module in the filtration of $M$ without ambiguity. $M^{g_i}$ is 
in the category $\mathcal{C}_{(S\otimes_RS,  \Omega, G)}$. 

Given a sequence $0 \to M \to N \to P \to 0$ of  $S\otimes_RS$ modules in $\mathcal{C}_{(S\otimes_RS, \Omega, G)}$,
 it is sheaf exact if for each $i, 1\leq i \leq n$, the sequence
\[0\to M^{g_i} \to N^{g_i} \to P^{g_i} \to 0\]
is exact as a sequence of $S\otimes_RS$ modules. 
An exact sequence $0 \to M \to N \to P \to 0$ in the category of $S\otimes_RS$ modules 
need not be sheaf exact, however every sheaf exact sequence in $\mathcal{C}_{(S\otimes_RS, \Omega, G)}$ is exact when viewed 
as a sequence in the category of $S\otimes_RS$ modules.  Below we will give an example of a sequence which is exact in the category of $S\otimes_RS$ modules, but is not sheaf exact in $\mathcal{C}_{(S\otimes_RS, \Omega, G)}$.

Although $\mathcal{C}_{(S\otimes_RS, \Omega, G)}$  does not depend on the ordering, that is  the objects and morphisms are the same for a different ordering $\Omega$ on the indices of $G$,  the concept of ``sheaf exactness" does. This can be shown using   the fact that finitely generated torsion free modules over $R$ are projective. We will omit the details. 
To demonstrate this we will give an example of a short exact sequence of $S\otimes_RS$ modules, which is sheaf exact in $\mathcal{C}_{(S\otimes_RS, \Omega, G)}$,
but not sheaf exact in $\mathcal{C}_{(S\otimes_RS, \Omega, G)}$, for a different ordering $\Omega_1$ on the indices of $G$. 

\begin{example}{\bf A Filtration for $\frak{S}$}   \end{example}\ \  Recall the $S\otimes_RS$ modules  $\frak{S} = S_{g_1} \oplus S_{g_2} \oplus \cdots \oplus S_{g_n}$, where each $S_{g_i}$ is a copy of $S$  given in section 2. 
As mentioned at the begining of section 3,  $\frak{S}$  has a natural filtration.
If we let  $\frak{S}^{g_i} = 0 \oplus 0 \oplus \cdots \oplus 0 \oplus S_{g_{i + 1}} \oplus S_{g_{i + 2}} \oplus \cdots 
\oplus S_{g_n}$, then it is easy to see that $\frak{S}^{g_i}/\frak{S}^{g_{i + 1}}$ is finitely generated and projective as a right $S$-module and that $sx = x\sigma_{g_i}(s)$ for all $s \in S, x \in \frak{S}^{g_i}/\frak{S}^{g_{i + 1}}$.  Hence $\frak{S} \in \mathcal{C}_{(S\otimes_RS, \Omega, G)}$ and 
$$\frak{S} = \frak{S}^{g_0} \supseteq \frak{S}^{g_1} \supseteq \frak{S}^{g_2} \supseteq \cdots \supseteq  \frak{S}^{g_n} = 0$$
is the filtration for $\frak{S}$.

\begin{example} \label{stensors} {\bf A Filtration for $S\otimes_RS$ } \end{example}
Let $L, K, S, R, \phi$,  $\phi_{g_k}$,  and $G = \{\sigma_{g_1}, \sigma_{g_2}, \dots , \sigma_{g_n}\}$ be as defined in section 2.  
Recall from Section 2 that $I_{g_{k+1}}  = \{\sigma_{g_{k+1}}, \sigma_{g_{k+2}}, \dots ,\sigma_{g_n}\}$ and 
$$(S\otimes_RS)_{I_{g_{k+1}}}  = (S\otimes_RS)^{g_k} = \{x \in S\otimes_RS | \phi_{g_i}(x) = 0 \ \hbox{for all} \  i \leq k\} $$
The $S\otimes_RS$ module  homomorphism  $\phi$ maps $S\otimes_RS$ to $\frak{S}$. We can pull back the natural filtration on $\frak{S}$ shown above to a filtration for  $S\otimes_RS$, namely with  $\phi^{-1}(\frak{S}^{g_k}) = (S\otimes_RS)^{g_k}$. 
 $\phi$ lifts to a one to one  $S\otimes_RS$ module  homomorphism $\hat{\phi} : (S\otimes_RS)^{g_{k - 1}}/(S\otimes_RS)^{g_k} \to \frak{S}^{g_{k - 1}}/\frak{S}^{g_k}$.  We have  $ \frak{S}^{g_{k - 1}}/\frak{S}^{g_k} \cong S[\sigma_{g_k}]$ and any submodule is finitely generated and projective as a right $S$-module, since $S$ is a Dedekind domain. Since $\hat{\phi}$ preserves the $S\otimes_RS$ module  action, we see that 
$sx = x\sigma_{g_k}(s)$ for each $x \in (S\otimes_RS)^{g_{k - 1}}/(S\otimes_RS)^{g_k} $. 
Hence, with $(S\otimes_RS)^{g_k}$ defined as above, 
{\small \[S\otimes_RS = (S\otimes_RS)^{g_0} \supseteq (S\otimes_RS)^{g_1}\supseteq (S\otimes_RS)^{g_2}\supseteq \dots \supseteq (S\otimes_RS)^{g_{n-1}} \supseteq (S\otimes_RS)^{g_n} = \{0\}, \]}
gives us the  filtration for $S\otimes_R S$.

The above definition gives very little information about  the nature of $(S\otimes_RS)^{g_k}$, however by using our results on the structure of $S\otimes_RS$, we can  gain some insight into the nature 
of these ideals, or their localizations. 

\begin{example} {\bf Special Case : $S\otimes_RS = R[\alpha]\otimes_RS$.}   \end{example}
Let $L, K, S, R, \frak{P}, \frak{Q}, \phi$,  $\phi_{g_k}$,  and $G = \{\sigma_{g_1}, \sigma_{g_2}, \dots , \sigma_{g_n}\}$ be as defined in section 2.  Suppose that $L = K(\theta)$ and 
 $S = R[\theta]$, as is the case for 
cyclotomic fields and certain quadratic extensions of $\Bbb{Q}$. Let   $A_{g_i}(\theta)  \in S\otimes_RS$ be as defined in Definition \ref{Atheta}.  In Lemma  \ref{basis} we saw that 
\[\{A_{g_0}(\theta), A_{g_1}(\theta), A_{g_1}(\theta)A_{g_2}(\theta), \cdots , \ A_{g_1}(\theta)A_{g_2}(\theta)\dots A_{g_{n-1}}(\theta)\}\]
is a basis for $S\otimes_R S$ as a right $S$-module.
Also  from  Lemma \ref{basis} we  see that 
\[(S\otimes_RS)^{g_k}  = (S\otimes_RS)_{I_{ g_{k+1}}} = (\prod_{j \leq k}A_{g_j}(\theta))(S\otimes_RS),\]
in fact,  by considering a reordering of the elements of $G$, we easily see that 
\[(S\otimes_RS)_I = (\prod_{\{i| \sigma_{g_i} \in I^c\}}A_{g_i}(\theta))(S\otimes_RS).\]
 
 Lemma \ref{basis}  also applies when $S \otimes_RS = R[\alpha]\otimes_RS$ for some $\alpha \in S$, with 
$L = K(\alpha)$.  
We also have, by Lemma \ref{basisP}   that 
$$(S\otimes_RS_{\frak{Q}})^{g_i} =  (\prod_{j \leq k}A_{g_j}(\alpha))(S\otimes_RS_{\frak{Q}})$$
when $S\otimes_RS_{\frak{Q}} = R[\alpha]\otimes_RS_{\frak{Q}}$ for some $\alpha \in S$, with 
$L = K(\alpha)$.  Since $S\otimes_RS_{\frak{Q}}$ splits into rings of this type for each maximal ideal 
$\frak{Q}$ of $S$, this gives us a description of the localizations $((S\otimes_RS)^{g_i})_{\frak{Q}}$ for each maximal ideal $\frak{Q}$ of $S$. (Note that $\alpha$ may vary as the ideals $\frak{Q}$ vary).

 \begin{example} {\bf Special Case: $S = \Bbb{Z}[\sqrt{d}]$. } \end{example} 
 Let $L = \Bbb{Q}(\sqrt{d}),  d \equiv 2, 3, \hbox{mod} \ 4$ and $K = \Bbb{Q}$. 
We have $S = \Bbb{Z}[\sqrt{d}]$, $R = \Bbb{Z}$  and $G = \{\sigma_{g_1} = \hbox{id}, \ \sigma_{g_2}\}$, 
where $\sigma_{g_2}(a + \sqrt{d}b) = a - \sqrt{d}b$, for $a, b \in \Bbb{Z}$.
As above, we have a filtration of $S\otimes_{\Bbb{Z}} S$ is given by :
\[S\otimes_{\Bbb{Z}}S = (S\otimes_{\Bbb{Z}}S)^{g_0} \supseteq A_{g_1}(\sqrt{d})(S\otimes_{\Bbb{Z}}S) = (S\otimes_{\Bbb{Z}}S)^{g_1} \supseteq 0 =  (S\otimes_{\Bbb{Z}}S)^{g_2} ,\]
 where 
$A_{g_1}(\sqrt{d}) = \sqrt{d} \otimes 1 - 1 \otimes \sqrt{d}$, 
 Since $\phi_{g_2}(A_{g_1}(\sqrt{d})) = -2\sqrt{d}$ we have 
$\phi_{g_2}((S\otimes_{\Bbb{Z}}S)^{g_1}) = 2\sqrt{d}S$. 
In fact in this case we can also throw light  on  the  structure of $S\otimes_RS$ by looking at the isomorphic  image 
$\phi(S\otimes_RS)$ in $S_{g_1} \oplus S_{g_2}$. 

\begin{lemma} Let $K = \Bbb{Q}, L = {\Bbb{Q}}(\sqrt{d}),  d \equiv 2, 3, \mod 4, \  S = \Bbb{Z}[\sqrt{d}]$  
and $R = \Bbb{Z}$. Let $\phi$ be as defined in section 2.  Then 
$\phi(S\otimes_{\Bbb{Z}}S) = \{(s_1, s_2)| s_1, s_2 \in S \ \hbox{and} \ (s_2 - s_1) \in 2\sqrt{d}S\}.$
\end{lemma}
\begin{proof}Let $x = \Sigma a_i\otimes b_j \in S\otimes_{\Bbb{Z}} S$. Then $\phi(x) = (\Sigma a_ib_j, \Sigma b_j\sigma_{g_2}(a_i))$.
We see that $\Sigma b_j\sigma_{g_2}(a_i) - \Sigma a_ib_j = \Sigma b_j(\sigma_{g_2}(a_i) - a_i) \in 2\sqrt{d}S$.
Hence $\phi(S\otimes_{\Bbb{Z}}S) \subseteq \{(s_1, s_2)| s_1, s_2 \in S \ \hbox{and} \ (s_2 - s_1) \in 2\sqrt{d}S\}$.
On the other hand, let $(s_1, s_2) \in S \oplus S$  be such that $s_2 = s_1 + 2\sqrt{d}s_3$ for some 
$s_3 \in S$. We see that 
 $(s_1, s_2) = \phi(1\otimes s_1 - A_{g_1}(\sqrt{d})(1\otimes s_3))$. This gives us the opposite inclusion and proves the lemma. 
\end{proof}

Now the quotient modules, regarded as S-S bimodules,  in the filtration given above are easy to determine, by examining the images under the map $\phi$. 
We see that 
$(S\otimes_{\Bbb{Z}} S)^{g_0}/(S\otimes_{\Bbb{Z}} S)^{g_1} \cong \phi_{g_1}(S\otimes_{\Bbb{Z}}S)  \cong   S[\sigma_{g_1}]$. 
Also  $(S\otimes_{\Bbb{Z}} S)^{g_1}/\{o\}  \cong \phi_{g_2}(S\otimes_{\Bbb{Z}}S) = 2\sqrt{d}S \cong  S[\sigma_{g_2}]$.

\begin{example}{\bf  A Sequence which is exact but not sheaf exact}\end{example}
Using the previous example, we now present a sequence which is exact in the category $S\otimes_RS$-mod, but is not sheaf exact
in the category $\mathcal{C}_{(S\otimes_RS, \Omega, G)}$.  Let $K = \Bbb{Q}, L = {\Bbb{Q}}(\sqrt{d}),  d \equiv 2, 3, \mod 4, \  S = \Bbb{Z}[\sqrt{d}]$,  $R = \Bbb{Z}$ and $G = \{\sigma_{g_1} = \hbox{id}, \ \sigma_{g_2}\}$, 
where $\sigma_{g_2}(a + \sqrt{d}b) = a - \sqrt{d}b$, for $a, b \in \Bbb{Z}$. Let $\Omega$ be the poset $\{g_1,g_ 2\}$ with the ordering ${g_1} < {g_2}$.   Consider the sequence:
\begin{equation}\label{asterisk}
0 \to (S\otimes_R S)_I \to S\otimes_R S \to (S\otimes_R S)/(S\otimes_R S)_I \to 0,  
\end{equation}
where $I = \{\sigma_{g_1}\}$ and $(S\otimes_RS)_I $ is defined as in section 2.  We have $(S\otimes_RS)_I = A_{g_2}(\sqrt{d})(S\otimes_RS) = \{x \in S\otimes_R S \  \ \hbox{such that}  \ \  \phi_{g_2}(x) = 0\}.$
It is easy to see that $\phi_{g_1}((S\otimes_R S)_I) = 2\sqrt{d}S_{g_1}$.  Since $S_{g_1}$ is isomorphic to $S[\sigma_{g_1}]$ as $S - S$ bimodules and $\phi_{g_1}$ is an S-S bimodule homomorphism, we can conclude that 
  a filtration for $(S\otimes_RS)_I$ is given by:
$$(S\otimes_RS)_I = (S\otimes_R S)_I^{g_0} \supset \{0\} = (S\otimes_R S)_I^{g_1} = (S\otimes_R S)_I^{g_2}.$$
Now  $\phi_{g_2}(S\otimes_RS)_I  = 0$, in fact, $(S\otimes_RS)_I$  is the kernel of the map $\phi_{g_2}: S\otimes_RS \to S_{g_2}$.  Hence
$(S\otimes_RS)/(S\otimes_RS)_I \cong  S_{g_2} \cong  S[\sigma_{g_2}]$ as bimodules,  and thus ${S\otimes_RS   \over (S\otimes_RS)_I} $ has a filtration:
\[{S\otimes_RS   \over (S\otimes_RS)_I} = \Bigg[{S\otimes_RS \over (S\otimes_RS)_I}\Bigg]^{g_0} = \Bigg[{S\otimes_RS \over (S\otimes_RS)_I}\Bigg]^{g_1}
\supset \{0\} = \Bigg[{S\otimes_RS \over (S\otimes_RS)_I}\Bigg]^{g_2}.\]

Consider now the restriction of the maps in the  exact sequence of $S-S$-bimodules, (\ref{asterisk}),  to the following sequence of modules: 
\[0 \to (S\otimes_RS)_I^{g_1} \to (S\otimes_RS)^{g_1} \to  \Bigg[{S\otimes_RS \over (S\otimes_RS)_I}\Bigg]^{g_1} \to 0.\]
Lifting the maps via $\phi$, we get 

\[\xymatrix{
0\ar[r] & ({S\otimes_RS})_I^{g_1}\ar[r] 
& (S\otimes_RS)^{g_1}\ar[r]^{p}\ar[d]^{\phi} &  {\Big[}\frac{(S\otimes_RS)}{(S\otimes_RS)_I}{\Big]}^{g_1}\ar[r]\ar[d]^{\beta} & 0\\
&    & (0, 2\sqrt{d}S[\sigma_{g_2}])\ar[r]^{\pi}  &  S[\sigma_{g_2}] & 
}\]
Here $\beta(s_1\otimes s_2 + (S\otimes_RS)_I) = \phi_{g_2}(s_1\otimes s_2) $ and $\pi(0, s_2) = s_2$, $p$ is the quotient map and $\pi$ is projection onto the second factor. It is easy to see that $\beta p = \pi\phi$. We know, from the discussion above,  that $\beta$ is an isomorphism. Hence, since $\pi$ is not onto, we have that  $p$ cannot be onto.  Hence the sequence on the top row is not exact as a sequence of $S-S$ bimodules and thus the  sequence, (\ref{asterisk}),   is not sheaf exact in the category $\mathcal{C}_{(S\otimes_RS, \Omega, G)}$.

\begin{example} {\bf The  role played by  the ordering in the definition of sheaf exact sequences} \end{example}
We now consider two different orderings on the indices of $G$ in the above example given by the  posets $\Omega$ and $\Omega_1$.  
To demonstrate the subtle role played by the ordering  in the concept of sheaf exactness for sequences,
we give an example of a 
sequence which is sheaf exact in the category $\mathcal{C}_{(S\otimes_RS, \Omega_, G)}$ but not 
sheaf exact in the category $\mathcal{C}_{(S\otimes_RS, \Omega_1, G)}$. 

 Let us use the sequence of $S-S$ bimodules, (\ref{asterisk}),
given above. We have shown that it is not sheaf exact in the category $\mathcal{C}_{(S\otimes_RS, \Omega, G)}$.  We now consider the category obtained by switching the ordering on the indices 
of $G$. To avoid confusion with modules, we will relabel the indices. Let $G = \{\sigma_{x_1}, \sigma_{x_2}\}$, 
where $\sigma_{x_1} = \sigma_{g_2}$ and $\sigma_{x_2} = \sigma_{g_1}$. Now let $\Omega_1 = \{x_1, x_2\}$ be the poset with ordering $x_1 < x_2$. 
 We will show that the sequence of $S- S$ bimodules is sheaf exact in the category  $\mathcal{C}_{(S\otimes_RS, \Omega_1, G)}$. 

The category $\mathcal{C}_{(S\otimes_RS, \Omega_1, G)}$ consists of all $S-S$ bimodules, $M$,  with filtrations
$$M = M^{x_0} \supseteq M^{x_1} \supseteq M^{x_2} = 0,$$
where $M^{x_{i - 1}}/ M^{x_{i}}$ is finitely generated and projective as a right $S$ module, with left action given by $su = u\sigma_{x_i}(s)$ for all $s\in S$ and $u \in M^{x_{i - 1}}/ M^{x_i}$.  Consider the sequence $(\ref{asterisk})$ above , I claim  that the sequences
\begin{equation} \label{dagger}
0 \to (S\otimes_R S)_I^{x_i}  \to (S\otimes_R S)^{x_i}  \to {\Big[}(S\otimes_R S)/(S\otimes_R S)_I{\Big]}^{x_i} \to 0,  
\end{equation}
are exact as sequences of  $S-S$ bimodules, for $i = 0, 1$ and $2$. When $i = 0$, the sequence, \ref{dagger}, is the original sequence, \ref{asterisk}, which is obviously exact. When $i = 2$, the sequence, \ref{dagger}, reduces to $0\to 0 \to 0 \to 0 \to 0$, which is obviously exact. It remains to determine the sequence for $i = 1$.  

Letting $I = \{\sigma_{g_1}\}$ as in the previous exampe, we have 
 $(S\otimes_RS)_I = \{x \in S\otimes_RS | \phi_{g_2}(x) = 0\}$, and $\phi_{g_1} $ gives an isomorphism from $(S\otimes_RS)_I$ to a submodule of $S[\sigma_{g_1}] = S[\sigma_{x_2}]$.  Hence $(S\otimes_RS)_I = (S\otimes_RS)_I^{x_1}$.  Since $(S\otimes_RS)_I$ is the kernel 
of the bimodule homomorphism $\phi_{g_2}: S\otimes_RS \to S[\sigma_{g_2}] = S[\sigma_{x_2}]$, we see that $(S\otimes_RS)^{x_1} = (S\otimes_RS)_I$
and  ${\Big[}(S\otimes_R S)/(S\otimes_R S)_I{\Big]}^{x_1} = 0$.  Hence, when $i = 1$, the sequence, \ref{dagger},  above becomes:
\[0 \to (S\otimes_R S)_I  \to (S\otimes_R S)_I  \to 0 \to 0,\]
and is indeed exact. Thus the sequence, \ref{asterisk},   is sheaf exact in the category $\mathcal{C}_{(S\otimes_RS, \Omega_1, G)}$, but not in $\mathcal{C}_{(S\otimes_RS, \Omega, G)}$.

\section{ \bf Projectives in the category  $\mathcal{C}_{(S\otimes_RS, \Omega, G)}$}
Let $L, K, S, R, \phi, \phi_{g_k}$ and $ G = \{\sigma_{g_1}, \sigma_{g_2}, \sigma_{g_3} \dots \sigma_{g_n} \}$  be as defined in Section 2. Let $\Omega = \{{g_1}, {g_2}, \cdots ,{g_ n}\}$  with ${g_1} < {g_2} < \cdots < {g_n}$  be a total ordering on the indices of  $G$ and let $\mathcal{C}_{(S\otimes_RS, \Omega, G)}$ denote the resulting 
stratified category defined in Section 5. 
   From our definition of projective modules in Section 5, we have an object $P$ of $\mathcal{C}_{(S\otimes_RS, \Omega, G)}$  is projective if $Hom(P, \ )$ is an exact functor, i.e.
for every sheaf exact sequence $0\to M \to N \to Q \to 0$ in $\mathcal{C}_{(S\otimes_RS, \Omega, G)}$, the sequence
\[0 \to Hom(P, M) \to Hom(P, N) \to Hom(P, Q) \to 0\]
is exact in the category of Abelian groups.
We have verified  that $S\otimes_RS$ is in the category $\mathcal{C}_{(S\otimes_RS, \Omega, G)}$
in Example \ref{stensors}. In fact  $S\otimes_RS$ is  projective in the category 
$\mathcal{C}_{(S\otimes_RS, \Omega, G)}$. 

\begin{lemma} 
$S\otimes_R S$ is projective in the category $\mathcal{C}_{(S\otimes_RS, \Omega, G)}$
\end{lemma}

\begin{proof}
We note first  that $S\otimes_R S$ is projective in the 
category of 
$S\otimes_R S$-modules. So if $0\to M_1 \to M_2 \to M_3 \to 0$ 
is sheaf exact in $\mathcal{C}_{(S\otimes_RS, \Omega, G)}$, it is exact in the category of 
$S\otimes_R S$-modules  and projectivity of $S\otimes_R S$ in this category 
guarantees  exactness of 
\[0 \to Hom(S\otimes_RS, M_1) \to Hom(S\otimes_RS, M_2) \to Hom(S\otimes_RS, M_3) \to 0\]
in the category of abelian groups. This tells us that $S\otimes_RS$ is projective 
in $\mathcal{C}_{(S\otimes_RS, \Omega, G)}$.
\end{proof}
It is clear that the proof applies also to the module $S\otimes_RS_{\frak{Q}}$ in the category 
$\mathcal{C}_{(S\otimes_RS_{\frak{Q}}, \Omega, G)}$:

\begin{corollary}
$S\otimes_R S_{\frak{Q}}$ is projective in the category $\mathcal{C}_{(S\otimes_RS_{\frak{Q}}, \Omega, G)}$.

\end{corollary}

According to Theorem \ref{progen}, the category $\mathcal{C}_{(S\otimes_RS, \Omega, G)}$ has a 
projective generator, $P$, such that for any   $M$  in $\mathcal{C}_{(S\otimes_RS, \Omega, G)}$, there is a sheaf exact 
sequence $0 \to N \to P^n \to M \to 0$ for some $n \geq 0$. This projective generator is not  unique.  We can  construct such  a projective generator for $\mathcal{C}_{(S\otimes_RS, \Omega, G)}$ as a direct sum of 
quotients of $S\otimes_RS$. By \cite{Dyer} [Theorem 1.19],  if we construct n projectives
$P_1, P_2, \dots P_n$ with $P_i = P_i^{g_{i-1}}$, (that is  $P_i^{g_{j-1}}/P_i^{g_j} = P_i/P_i^{g_j} = (0)$
for $j < i$), and $P_i/P_i^{g_i} = P_i^{g_{i-1}}/P_i^{g_i} \cong S[\sigma_{g_i}]$, as S-S bimodules,
then $P = P_1 \oplus P_2 \oplus \dots \oplus P_n$
is a projective generator for $\mathcal{C}_{(S\otimes_RS, \Omega, G)}$.

Recall from section 2  that $I_{g_i} = \{\sigma_{g_k}\}_{k  \geq  i}$ and for $I \subseteq G$, we have $(S\otimes_RS)_I = 
\{x \in S\otimes_RS| \phi_{g_k}(x) = 0 \ \hbox{for} \  \sigma_{g_k} \not\in  I\}$.

\begin{definition}  \label{P}
Let $K, L, S$ and $R$ be as defined in section 2. 
We let 
$$(S\otimes_RS)_i =  (S\otimes_RS)_{I_{g_{i }}^c} = \{ x \in S\otimes_RS| \phi_{g_k}(x) = 0 \ \hbox{for} \  k \geq i\}   \ \  i = 1, 2, \dots , n.$$ 
and we let $P_i = (S\otimes_RS)/(S\otimes_RS)_{i}$.
\end{definition}

\begin{theorem}
\label{PlusP} Let $L, K, S, R, \frak{Q}, \frak{P}, G = \{\sigma_{g_1}, \sigma_{g_2}, \cdots , \sigma_{g_n}\}, \phi$ 
and $\phi_{g_k}$ be as defined in Section 2. 
 Let $P_i$ be as in Definition \ref{P} for $ 1\leq i \leq n$. Then $(P_i)^{g_{j - 1}}/P_i^{g_j} = 0$ if $j < i$
and $ P_i^{g_{i - 1}}/P_i^{g_i} = P_i/P_i^{g_i} \cong S[\sigma_{g_i}]$. Also $P_i$ is projective in the category $\mathcal{C}_{(S\otimes_RS, \Omega, G)}$. 
\end{theorem}
 \begin{proof}  We note that $P_i$ is isomorphic to the projection of $\phi(S\otimes_RS)$ onto the last $n - i + 1$ factors of $S_{g_1} \oplus S_{g_2} \oplus S_{g_3} \oplus \dots \oplus S_{g_n}$. We will again exploit the natural filtration on 
 $S_{g_1} \oplus S_{g_2} \oplus \dots \oplus S_{g_n}$ to get a filtration for $P_i$. 
 
 For any $i$ and $k$, with $1 \leq i \leq n$ and $k \geq i$,  we have  a lifting of $\phi_{g_k} : S\otimes_RS \to S[\sigma_{g_k}]$
 to $\phi_{g_k, i} : P_i \to S[\sigma_{g_k}]$, since $(S \otimes_R S)_i \subseteq \ker \phi_{g_k}$. For each $i$, with $1\leq i \leq n$, let 
 $$P_i^{g_j}  =  \Big\{\  \begin{matrix}
& P_i & \hbox{if}  & 0\leq j < i  \\
&   \{x \in P_i | \phi_{g_k, i}(x) = 0 \  \hbox{for} \  i \leq k \leq j\} &\mbox{if}  &\ j \geq i  &
\end{matrix}
$$
Then for each $i$, $1 \leq i \leq n$, we claim that 
\begin{equation} 
\label{filt} P_i^{g_0} = P_i \supseteq P_i^{g_1} \supseteq P_i^{g_2} \supseteq \dots \supseteq P_i^{g_n} = 0
\end{equation}
is a filtration for $P_i$ in the category $\mathcal{C}_{(S\otimes_RS, \Omega, G)}$.

For  $j\geq i$, we let   $\bar{\phi}_{{g_j}, i} :  P_i^{g_{j-1}} \to S[\sigma_{g_j}]$ denote the   restriction of the map $\phi_{{g_j}, i}: P_i \to S[\sigma_{g_j}]$ defined  above. 
By definition, we have $P_i^{g_j} \subseteq \ker \bar{\phi}_{{g_j}, i}$.  On the other hand,  if $x \in P_i^{g_{j - 1}}$ is in $\ker \bar{\phi}_{{g_j}, i}$, then $x \in P_i$ and  $\phi_{{g_k}, i}(x) = 0$ for $i \leq k \leq j$ since $x \in P_i^{g_{j - 1}}$ and in 
$\ker{\bar{\phi}_{{g_j}, i}}  \subseteq \ker \phi_{{g_j}, i}$. Hence $x \in P_i^{g_j}$ and $\ker \bar{\phi}_{{g_j}, i} = P_i^{g_j}$. 
Thus using the $S\otimes_RS$ module homomorphism $\bar{\phi}_{{g_j}, i}$, we see that 
 $$P_i^{g_{j - 1}}/P_i^{g_j}  \cong   \Big\{\  \begin{matrix}
& 0 & \hbox{if}  & 0\leq j < i  \\
&  \hbox{a submodule of} \ \ S[\sigma_{g_j}]  
&\mbox{if}  &\ j \geq i  &
\end{matrix}
$$
This shows that (\ref{filt})  is a filtration of $P_i$, since the quotients are finitely generated and projective as right $S$-modules and have the appropriate left action by $S$.  We can thus conclude 
that $P_i$ is in the category $\mathcal{C}_{(S\otimes_RS, \Omega, G)}$, for $1 \leq i \leq n$. 

We can see that $\bar{\phi}_{{g_i}, i} : P_i^{g_{i - 1}}/P_i^{g_i} \to S[\sigma_{g_i}]$ is an isomorphism, by examining   the image of the coset class  $[1\otimes s] \in P_i = P_i^{g_{i - 1}}$, for  $s\in S$. We have $\bar{\phi}_{{g_i}, i}([1\otimes s] + P_i^{g_i}) = \phi_{{g_i}, i}([1\otimes s]) = \phi_{{g_i}, i}(1\otimes s  + (S\otimes_RS)_i)  = \phi_{g_i}(1\otimes s) = s$, since $\phi(S\otimes_RS)_i = 0$. Hence $\bar{\phi}_{{g_i}, i} :  P_i^{g_{i - 1}}/P_i^{g_i} \to S[\sigma_{g_i}]$ is onto and since $P_i^{g_i} = \ker \bar{\phi}_{{g_i}, i} $, it is  an isomorphism.

It remains to show that $P_i$ is projective in the category $\mathcal{C}_{(S\otimes_RS, \Omega, G)}$ for each $i, \  1\leq i \leq n$. 
For each $i$, $1 \leq i \leq n$, we need only show that the map $Hom(P_i, M_2) \to Hom(P_i,
M_3)$ is onto whenever 
\begin{equation}\label{3ms}
0 \to M_1 \to M_2 \to M_3 \to 0
\end{equation}
is a sheaf  exact sequence in Category $\mathcal{C}_{(S\otimes_RS, \Omega, G)}$. Let $f:P_i \to M_3$ be a homomorphism of $S\otimes_R S$ modules. We know that 
$f$ maps $P_i$ into $M_3^{g_{i-1}}$ since $P_i = P_i^{g_{i-1}}$. Thus we can restrict our attention
to the sequence 
\[0 \to M_1^{g_{i-1}} \to M_2^{g_{i-1}} \to M_3^{g_{i-1}} \to 0,\]
which is also exact since \ref{3ms} is a sheaf  exact sequence in category $\mathcal{C}_{(S\otimes_RS, \Omega, G)}$.
 We can lift 
$f$ to a map $f^{!}: S\otimes_RS \to M_3^{g_{i-1}}$ by composing $f$ with the projection 
homomorphism. Now since $S\otimes_RS$ is projective in the category of $S\otimes_RS$-modules we can lift 
$f^!$ to a map $f^{!!}:S\otimes_RS \to M_2^{g_{i-1}}$. 

\[\xymatrix{
& & & {S\otimes_RS}\ar[ddl]_{f^{!!} }\ar[d]\ar@/^/[dd]^{f^{!}}
&\\
& & & {P_i}\ar[d]_f
&\\
0\ar[r] & {M_1^{g_{i - 1}}}\ar[r]  & {M_2^{g_{i - 1}}}\ar[r]  & {M_3^{g_{i - 1}}}\ar[r] &  0 }\]

Recall that $P_i = S\otimes S/(S\otimes_RS)_{i}$,
where \[(S\otimes_RS)_i = \{ x \in S\otimes_RS| \phi_k(x) = 0 \ \hbox{for} \   k \geq i \}.\]
We will show that $f^{!!}((S\otimes_RS)_i) = 0$, thus allowing us to lift $f^{!!}$ to the sought after element of 
$Hom(P_i, M_2)$. 

To achieve this, we shall consider $f^{!!}: S\otimes_RS \to M_2^{g_{i-1}}$ as a map in a stratified category constructed using a different ordering on the indices of $G$. With $i$ as in the above diagram, 
let $\sigma_{x_1} =  \sigma_{g_i}, \  \sigma_{x_2} = \sigma_{g_{i + 1}}, \   \cdots ,  \  \sigma_{x_{n - i + 1}} = \sigma_{g_n},   \  \sigma_{x_{n - i + 2}} = \sigma_{g_1}, \cdots  , \ \sigma_{x_n} = \sigma_{g_{i - 1}}$. 
Let $\Omega_2$ be the poset giving the total ordering on the indices of $G$ ;  $x_1 < x_2 <  \cdots
< x_n$. 
Let  $\mathcal{C}_{(S\otimes_RS, \Omega_2, G)}$ be the stratified exact category corresponding to the ordering $\Omega_2$. An  $S-S$ bimodule, $N$, may be an object in both categories, $\mathcal{C}_{(S\otimes_RS, \Omega, G)}$ and $\mathcal{C}_{(S\otimes_RS, \Omega_2, G)}$.  For such a module, $N$,  we let $N^{g_i}$ denote a submodule in its filtration in $\mathcal{C}_{(S\otimes_RS, \Omega, G)}$, and we let $N^{x_i}$ denote a submodule in its filtration in the category $\mathcal{C}_{(S\otimes_RS, \Omega_2, G)}$.

The module  $S\otimes_RS$ is  in the category $\mathcal{C}_{(S\otimes_RS, \Omega_2, G)}$, with a  filtration in that category given by
\[S\otimes_R S = (S\otimes_R S)^{x_0} \supseteq (S\otimes_RS)^{x_1} \supseteq \dots \supseteq (S\otimes_RS)^{x_n} =0\]
where $(S\otimes_RS)^{x_i} = \{x \in S\otimes_RS | \phi_{x_k}(x) = 0 \ \hbox{if} \ k \leq i \}$.  
This follows from Lemma \ref{stensors}, since the order chosen on the group did not play a part in that proof. We see that the $S-S$ bimodule, $(S\otimes_RS)_i$ coincides with the 
$S-S$ bimodule $(S\otimes_RS)^{x_{n - i + 1}}$  in the category $\mathcal{C}_{(S\otimes_RS, \Omega_2, G)}$  

We also have that the 
$S-S$ bimodule $M_2^{g_{i - 1}}$ is in the category $\mathcal{C}_{(S\otimes_RS, \Omega_2, G)}$, with filtration given by:
$$ M_2^{g_{i - 1}} = (M_2^{g_{i - 1}})^{x_0} \supseteq  (M_2^{g_{i - 1}})^{x_1}   \supseteq \dots \supseteq 
 (M_2^{g_{i - 1}})^{x_n} = 0
$$
where
 $$(M_2^{g_{i - 1}})^{x_k} =  \Big\{\  \begin{matrix}
& M_2^{g_{k + i - 1}} & \hbox{if}  & 0\leq k <  n - i + 1  \\
&   0  &\mbox{if}  &\  k \geq n - i + 1  &
\end{matrix}
$$

Since  $(M_2^{g_{i - 1}})^{x_{k - 1}}/ (M_2^{g_i})^{x_k} \cong M_2^{g_{k + i -2}}/M_2^{g_{k + i - 1}} , 1 \leq k \leq n-i + 1$, this module is isomorphic to a submodule of $S[\sigma_{g_{k + i - 1}}] = S[\sigma_{x_{k}}]$. Hence the quotients have the required structure.  
Now $f^{!!} \in Hom_{S\otimes_RS}(S\otimes_RS, M_2^{g_{i - 1}})$ and since $\mathcal{C}_{(S\otimes_RS, \Omega_2, G)}$ is a full subcategory of $S\otimes_RS$ modules $f^{!!}$ must map $(S\otimes_RS)^{x_{n - i + 1}}$ into 
$M_2^{x_{n - i + 1}} = 0$ by Lemma \ref{maps}.  Hence $f^{!!}$ factors through a map $f^{!!!} : P_i \to M_2^{g_{i - 1}}$ and it is clear from the construction of these maps that $f^{!!!}$ is a lifting of the original map $f \in Hom_{S\otimes_RS}(P_i, M_3)$. 

\[\xymatrix{
& & & {S\otimes_RS}\ar[ddl]_{f^{!!} }\ar[d]\ar@/^/[dd]^{f^{!}}
&\\
& & & {P}\ar[dl]^{f^{!!!}}\ar[d]_f
&\\
0\ar[r] & {M_1^{g_{i - 1}}}\ar[r]  & {M_2^{g_{i - 1}}}\ar[r]  & {M_3^{g_{i - 1}}}\ar[r] &  0 }\]

Since $Hom$ is a left exact functor, this suffices to prove that for  $1 \leq i \leq n$, the sequence 
\[0 \to Hom_{S \otimes_RS}(P_i, M_1) \to Hom_{S \otimes_RS}(P_i, M_2) \to Hom_{S \otimes_RS}(P_i, M_3) \to 0\]
is exact in the category of Abelian groups  for every exact sequence $0 \to M_1 \to M_2 \to M_3 \to 0$ in the category $\mathcal{C}_{(S\otimes_RS, \Omega, G)}$ and hence that $P_i$ is projective in this category for $1\leq i \leq n$. 

\end{proof}

\begin{definition}  \label{Q}
Let $K, L, S, R, \frak{Q}, \frak{P}, \phi,$  and $\phi_{g_k}$ be as defined in Section 2.  Recall that 
$I_{g_i} = \{\sigma_{g_{i }}, \sigma_{g_{i + 2}}, \cdots , \sigma_{g_n}\}$. 
We let 
$$(S\otimes_RS_{\frak{Q}})_i =  (S\otimes_RS_{\frak{Q}})_{I_{g_{i }}^c} = \{ x \in S\otimes_RS_{\frak{Q}}| \phi_{g_k}(x) = 0 \ \hbox{for} \  k \geq i\}   \ \  i = 1, 2, \dots , n.$$ 
and we let $P_{\frak{Q}, i}  = (S\otimes_RS_{\frak{Q}})/(S\otimes_RS_{\frak{Q}})_{i}$. 
\end{definition}

It is clear that the above proof applies when we replace $P_i$   by $P_{\frak{Q}, i}$, $1 \leq i \leq n$.  In particular we have :

\begin{corollary} \label{PlusQ} Let $L, K, S, R, \frak{Q}, \frak{P}, G = \{\sigma_{g_1}, \sigma_{g_2}, \cdots , \sigma_{g_n}\}, \phi$ 
and $\phi_{g_k}$ be as defined in Section 2. 
 Let $P_{\frak{Q}, i}$ be as in Definition \ref{Q} for $ 1\leq i \leq n$. Then $P_{\frak{Q}, i}^{g_{j - 1}}/P_{\frak{Q}, i}^{g_j} = 0$ if $j < i$
and $P_{\frak{Q}, i}^{g_{i - 1}}/P_{\frak{Q}, i}^{g_i} = P_{\frak{Q}, i}/P_{\frak{Q}, i}^{g_i} \cong S_{\frak{Q}}[\sigma_{g_i}]$. Also $P_{\frak{Q}, i}$ is projective in the category $\mathcal{C}_{(S\otimes_RS_{\frak{Q}}, \Omega, G)}$. 

\end{corollary}

We will also apply this result with the Galois group $G$ replaced by the Galois group $E = E(\frak{Q}|\frak{P})$ and the 
 ring $R$ replaced by $S_E$. To avoid confusion later, we will make the appropriate definitions here.

\begin{definition}  \label{EQ}
Let $K, L, S, R, \frak{Q}, \frak{P}, E = E(\frak{Q}|\frak{P}) = \{\sigma_{e_1}, \sigma_{e_2}, \cdots , 
\sigma_{e_{|E|}}\},       \phi'$  and $\phi_{e_k}', 1\leq k \leq |E|$ be as defined in Section 2.  Letting   
$I_{E, e_i} = \{\sigma_{e_{i }}, \sigma_{e_{i + 1}}, \cdots , \sigma_{e_{|E|}}\}$
and  
$$(S\otimes_{S_E}S_{\frak{Q}})_{E, i} =  (S\otimes_{S_E}S_{\frak{Q}})_{I_{E, e_i}^c} = \{ x \in S\otimes_{S_E}S_{\frak{Q}}|  \phi_{e_k}'(x) = 0 \ \hbox{for} \  k \geq i\}   \ \  i = 1, 2, \dots , n,$$ 
we define  $P_{E, \frak{Q}, i}  $ as   $P_{E, \frak{Q}, i}  = (S\otimes_{S_E}S_{\frak{Q}})/(S\otimes_{S_E}S_{\frak{Q}})_{I_{E, e_i}^c}$. 
\end{definition}

It is clear that the above proof applies when we replace $R$ by $S_E$, $G$ by $E$ and  $P_i$   by $P_{E, \frak{Q}, i} $, $1 \leq i \leq |E|$.  In particular we have :

\begin{corollary} \label{PlusEQ} Let $L, K, S, R, \frak{Q}, \frak{P}, E = E(\frak{Q}|\frak{P}) = \{\sigma_{e_1}, \sigma_{e_2}, \cdots , 
\sigma_{e_{|E|}}\},       \phi'$  and $\phi_{e_k}'$ be as defined in Section 2.   Let $P_{E, \frak{Q}, i}$ be as in Definition \ref{EQ} for $ 1\leq i \leq |E|$. 
Let $\Omega'$ be the poset $\{e_1, e_2, \cdots , e_{|E|}\}$ with ordering $e_1 < e_2 < \cdots < e_{|E|}$, and let 
$$P_{E, \frak{Q}, i} = P_{E, \frak{Q}, i}^{e_0} \supseteq P_{E, \frak{Q}, i}^{e_1} \supseteq \cdots 
\supseteq P_{E, \frak{Q}, i}^{e_{|E|}} = \{0\}$$
be a filtration for $P_{E, \frak{Q}, i}$ in the category  $\mathcal{C}_{(S\otimes_{S_E}S_{\frak{Q}}, \Omega', E)}$. 
Then $P_{E, \frak{Q}, i}^{e_{j - 1}} /P_{E, \frak{Q}, i}^{e_j} = 0$ if $j < i$
and $P_{E, \frak{Q}, i}^{e_{i - 1}} /P_{E, \frak{Q}, i}^{e_i}  = P_{E, \frak{Q}, i}/P_{E, \frak{Q}, i}^{e_i}  \cong S_{\frak{Q}}[\sigma_{e_i}]$. Also $P_{E, \frak{Q}, i}$ is projective in the category $\mathcal{C}_{(S\otimes_{S_E}S_{\frak{Q}}, \Omega', E)}$. 
\end{corollary}

\section{ The Algebra $\mathcal{A}_{S\otimes_RS}(P)$} 
Let $L, K, S, R, \phi, \phi_{g_k}$ and $ G = \{\sigma_{g_1}, \sigma_{g_2}, \sigma_{g_3} \dots \sigma_{g_n} \}$  be as defined in Section 2. Let us assume that the indices of $G$ form a poset, $\Omega = \{{g_1}, {g_2}, \cdots , {g_n}\}$,   with ordering, ${g_1} < {g_2} < \cdots < {g_n}$.  Let $\mathcal{C}_{(S\otimes_RS, \Omega, G)}$ denote the resulting 
stratified category defined in Section 5. 
Let  $P_i, \ 1\leq i \leq n$, be  the modules defined in Definition \ref{P}, and let   $P = P_1 \oplus P_2 \oplus P_3 \oplus \dots \oplus P_n$. By Theorem \ref{PlusP}, and \cite{Dyer}[Theorem 1.19], $P$ is  a projective generator for the category $\mathcal{C}_{(S\otimes_RS, \Omega, G)}$. 
Recall that  $\mathcal{A}_{S\otimes_RS}(P) = End_{S\otimes_RS}(P)^{op}$.
By Theorem \ref{progen}, 
we have a functor 
\[ F = Hom_{S \otimes_RS}(P, -):M \mapsto Hom_{S \otimes_RS}(P, M),\]
from $\mathcal{C}_{(S\otimes_RS, \Omega, G)}$  to the category of $\mathcal{A}_{S\otimes_RS}(P)$-modules. This functor is fully faithful and 
\[0 \to M \to N \to Q \to 0\] 
is sheaf exact in $\mathcal{C}_{(S\otimes_RS, \Omega, G)}$  if and only if 
\[0 \to FM \to FN \to FQ \to 0\] 
is exact in the category of $\mathcal{A}_{S\otimes_RS}(P)$-modules.  
The definition of $\mathcal{A}_{S\otimes_RS}(P)$ depends on the choice of projective generator from  the category $\mathcal{C}_{(S\otimes_RS, \Omega, G)}$. Although a different choice of projective generator, $P'$, 
gives a different algebra, $\mathcal{A}_{S\otimes_RS}(P')$, we have that the categories $\mathcal{A}_{S\otimes_RS}(P)-$mod and $\mathcal{A}_{S\otimes_RS}(P')-$mod  are Morita equivalent, see Lemma \ref{Newprogen} and the discussion prior to it. 

In this section, we determine the structure of $\mathcal{A}_{S\otimes_RS}(P)$ explicitly as a quotient of a matrix ring.   First we calculate $Hom_{S\otimes_RS}(P_i, P_j)$, where $P_i$ and $P_j$ are defined 
in the previous section. 

\begin{lemma} \label{ringmaps} Let $T$ be a commutative ring with identity and $I$ and $J$ ideals of $T$. 
Let 
\[ [J : I]  = \{x \in T | \ xI \subseteq J \}.\]
Let $f $ be an element of $ Hom_T(T/I, T/J)$, then $f(1 + I)  \in  [J : I]/J $
and the map 
$F:  Hom_T(T/I, T/J)  \to   [J : I]/J $ given by 
$$F(f) = f(1 + I)$$
is an isomorphism of T- modules with inverse $\theta: [J : I]/J \mapsto Hom(T/I, T/J)$ given by:
$$[\theta(x + J)](y + I) = xy + J$$
\end{lemma}

\begin{proof} 
The T-module action on $Hom_T(T/I, T/J) $ is given by $(tf)(t_1) = f(tt_1)$, for $f \in Hom_T(T/I, T/J) $
and $t, t_1 \in T$.  It is easy to see that $F$ is a  homomorphism of $T$ modules, by checking that 
$F(tf) = tF(f)$ and 
$F(f_1 + f_2) = F(f_1) + F(f_2)$ for all $t \in T$ and $f, f_1, f_2 \in Hom_T(T/I, T/J) . $ Also, since 
$if(1 + I) = f(i + I) = 0 + J$ for all $i \in I$,  we must have $f(1 + I) \in  [J : I]/J$. 

To show that $F$ is an isomorphism we show that the inverse of $F$ is given by the   homomorphism of $T$ modules, $\theta : [J : I]/J \mapsto Hom(T/I, T/J)$, defined as 
\begin{equation}
\label{eqn} [\theta(x + J)](y + I) = xy + J.
\end{equation}
That $\theta$ is well defined, follows from the definition of $[J : I]/J$. That it is a T-homomorphism 
is also obvious. 
  Given any  homomorphism $f: T/I \mapsto T/J$ and $y \in T$, 
  $$\theta(F(f))(y + I) = \theta(f(1 + I))(y + I) =  yf(1 + I) = f(y + I)$$
  On the other hand, for $x \in T$, 
$$F(\theta(x + J))  = [\theta(x + J)](1 + I) = x + J.$$
Hence $F$ and $\theta$ are inverses and are isomorphisms of $T$ modules.
\end{proof}

Recall that if $I$ is a subset of 
$G  = \{\sigma_{g_1}, \sigma_{g_2}, \dots , \sigma_{g_n}\}$, then $(S\otimes_RS)_I = \{ x \in S\otimes_RS | \phi_{g_i}(x) = 0 \ \hbox{for all} \ x \ \hbox{such that} \ \sigma_{g_i} \not\in I\}$. 

\begin{theorem} \label{ringmapsstensors} Let $I_{g_i} = \{\sigma_{g_i}, \sigma_{g_{i + 1}}, \dots ,\sigma_{g_n}\}$ and $I_{g_j} = \{\sigma_{g_j}, \sigma_{g_{j + 1}}, \dots , \sigma_{g_n}\}$   for  $1 \leq i,  \leq n$.
 We have 
\[P_i = S\otimes_RS/(S\otimes_RS)_{I_{g_i}^c},\ \ \  P_j = S\otimes_RS/(S\otimes_RS)_{I_{g_j}^c}.\]
Then 
\[[(S \otimes_RS)_{I_{g_j}^c} : (S \otimes_RS)_{I_{g_i}^c}]  = (S \otimes_RS)_{I_{g_i} \cup I_{g_j}^c}\]
and 
\[ F:  Hom_{S\otimes_RS}(P_i, P_j) \to  (S\otimes_RS)_{I_{g_i}\cup I_{g_j}^c}/(S\otimes_RS)_{I_{g_j}^c},  1 \leq i, j \leq n\]
where $F(f) = f(1 \otimes 1 +  (S\otimes_RS)_{I_{g_i}^c})$,  for $f \in Hom_{S\otimes_RS}(P_i, P_j) $,   is an isomorphism of $S\otimes_RS$ modules. 
\end{theorem}

\begin{proof}
We see that $[(S \otimes_RS)_{I_{g_j}^c} : (S \otimes_RS)_{I_{g_i}^c}] = \{x \in S \otimes_RS | x(S \otimes_RS)_{I_{g_i}^c} \subseteq (S \otimes_RS)_{I_{g_j}^c}\}$. 
Now looking at the image of $S \otimes_RS$ under the map $\phi$ from section 2, we see that $x(S \otimes_RS)_{I_{g_i}^c}
\subseteq (S \otimes_RS)_{I_{g_j}^c}$ if and only if
$\phi_{g_k}(x) = 0$ when $\sigma_{g_k}  \in I_{g_i}^c \cap I_{g_j}$, that is if and only if $x \in (S \otimes_RS)_{I_{g_i} \cup I_{g_j}^c}$.
 Hence by the above lemma, the result is true. 
\end{proof}

It is not difficult to apply the same proof to $S\otimes_RS_{\frak{Q}}$:

\begin{corollary} \label{homQ}  Let $I_{g_i} = \{\sigma_{g_i}, \sigma_{g_{i + 1}}, \dots ,\sigma_{g_n}\}$ and $I_{g_j} = \{\sigma_{g_j}, \sigma_{g_{j + 1}}, \dots , \sigma_{g_n}\}$   for  $1 \leq i,  \leq n$. We have 
\[P_{\frak{Q}, i}  = S\otimes_RS_{\frak{Q}}/(S\otimes_RS_{\frak{Q}})_{I_{g_i}^c},\ \ \  P_{\frak{Q}, j} = S\otimes_RS_{\frak{Q}}/(S\otimes_RS_{\frak{Q}})_{I_{g_j}^c}.\]
Then the map
\[ F : Hom_{S\otimes_RS_{\frak{Q}}}(P_{\frak{Q}, i}, P_{\frak{Q}, j}) \to  (S\otimes_RS_{\frak{Q}})_{I_{g_i}\cup I_{g_j}^c}/(S\otimes_RS_{\frak{Q}})_{I_{g_j}^c},  1 \leq i, j \leq n\]
where $F(f) = f(1 \otimes 1 +  (S\otimes_RS_{\frak{Q}})_{I_i^c})$,   for $f \in
 Hom_{S\otimes_RS_{\frak{Q}}}(P_{\frak{Q}, i}, P_{\frak{Q}, j})$   is an isomorphism of $S\otimes_RS_{\frak{Q}}$ modules. 
\end{corollary}

We can now find a concrete realization of $\mathcal{A}_{S\otimes_RS}(P)$  as a quotient of a subring of $M_{n\times n}(S\otimes_RS)$, 
the ring of n by n matrices over $S\otimes_RS$. As above, we let $I_{g_l} = \{\sigma_{g_l},  \sigma_{g_{l + 1}}, \dots \sigma_{g_n}\}$ for $1 \leq l \leq n$.
We let $\mathcal{S}_{S\otimes_RS}(P)$ be the subring of $M_{n\times n}(S\otimes_RS)$ 
consisting of all matrices with $(i, j)$ entry in $(S\otimes_RS)_{I_{g_i}\cup I_{g_j}^c}$. We represent this ring 
as an array:
\[\mathcal{R}_{S\otimes_RS}(P) = 
\begin{pmatrix}
S\otimes_RS & S\otimes_RS & S\otimes_RS & \dots \dots & \\
(S\otimes_RS)_{I_{g_2}} & S\otimes_RS & S\otimes_RS & \dots \dots &\\
(S\otimes_RS)_{I_{g_3}} & (S\otimes_RS)_{I_{g_3} \cup I_{g_2}^c} & S\otimes_RS & \dots \dots &\\
(S\otimes_RS)_{I_{g_4}} & (S\otimes_RS)_{I_{g_4} \cup I_{g_2}^c} & (S\otimes_RS)_{I_{g_4}\cup I_{g_3}^c} & \dots \dots &\\
\vdots & \vdots & \vdots & \ & \\
\ &\ &\ &\ &
\end{pmatrix}
\]

To show that $\mathcal{R}_{S\otimes_RS}(P)$ is indeed a subring of $M_{n\times n}(S\otimes_RS)$, we note  that, identifying
$S \otimes_RS$ with its image under the map $\phi$,  
$(S\otimes_RS)_A(S\otimes_RS)_B$ has zero components except perhaps in $A \cap B$. Now since 
$(I_{g_i} \cup I_{g_k}^c) \cap (I_{g_k} \cup I_{g_j}^c) \subseteq (I_{g_i} \cup I_{g_j}^c)$ we see that
\[(S\otimes_RS)_{I_{g_i} \cup I_{g_k}^c}
(S\otimes_RS)_{I_{g_k} \cup I_{g_j}^c} \subseteq (S\otimes_RS)_{I_{g_i} \cup I_{g_j}^c}.\]

We define $\mathcal{I}_{S\otimes_RS}(P)$ to be the set of matrices in $M_{n \times n}(S\otimes_RS)$ with $(i, j)$ entry
in $(S\otimes_RS)_{I_{g_j}^c}$. I claim that $\mathcal{I}_{S\otimes_RS}(P)$  is an ideal of $\mathcal{R}_{S\otimes_RS}(P)$. Clearly $\mathcal{I}_{S\otimes_RS}(P) \subset \mathcal{R}_{S\otimes_RS}(P)$ and if $
A_1, A_2 \in  \mathcal{I}_{S\otimes_RS}(P)$, then $A_1 - A_2 \in  \mathcal{I}_{S\otimes_RS}(P)$. 
 Since each $(S\otimes_RS)_{I_{g_j}^c}$ is an ideal of $(S\otimes_RS)$ it is clear that
 $(S\otimes_RS)_{I_{g_i}\cup I_{g_k}^c}(S\otimes_RS)_{I_{g_j}^c} \subseteq (S\otimes_RS)_{I_{g_j}^c}$ and $\mathcal{I}_{S\otimes_RS}(P)$ is a left 
ideal of $\mathcal{R}_{S\otimes_RS}(P)$. 
Since $I_{g_k}^c \cap (I_{g_k}\cup I_{g_j}^c) \subseteq I_{g_j}^c$, we have $(S\otimes_RS)_{I_{g_k}^c}(S\otimes_RS)_{I_{g_k} \cup I_{g_j}^c} \subseteq (S\otimes_RS)_{I_{g_j}^c}$.
Hence $\mathcal{I}_{S\otimes_RS}(P)$ is an ideal of $\mathcal{R}_{S\otimes_RS}(P)$. 

Now since $\mathcal{R}_{S\otimes_RS}(P)$  is a ring and $\mathcal{I}_{S\otimes_RS}(P)$ is an ideal of $\mathcal{R}_{S\otimes_RS}(P)$ we can form a quotient  ring, in fact an $S\otimes_RS$
algebra, $\mathcal{R}_{S\otimes_RS}(P)/\mathcal{I}_{S\otimes_RS}(P)$. We claim that
the algebra $\mathcal{A}_{S\otimes_RS}(P)$ is isomorphic to this algebra. 

\begin{theorem} \label{endpopp} Let $\mathcal{R}_{S\otimes_RS}(P)$, $\mathcal{I}_{S\otimes_RS}(P)$ and $\mathcal{A}_{S\otimes_RS}(P)$ be as above, then 
\[End_{S\otimes_RS}(P)^{op} = \mathcal{A}_{S\otimes_RS}(P)\cong \mathcal{R}_{S\otimes_RS}(P)/\mathcal{I}_{S\otimes_RS}(P).\]
We have under this isomorphism, that $\mathcal{A}_{S\otimes_RS}(P)$ is an $S\otimes_RS$ algebra, 
Furthermore the natural map $S\otimes_RS \to Cen(End_{S\otimes_RS}(P))$ given by 
$x \to (p \to xp)$, for $x \in S\otimes_RS$ and $p \in P$,  is an isomorphism, 
and  $S\otimes_RS \cong  Cen(\mathcal{A}_{S\otimes_RS}(P))$.
\end{theorem}
\begin{proof} We have $P = P_1 \oplus P_2 \oplus \cdots \oplus P_n$. We define the generalized matrix ring, see Hahn and O'Meara  \cite{HandOM}[Section 4.2A] ,   $\mathcal{M}_{S\otimes_RS}(P)$, as the set of $n\times n$ matrices with $i, j$ entry in $Hom_{S\otimes_RS}(P_j, P_i)$. Let $F = (f_{i, j})$ and 
$G = (g_{i, j})$ be elements of $\mathcal{M}_{S\otimes_RS}(P)$. 
$\mathcal{M}_{S\otimes_RS}(P)$ has the structure of an associative $S\otimes_RS$ algebra with identity. Multiplication is given by the matrix multiplication 
$$F\cdot G = \big(\sum f_{i,k}g_{k, j}\big)$$
and addition and scalar multiplication by,
$$F + G = (f_{i, j} + g_{i, j}), \ \ \ \ \ \mbox{and} \ \ \ \ \ xF = (xf_{i, j}),$$
for $x \in S\otimes_RS$. The identity of the ring is given by $I = (e_{i, j})$, where $e_{i, j}$ is the 
zero map if $i \not= j$ and $i_{i, i}$ is the identity map from $P_i$ to $P_i$.  

We have an $S\otimes_RS$ algebra homomorphism, see  \cite{HandOM}[Section 4.2A] ,  \\
$\Delta: End_{S\otimes_RS}(P) \to \mathcal{M}_{S\otimes_RS}(P)$ given by $f \to (f_{i, j})$ where $f_{i, j}$ is defined to be the composite map
\[\xymatrix{
P_j\ar[r]&P\ar[r]^{f}&P\ar[r]&P_i}\]
If $P = p_1 + p_2 + \cdots + p_n \in P$, we can write $p$ as a matrix column 
$$p = \begin{pmatrix} p_1\\ p_2 \\ \vdots  \\ p_n \end{pmatrix}.$$
For  $f \in End_{S\otimes_RS}(P)$, we have 
$$f(p) = \begin{pmatrix}f_{i, j}\end{pmatrix}\begin{pmatrix}p_1\\ p_2 \\ \vdots  \\ p_n \end{pmatrix}
= \begin{pmatrix}\sum f_{i, k}p_k\\ \vdots  \\ \sum f_{n,k}p_k \end{pmatrix}.$$
From this we see that the map $\Delta$ is one to one and onto. It is not difficult to see that $\Delta$ is an isomorphism of $S\otimes_RS$ algebras. 

From Theorem \ref{ringmapsstensors}, we have 
\[[(S \otimes_RS)_{I_{g_j}^c} : (S \otimes_RS)_{I_{g_i}^c}]  = (S \otimes_RS)_{I_{g_i} \cup I_{g_j}^c}.\]

By lemma \ref{ringmaps}, we have an $S\otimes_RS$ module homomorphism 
$$\theta_{i,j}: (S\otimes_RS)_{I_{g_i} \cup I_{g_j}^c} \to Hom(P_i, P_j)$$
given by 
$$\theta_{i, j}(x)(y + (S\otimes_RS)_{I_{g_i}^c}) = xy +  (S\otimes_RS)_{I_{g_j}^c},$$
with kernel
$$\ker \theta_{i, j} = (S\otimes_RS)_{I_{g_j}^c}.$$
Let  $ \mathcal{R}_{S\otimes_RS}(P)^{op}$ denote the opposite ring of 
$\mathcal{R}_{S\otimes_RS}(P)$. We can view  $ \mathcal{R}_{S\otimes_RS}(P)^{op}$
as the ring of transposes,  $ \mathcal{R}_{S\otimes_RS}(P)^{op} = \{S^t | S \in
 \mathcal{R}_{S\otimes_RS}(P)\}$ endowed with the usual matrix multiplication from $M_n(S\otimes_RS)$. 
we can construct a homomorphism 
$$\hat{\theta}: \mathcal{R}_{S\otimes_RS}(P)^{op} \to \mathcal{M}_{S\otimes_RS}(P),$$
where 
$\hat{\theta}((s_{i, j})^t) = (\theta_{i, j}(s_{i, j}))^t$, for $(s_{i, j}) \in  \mathcal{R}_{S\otimes_RS}(P)$. 
It is not difficult to verify that $\hat{\theta}$ is a homomorphism of $S\otimes_RS$ algebras with Kernel 
$$\ker\hat{\theta} = \{H^t | H \in  \mathcal{I}_{S\otimes_RS}(P)\}.$$
Hence we see that we have an $S\otimes_RS$  algebra isomorphism 
\[End_{S\otimes_RS}(P)^{op} \cong  \mathcal{A}_{S\otimes_RS}(P) \cong 
\frac{ \mathcal{R}_{S\otimes_RS}(P)}{ \mathcal{I}_{S\otimes_RS}(P)}.\]

That $S\otimes_RS \cong  Cen(\mathcal{A}_{S\otimes_RS}(P))$ follows from the lemma below; Lemma \ref{Lastone}, since 
$P_1 = S\otimes_RS$. 
\end{proof}

\begin{lemma} \label{Lastone} Let $T$ be a commutative ring and let $M$ be a left module over $T$. We have 
$$End_T(T\oplus M) \cong \mathcal{M}_T(T\oplus M),$$
where $\mathcal{M}_T(T\oplus M)$ is the generalized matrix ring defined in the proof of Theorem \ref{endpopp}. 
Moreover the map $T \to Cen(End_T(T\oplus M))$ given by $t \to l_t$, where $l_t (y)  =  ty$,  is an isomorphism. 
\end{lemma}

\begin{proof} That  $End_T(T\oplus M) \cong \mathcal{M}_T(T\oplus M),$ follows as in the proof of Theorem \ref{endpopp}. 
In this case each element of  $\mathcal{M}_T(T\oplus M)$ has the form 
$$\left(\begin{array}{cc}f_{1, 1} & f_{1, 2} \\
f_{2, 1} & f_{2, 2}\end{array}\right),$$
where $  \ f_{1, 1} : T \to T, \ \ f_{1, 2}: M \to T, \ \ f_{2, 1}: T \to M, \ \ f_{2, 2}: M \to M. $ and 
as in the proof of Theorem  \ref{endpopp}, each $f\in End_T(T\oplus M) \cong \mathcal{M}_T(T\oplus M)$ maps to such a matrix, 
$(f_{i, j})$. 
It is easy to see that the map $ End_T(T\oplus M) \to \mathcal{M}_T(T\oplus M)$ described in the proof of Theorem  \ref{endpopp}
carries $l_t$ to the matrix
$$\left(\begin{array}{cc} t 1_T& 0_{M T} \\
0_{T M}& t1_M \end{array}\right),$$
for $t \in T$, where where $1_T:$ and $1_M$ are the identity maps on $T$ and $M$ respectively and $0_{M T} : M \to T$ and $0_{T M} : T \to M$ are the 
zero maps. 

Let $A =  \left(\begin{array}{cc}a_{1, 1} & a_{1, 2} \\
a_{2, 1} & a_{2, 2}\end{array}\right)$ be in the center of $\mathcal{M}_T(T\oplus M)$.  Then, for each $m \in M$,  $A$ commutes with 
$$e_{1 1} =  \left(\begin{array}{cc}1_T & 0_{M T} \\
0_{TM}  &  0_M \end{array}\right),$$
where $1_T, \ 0_{T M}$ and $0_{M T}$ are as above and  $0_T : T \to T$ is the zero map.  
The matrix  identity $Ae_{1, 1}  = e_{1, 1}A$ gives that $a_{2, 1} = 0_{TM}$ and $a_{1, 2} = 0_{MT}$. 
 Hence 
 $$A =  \left(\begin{array}{cc} a_{1 1} & 0_{M T}  \\
0_{T M} & a_{2, 2}  \end{array}\right),$$
For each $m \in M$, 
let $E_{m} =   \left(\begin{array}{cc}1_T & 0_{M T} \\
f_m  &  1_M \end{array}\right),$ where $f_m : T \to M$ is the T homomorphism with $f_m(1) = m$. 
The matrix equation $AE_m = E_mA$ gives that 
$a_{2, 2}f_m = f_ma_{1, 1}$ for all $m \in M$.  Hence $a_{2, 2}f_m(1) = a_{2, 2}(m) = f_m(a_{1, 1}(1))  = a_{1, 1}(1) m $. Letting  $t = a_{1, 1}(1)$, 
we have 
$$A =  \left(\begin{array}{cc}t 1_T & 0_{M T}  \\
0_{TM}  & t 1_M \end{array}\right),$$
proving the Lemma.
\end{proof}

The arguments from the proof of Theorem  \ref{endpopp}  also apply to the $S\otimes_RS_{\frak{Q}}$ module $P_{\frak{Q}} = \oplus_{i = 1}^nP_{\frak{Q}, i}$ to give
\begin{corollary} \label{AlgEQ} 
\[End_{S\otimes_{R}S_{\frak{Q}}}(P_{ \frak{Q}})^{op} = \mathcal{A}_{S\otimes_{R}S_{\frak{Q}}}(P_{ \frak{Q}})\cong \mathcal{R}_{S\otimes_{R}S_{\frak{Q}}}(P_{ \frak{Q}})/\mathcal{I}_{S\otimes_{R}S_{\frak{Q}}}(P_{ \frak{Q}}).\]
\end{corollary}

Let $P_{E, \frak{Q}, i}, 1 \leq i \leq |E|$ be as in definition \ref{EQ} and let $P_{E, \frak{Q}} = \oplus_{i = 1}^{|E|} P_{E, \frak{Q}, i}$. The arguments 
given above apply with the appropriate substitutions to give the following:
\begin{corollary} \label{AlgEQ} 
\[End_{S\otimes_{S_E}S_{\frak{Q}}}(P_{E, \frak{Q}})^{op} = \mathcal{A}_{S\otimes_{S_E}S_{\frak{Q}}}(P_{E, \frak{Q}})\cong \mathcal{R}_{S\otimes_{S_E}S_{\frak{Q}}}(P_{E, \frak{Q}})/\mathcal{I}_{S\otimes_{S_E}S_{\frak{Q}}}(P_{E, \frak{Q}}).\]
\end{corollary}

\section{Semisimplicity of $\mathcal{A}_{S\otimes_RS}(P) \otimes_SF_{\frak{Q}}$}

Let $L, K, S, R$, $\frak{Q}, \frak{P}$, $G = \{\sigma_{g_1}, \sigma_{g_2}, \cdots , \sigma_{g_3}\}$,    $E = E(\frak{Q}|\frak{P})$, $\phi$ and $\phi_{g_k}$  be as defined in Section 2. 
Let $\Omega$ be the poset $\{g_1,g_ 2, \cdots , g_n\}$ with the   total ordering: $g_1 < g_2 < \cdots < g_n$.
 Let $P_i$ be as defined in Definition \ref{P} and let $P = P_1 \oplus P_2 \oplus \cdots \oplus P_n$. We have that $P$ is a projective generator for the category $\mathcal{C}_{(S\otimes_RS, \Omega, G)}$.  In the previous section, we examined the structure of the algebra $\mathcal{A}_{S\otimes_RS}(P) = End_{S\otimes_RS}(P)^{op}$.  We let $F_{\frak{Q}}$ denote the 
 residue class field   $S/\frak{Q}$.  In this section we will consider the structure of the algebra $\mathcal{A}_{S\otimes_RS}(P) \otimes_SF_\frak{Q}$
 over the field $S/\frak{Q}$. 
Since the center of $\mathcal{A}_{S\otimes_RS}(P)$  contains $S\otimes_RS$, we have $1\otimes F_{\frak{Q}}$ is contained in the center of $\mathcal{A}_{S\otimes_RS}(P) \otimes_SF_\frak{Q}$ and $\mathcal{A}_{S\otimes_RS}(P)\otimes_SF_\frak{Q}$ is central over $F_\frak{Q}$. 

\begin{theorem}   The algebra $\mathcal{A}_{S\otimes_RS}(P) \otimes_S F_\frak{Q}$ is semisimple if and only if $\frak{Q}$ is unramified in $S$. 
\vskip .1in
\noindent
Note: Semisimplicity of  $\mathcal{A}_{S\otimes_RS}(P) \otimes_S F_\frak{Q}$ is independent of the choice of projective generator $P$, by Lemma \ref{Newprogen}  and the discussion prior to it. 
\end{theorem}
\begin{proof} If $\frak{Q}$ is unramified in $S$, then $E = E(\frak{Q}|\frak{P})$ is the trivial subgroup of $G$, see Marcus 
\cite{Marcus}[Theorem 28, Chapter 4]. 
In Lemma \ref{unram}, we saw that $S\otimes_RS_{\frak{Q}}$ is isomorphic to $S_{{g_1}, \frak{Q}} \oplus S_{{g_2}, \frak{Q}} \oplus \dots \oplus S_{{g_n}, \frak{Q}}$ via the map $\phi$, where each $S_{{g_k}, \frak{Q}}$ is a copy of $S_{\frak{Q}}$, 
$1 \leq k \leq n$.  
Since localization is flat, we have $\mathcal{A}_{S\otimes_RS}(P)\otimes_S S_{\frak{Q}}  \cong (\mathcal{R}_{S\otimes_RS}(P)\otimes_SS_{\frak{Q}})/(\mathcal{I}_{S\otimes_RS}(P) \otimes_SS_{\frak{Q}})$. Using the fact that $((S\otimes_RS)_I )_{\frak{Q}} = (S\otimes_RS_{\frak{Q}} )_I$ for any subset $I$ of $G$, it is easy to see that
$\frac{\mathcal{R}_{S\otimes_RS}(P)\otimes_SS_{\frak{Q}}}{\mathcal{I}_{S\otimes_RS}(P)\otimes_SS_{\frak{Q}}} = $
$$\frac{
\begin{pmatrix}
S\otimes_RS_{\frak{Q} }& S\otimes_RS_{\frak{Q}} & S\otimes_RS_{\frak{Q}} &\cdots & S\otimes_RS_{\frak{Q}}   \\
(S\otimes_RS_{\frak{Q}})_{I_{g_2}} & S\otimes_RS_{\frak{Q}} & S\otimes_RS_{\frak{Q}} &\cdots & S\otimes_RS_{\frak{Q}} \\
(S\otimes_RS_{\frak{Q}})_{I_{g_3}} & (S\otimes_RS_{\frak{Q}})_{I_{g_3} \cup I_{g_2}^c} & S\otimes_RS_{\frak{Q}} & \cdots & S\otimes_RS_{\frak{Q}} \\
(S\otimes_RS_{\frak{Q}})_{I_{g_4}} & (S\otimes_RS_{\frak{Q}})_{I_{g_4} \cup I_{g_2}^c} & (S\otimes_RS_{\frak{Q}})_{I_{g_4}\cup I_{g_3}^c} &\cdots &S\otimes_RS_{\frak{Q}} \\
\vdots & \vdots & \vdots & \ & \\
(S\otimes_RS_{\frak{Q}})_{I_{g_n}}  & (S\otimes_RS_{\frak{Q}})_{I_{g_n} \cup I_{g_2}^c}  &  (S\otimes_RS_{\frak{Q}})_{I_{g_n} \cup I_{g_3}^c}  &\ & S\otimes_RS_{\frak{Q}}  
\end{pmatrix}}
{\begin{pmatrix}
0& (S\otimes_RS_{\frak{Q}})_{I_{g_2}^c}   & (S\otimes_RS_{\frak{Q}})_{I_{g_3}^c}  &\cdots & (S\otimes_RS_{\frak{Q}})_{I_{g_n}^c}  \\
0& (S\otimes_RS_{\frak{Q}})_{I_{g_2}^c}  & (S\otimes_RS_{\frak{Q}})_{I_{g_3}^c}  &\cdots& (S\otimes_RS_{\frak{Q}})_{I_{g_n}^c} \\
0& (S\otimes_RS_{\frak{Q}})_{I_{g_2}^c}  & (S\otimes_RS_{\frak{Q}})_{I_{g_3}^c}  & \cdots& (S\otimes_RS_{\frak{Q}})_{I_{g_n}^c} \\
0& (S\otimes_RS_{\frak{Q}})_{I_{g_2}^c}  & (S\otimes_RS_{\frak{Q}})_{I_{g_3}^c}   &\cdots& (S\otimes_RS_{\frak{Q}})_{I_{g_n}^c} \\
\vdots & \vdots & \vdots & \ & \\
0& (S\otimes_RS_{\frak{Q}})_{I_{g_2}^c}  & (S\otimes_RS_{\frak{Q}})_{I_{g_3}^c}   &\cdots& (S\otimes_RS_{\frak{Q}})_{I_{g_n}^c} \\
\end{pmatrix}}
$$
By applying the isomorphism $\phi$, we get that this is isomorphic to the quotient of matrix rings below, where 
$\frak{S}_\frak{Q} = S_{{g_1}, \frak{Q}} \oplus S_{{g_2}, \frak{Q}} \oplus \dots \oplus S_{{g_n}, \frak{Q}} $, the image of 
$S\otimes_RS_{\frak{Q}}$ under the map $\phi$. 
$$\ \frac{
\begin{pmatrix}
\frak{S}_\frak{Q} & \frak{S}_\frak{Q} &\cdots  &\frak{S}_\frak{Q}\\
 S_{{g_2}, \frak{Q}} \oplus \dots \oplus S_{{g_n}, \frak{Q}}& \frak{S}_\frak{Q} & \cdots &\frak{S}_\frak{Q}\\
 S_{{g_3}, \frak{Q}} \oplus \dots \oplus S_{{g_n}, \frak{Q}}& S_{{g_1}, \frak{Q}} \oplus S_{{g_3}, \frak{Q}} \oplus \dots \oplus S_{{g_n}, \frak{Q}}& \cdots &\frak{S}_\frak{Q}\\
 S_{{g_4}, \frak{Q}} \oplus \dots \oplus S_{{g_n}, \frak{Q}}&  S_{{g_1}, \frak{Q}} \oplus S_{{g_4}, \frak{Q}} \oplus \dots \oplus S_{{g_n}, \frak{Q}}&  \cdots &\frak{S}_\frak{Q}\\
\vdots & \vdots &  & \\
S_{{g_n}, \frak{Q}}& S_{{g_1}, \frak{Q}} \oplus S_{{g_n}, \frak{Q}}&\cdots &\frak{S}_\frak{Q}
\end{pmatrix}}
{\begin{pmatrix}
0&S_{{g_1}. \frak{Q}}   & S_{{g_1}, \frak{Q}} \oplus S_{{g_2}, \frak{Q}}   & \cdots &\frak{S}_\frak{Q} \\
0&S_{{g_1}. \frak{Q}}   & S_{{g_1}, \frak{Q}} \oplus S_{{g_2}, \frak{Q}}  & \cdots &\frak{S}_\frak{Q} \\
0&S_{{g_1}. \frak{Q}}   & S_{{g_1}, \frak{Q}} \oplus S_{{g_2}, \frak{Q}}   & \cdots &\frak{S}_\frak{Q} \\
0&S_{{g_1}. \frak{Q}}   & S_{{g_1}, \frak{Q}} \oplus S_{{g_2}, \frak{Q}}    & \cdots  &\frak{S}_\frak{Q} \\
\vdots & \vdots & \vdots & \ & \\
0& S_{{g_1}. \frak{Q}}    &S_{{g_1}, \frak{Q}} \oplus S_{{g_2}, \frak{Q}}   &\cdots  &\frak{S}_\frak{Q} 
\end{pmatrix}}
$$
\vskip .1in
\noindent
This ring is isomorphic to the subring of $M_{n \times n}(\frak{S}_\frak{Q})$ given below
\vskip .1in
\noindent
$$
\begin{pmatrix}
\frak{S}_\frak{Q}&  S_{{g_2}, \frak{Q}} \oplus \dots \oplus S_{{g_n}, \frak{Q}} &  S_{{g_3}, \frak{Q}} \oplus \dots \oplus S_{{g_n}, \frak{Q}} &  \dots & S_{{g_n}, \frak{Q}} \\
 S_{{g_2}, \frak{Q}} \oplus \dots \oplus S_{{g_n}, \frak{Q}}&  S_{{g_2}, \frak{Q}} \oplus \dots \oplus S_{{g_n}, \frak{Q}}&  S_{{g_3}, \frak{Q}} \oplus \dots \oplus S_{{g_n}, \frak{Q}}& \cdots &S_{{g_n}, \frak{Q}} \\
 S_{{g_3}, \frak{Q}} \oplus \dots \oplus S_{{g_n}, \frak{Q}}&  S_{{g_3}, \frak{Q}} \oplus \dots \oplus S_{{g_n}, \frak{Q}}&  S_{{g_3}, \frak{Q}} \oplus \dots \oplus S_{{g_n}, \frak{Q}} & \cdots &S_{{g_n}, \frak{Q}}\\
 S_{{g_4}, \frak{Q}} \oplus \dots \oplus S_{{g_n}, \frak{Q}}&   S_{{g_4}, \frak{Q}} \oplus \dots \oplus S_{{g_n}, \frak{Q}}& S_{{g_4}, \frak{Q}} \oplus \dots \oplus S_{{g_n}, \frak{Q}}& \cdots &S_{{g_n}, \frak{Q}}\\
\vdots & \vdots & \vdots & \ & \\
S_{{g_n}, \frak{Q}}&S_{{g_n}, \frak{Q}} &S_{{g_n}, \frak{Q}} &\cdots &S_{{g_n}, \frak{Q}}
\end{pmatrix}
$$
Since $S_{\frak{Q}}/S_{\frak{Q}}\frak{Q}  \cong S/{\frak{Q}}$, we will also denote $S_{\frak{Q}}/S_{\frak{Q}}\frak{Q} $ by $F_{\frak{Q}}$. 
 $\mathcal{A}_{S\otimes_RS}(P)  \otimes _SS/{\frak{Q}} =\mathcal{A}_{S\otimes_RS}(P)  \otimes _SS_{\frak{Q}}
 \otimes_{S_{\frak{Q}}}S_{\frak{Q}}/S_{\frak{Q}}\frak{Q} = \mathcal{A}_{S\otimes_RS}(P)  \otimes _SS_{\frak{Q}}
 \otimes_{S_{\frak{Q}}}F_{\frak{Q}}$.  Thus it is easy to see that if $\frak{Q}$ is unramified in $S$, we have 
 $$\mathcal{A}_{S\otimes_RS}(P) \otimes_SF_{\frak{Q}} \cong M_1((F_{\frak{Q}})_{g_1}) \oplus M_2((F_{\frak{Q}})_{g_2}) \oplus \dots \oplus 
 M_n((F_{\frak{Q}})_{g_n}) $$
 and is therefore semisimple.

 It remains to consider the case where   $\frak{Q}$ is ramified in $S$. In this case,  we will show that $\mathcal{A}_{S\otimes_RS}(P) \otimes_SF_{\frak{Q}}$ has a non-zero nilpotent ideal. Let $|E|  > 1$, and let $m = |G|/|E| = n/|E|$.  Let $\{\sigma_{x_1}, \sigma_{x_2}, \dots , \sigma_{x_{m}}\}$ be a set of coset representatives for $E$ in $G$, where $\sigma_{x_1}$ is  the identity element of $G$.   By Corollary \ref{corstar}, there exists a set of 
 orthogonal idempotents $\{x_1, x_2, \dots , x_{m}\}$  in $S\otimes_RS_{\frak{Q}}$ such that 
\[S\otimes_RS_{\frak{Q}} = (S\otimes_RS_{\frak{Q}})x_1 \oplus (S\otimes_RS_{\frak{Q}})x_2 \oplus  \dots \oplus (S\otimes_RS_{\frak{Q}})x_{m}\]
and
  $$\phi_{g_k}(x_i)  =  \Big\{\  \begin{matrix}
& 1 & \hbox{if}  & \sigma_{{g_k}} \in E\sigma_{x_i} \\
& 0&\mbox{otherwise}  &\   &
\end{matrix}
$$
From Lemma \ref{isomap}, we  have an  isoomorphism of rings  $\gamma : S\otimes_{S_E}S_{\frak{Q}} \to (S\otimes_{R}S_{\frak{Q}})x_1$ with $\gamma(s_1 \otimes s_2) = (s_1 \otimes s_2)x_1$, which extends to an isomorphism of vector spaces over $F_{\frak{Q}}$;
 $\gamma\otimes 1: S\otimes_{S_E}S_{\frak{Q}} \otimes_{S_{\frak{Q}}}F_{\frak{Q}}  \to
 (S\otimes_{S}S_{\frak{Q}})x_1 \otimes_{S_{\frak{Q}}}F_{\frak{Q}}$.
Also from   Lemma \ref{Pie} and Lemma \ref{structure}, we have 
that $S\otimes_{S_E}S_{\frak{Q}} =  S_E[\Pi]\otimes_{S_E}S_{\frak{Q}}$ for some $\Pi \in \frak{Q} \subset S$. 
Letting  $E = \{\sigma_{e_1}, \sigma_{e_2}, \dots , \sigma_{e_{|E|}}\}$ and   $A_{e_i}(\Pi) = \Pi \otimes 1 - 1\otimes \sigma_{e_i}(\Pi) \in S\otimes_{S_E}S_{\frak{Q}}, \  i = 1, 2, \dots , |E|$,  we  have $\phi_{e_i}(A_{e_i}(\Pi)) = 0, i = 1, 2, \dots , |E|$.  By Lemma \ref{basisP}, $\{1, A_{e_1}(\Pi), \prod_1^2A_{e_i}(\Pi), \cdots ,
\prod_1^{n-1}A_{e_i}(\Pi) \}$ is a basis for $S\otimes_{S_E}S_{\frak{Q}} $ as a right $S_{\frak{Q}}$ module.

We define  $\Psi: S\otimes_RS_{\frak{Q}} \to S\otimes_RS_{\frak{Q}}\otimes_{S_{\frak{Q}}}F_{\frak{Q}}$, by 
$\Psi(s_1\otimes s_2) = s_1\otimes s_2 \otimes 1$. Let  $\Psi':S\otimes_{S_E}S_{\frak{Q}} \to S\otimes_{S_E}S_{\frak{Q}}\otimes_{S_{\frak{Q}}}F_{\frak{Q}}$ be the map introduced in 
Lemma \ref{FQ}.  It is not difficult to see that 
\[\Psi(\gamma(s_1\otimes s_2)) = (\gamma \otimes 1)(\Psi'(s_1\otimes s_2)) , \ \ \ \ \mbox{for} \ \ \ \  s_1\otimes s_2 \in S\otimes_{S_E}S_{\frak{Q}}.\]
From Lemma \ref{FQ}, we have \ $\Psi'(A_{e_i}(\Pi)) = \Psi'(A_{e_j}(\Pi))$ for $1 \leq i, j \leq |E|$ \ and 
if $\alpha = \Psi'(A_{e_j}(\Pi))$, we have $\alpha^{|E|} = 0$.  Hence, if we let $\bar{\alpha} = (\gamma \otimes 1)\Psi'(A_{e_i}(\Pi)) = \Psi(\gamma (A_{e_i}(\Pi)))$ we must have $\bar{\alpha}^{|E|} = 0$.
Also since $\{1, \alpha, \alpha^2, \cdots , \alpha^{n - 1}\}$ is a basis for $S\otimes_{S_E}S_{\frak{Q}} \otimes F_{\frak{Q}}$ as a vector space over $F_{\frak{Q}}$, we have $\{1, \bar{\alpha}, \bar{\alpha}^2, \cdots , \bar{\alpha}^{n - 1}\}$ is a basis for $(S\otimes_{R}S_{\frak{Q}})x_1  \otimes F_{\frak{Q}}$ and $\bar{\alpha} \not= 0$. 

We are now ready to construct a non-zero, nilpotent ideal  of the algebra.  It is not difficult to see that $\mathcal{I}_{S\otimes_RS}(P), \mathcal{R}_{S\otimes_RS}(P)$ and $\mathcal{A}_{S\otimes_RS}(P)$ are finitely generated and  projective as right $S$-modules. Hence the short exact sequence 
$$0 \to \mathcal{I}_{S\otimes_RS}(P) \to \mathcal{R}_{S\otimes_RS}(P) \to \mathcal{A}_{S\otimes_RS}(P) \to 0$$
is a split exact sequence of right $S$ modules. Thus, applying the functor $ \otimes_SF_\frak{Q}$, we get 
$\mathcal{A}_{S\otimes_RS}(P) \otimes_SF_\frak{Q} \cong $
$$  \frac{
\begin{pmatrix}
\Psi(S\otimes_RS_{\frak{Q} })& \Psi(S\otimes_RS_{\frak{Q}}) & \cdots & \Psi(S\otimes_RS_{\frak{Q}} ) \\
\Psi(S\otimes_RS_{\frak{Q}})_{I_{g_2}} )& \Psi(S\otimes_RS_{\frak{Q}} )&  \cdots  &  \Psi(S\otimes_RS_{\frak{Q}} )\\
\Psi((S\otimes_RS_{\frak{Q}})_{I_{g_3}} )& \Psi((S\otimes_RS_{\frak{Q}})_{I_{g_3} \cup I_{g_2}^c} )&  \cdots  &  \Psi(S\otimes_RS_{\frak{Q}} )\\
\Psi((S\otimes_RS_{\frak{Q}})_{I_{g_4}}) &\Psi( (S\otimes_RS_{\frak{Q}})_{I_{g_4} \cup I_{g_2}^c} )&  \cdots & \Psi(S\otimes_RS_{\frak{Q}} ) \\
\vdots & \vdots & &\vdots  \\
\Psi((S\otimes_RS_{\frak{Q}})_{I_{g_n}}) &\Psi( (S\otimes_RS_{\frak{Q}})_{I_{g_n} \cup I_{g_2}^c} )& \cdots  & \Psi(S\otimes_RS_{\frak{Q}} )\\
\end{pmatrix}}
{\begin{pmatrix}
0& \Psi((S\otimes_RS_{\frak{Q}})_{I_{g_2}^c} )  & \Psi((S\otimes_RS_{\frak{Q}})_{I_{g_3}^c}  )& \cdots  &\Psi( (S\otimes_RS_{\frak{Q}})_{I_{g_n}^c} ) \\
0& \Psi((S\otimes_RS_{\frak{Q}})_{I_{g_2}^c})  &\Psi( (S\otimes_RS_{\frak{Q}})_{I_{g_3}^c} ) & \cdots & \Psi( (S\otimes_RS_{\frak{Q}})_{I_{g_n}^c} )\\
0&\Psi( (S\otimes_RS_{\frak{Q}})_{I_{g_2}^c} ) &\Psi( (S\otimes_RS_{\frak{Q}})_{I_{g_3}^c}  )& \cdots & \Psi( (S\otimes_RS_{\frak{Q}})_{I_{g_n}^c} )\\
0&\Psi( (S\otimes_RS_{\frak{Q}})_{I_{g_2}^c} ) & \Psi((S\otimes_RS_{\frak{Q}})_{I_{g_3}^c}  ) & \cdots  &
\Psi( (S\otimes_RS_{\frak{Q}})_{I_{g_n}^c} )\\
\vdots & \vdots & \vdots & \ & \\
0&\Psi( (S\otimes_RS_{\frak{Q}})_{I_{g_2}^c} ) & \Psi((S\otimes_RS_{\frak{Q}})_{I_{g_3}^c}  ) & \cdots  &
\Psi( (S\otimes_RS_{\frak{Q}})_{I_{g_n}^c} )\\
\end{pmatrix}}
$$
Consider the subset of $\mathcal{A}_{S\otimes_RS}(P)\otimes_SS_{\frak{Q}}\otimes_{S_{\frak{Q}}}S_{\frak{Q}}/\frak{Q}S_{\frak{Q}} $ given by 
$\mathcal{B} =$
$$  \frac{
{  \begin{pmatrix}
\bar{\alpha}\Psi((S\otimes_RS_{\frak{Q} })x_1)&0&0& \cdots &0 \\
\bar{\alpha}\Psi((S\otimes_RS_{\frak{Q}})_{I_{g_2}} )x_1)&0&0& \cdots &0\\
\bar{\alpha}\Psi(((S\otimes_RS_{\frak{Q}})_{I_{g_3}})x_1 )& 0& 0& \cdots &0\\
\bar{\alpha}\Psi(((S\otimes_RS_{\frak{Q}})_{I_{g_4}})x_1) &0&0& \cdots  &0\\
\vdots & \vdots & \vdots & \ &\vdots \\
\bar{\alpha}\Psi(((S\otimes_RS_{\frak{Q}})_{I_{g_n}})x_1)&\ 0&\ 0&\ &0
\end{pmatrix} } + {( \mathcal{I}_{S\otimes_RS}(P) \otimes_SF_\frak{Q} })}
{(\mathcal{I}_{S\otimes_RS}(P) \otimes_SF_\frak{Q} )}
$$
where $\mathcal{I}_{S\otimes_RS}(P) \otimes_SF_\frak{Q} $ = 
$$ \begin{pmatrix}
0& \Psi((S\otimes_RS_{\frak{Q}})_{I_{g_2}^c} )  & \Psi((S\otimes_RS_{\frak{Q}})_{I_{g_3}^c}  )& \cdots  &\Psi( (S\otimes_RS_{\frak{Q}})_{I_{g_n}^c} ) \\
0& \Psi((S\otimes_RS_{\frak{Q}})_{I_{g_2}^c})  &\Psi( (S\otimes_RS_{\frak{Q}})_{I_{g_3}^c} ) & \cdots & \Psi( (S\otimes_RS_{\frak{Q}})_{I_{g_n}^c} )\\
0&\Psi( (S\otimes_RS_{\frak{Q}})_{I_{g_2}^c} ) &\Psi( (S\otimes_RS_{\frak{Q}})_{I_{g_3}^c}  )& \cdots & \Psi( (S\otimes_RS_{\frak{Q}})_{I_{g_n}^c} )\\
0&\Psi( (S\otimes_RS_{\frak{Q}})_{I_{g_2}^c} ) & \Psi((S\otimes_RS_{\frak{Q}})_{I_{g_3}^c}  ) & \cdots  &
\Psi( (S\otimes_RS_{\frak{Q}})_{I_{g_n}^c} )\\
\vdots & \vdots & \vdots & \ & \\
0&\Psi( (S\otimes_RS_{\frak{Q}})_{I_{g_2}^c} ) & \Psi((S\otimes_RS_{\frak{Q}})_{I_{g_3}^c}  ) & \cdots  &
\Psi( (S\otimes_RS_{\frak{Q}})_{I_{g_n}^c} )\\
\end{pmatrix}. $$
It is clear that $\mathcal{B}$ is a non zero  ideal of $\mathcal{A}_{S\otimes_RS}(P) \otimes_SS_{\frak{Q}}\otimes_{S_{\frak{Q}}}F_{\frak{Q}} $.  First we show it is nilpotent. Each   product of $|E|$ matrices from the ideal  $\mathcal{B}$ has a representative matrix of  the form: 
\[
\begin{pmatrix}
a_1(1) & 0 & 0 & \cdots & 0 \\
a_2(1) & 0 & 0 & \cdots  &0 \\
a_3(1)& 0 & 0 & \cdots & 0 \\
\vdots & \vdots & \vdots & \ & \\
a_n(1)&0 & 0 & \cdots & 0 
\end{pmatrix}
\begin{pmatrix}
a_1(2) & 0 & 0 & \cdots & 0 \\
a_2(2) & 0 & 0 & \cdots  &0 \\
a_3(2)& 0 & 0 & \cdots & 0 \\
\vdots & \vdots & \vdots & \ & \\
a_n(2)&0 & 0 & \cdots & 0 
\end{pmatrix}
\ \ \ 
\cdots  \ \ \ 
\begin{pmatrix}
a_1({|E|}) & 0 & 0 & \cdots & 0 \\
a_2({|E|}) & 0 & 0 & \cdots  &0 \\
a_3({|E|})& 0 & 0 & \cdots & 0 \\
\vdots & \vdots & \vdots & \ & \\
a_n({|E|}) &0 & 0 & \cdots & 0 
\end{pmatrix}
\]

\[ = 
\begin{pmatrix}
a_1(1)a_1(2)a_1(3) \cdots a_1({|E|}) & 0 & 0 & \cdots & 0 \\
a_2(1) a_1(2)a_1(3) \cdots a_1({|E|})  & 0 & 0 & \cdots  &0 \\
a_3(1)a_1(2)a_1(3) \cdots a_1({|E|}) & 0 & 0 & \cdots & 0 \\
\vdots & \vdots & \vdots & \ & \\
a_n(1)a_1(2)a_1(3) \cdots  a_1({|E|})  &0 & 0 & \cdots & 0 
\end{pmatrix}
\]
Now we see that $0 =\bar{\alpha}^{|E|}$ divides each product of $|E|$ elements of the form $a_i(j)$, $1\leq i \leq n$, $1\leq j \leq |E|$. 
Hence $\mathcal{B}^{|E|} = 0$ and $\mathcal{B}$ is a non zero nilpotent ideal of 
 $\mathcal{A}_{S\otimes_RS}(P) \otimes_SS_{\frak{Q}}\otimes_{S_{\frak{Q}}}F_{\frak{Q}} $.
Therefore  $\mathcal{A}_{S\otimes_RS}(P) \otimes_SS_{\frak{Q}}\otimes_{S_{\frak{Q}}}F_{\frak{Q}} $ has a non zero radical and is not a semisimple algebra over $F_{\frak{Q}}$. 
\end{proof}

\section{Dualiy}
Let $L, K, S, R, \frak{Q}, \frak{P}, \phi$ and $\phi_{g_k}$  be as defined in Section 2.  Let  $G$ be  the Galois group of $L$ over $K$ such that $G = \{\sigma_{g_1}, \sigma_{g_2}, \cdots , \sigma_{g_n}\}$. Let $\Omega = \{g_i |g_i \in \Omega\}$ be a poset giving the  total ordering,  $g_1 < g_2 < \dots < g_n$ on the indices of $G$.  If $Y = \{\sigma_{y_1}, \sigma_{y_2}, \dots , \sigma_{y_k}\}$ is any subset of $G$, we let $\Omega|_Y$ denote the restriction of the ordering on $\Omega$ to the indices of the elements of  $Y$.  

Let $E = E(\frak{Q}|\frak{P}) = \{\sigma_{e_1}, \sigma_{e_2}, \cdots , \sigma_{e_{|E|}}\}$ denote the inertia group with respect to 
$\frak{Q}$.  
Let $\{\sigma_{x_1}, \sigma_{x_2}, \dots , \sigma_{x_m}\} $ be a set of right coset representatives for $E$ in $G$, where $m = n/|E|$. . 
Recall from Corollary \ref{corstar}  that there exist orthogonal idempotents $\{x_1, x_2, \cdots , x_m\}$ in $S\otimes_RS_{\frak{Q}}$
such that
$$S\otimes_RS_{\frak{Q}} = (S\otimes_RS_{\frak{Q}})x_1 \oplus (S\otimes_RS_{\frak{Q}})x_2 \oplus \dots \oplus
(S\otimes_RS_{\frak{Q}})x_m$$
where
 $$\phi_{g_t}(x_j) =  \Big\{\  \begin{matrix}
& 1 & \hbox{if}  & \sigma_{g_t} \in E\sigma_{x_j} \\
& 0 &  & \mbox{otherwise} &
\end{matrix}
$$ 
We let $\{x_1, x_2, \cdots , x_m\}$ denote such a set of idempotents in $S\otimes_RS_{\frak{Q}}$,
throughout this section.

Let $T$ be a commutative ring and let $A$ and $U$ be Noetherian commutative $T$-algebras. Let
 $H = \{\sigma_{h_i} \}_{i = 1}^{i = t}$ be a family of pairwise distinct $T$-algebra homomorphisms $\sigma_{h_i} : U \to A$,  indexed by
  $\Lambda = \{h_i\}_{i = 1}^{i = t}$, a  poset with  a total ordering $h_1 < h_2 < \cdots <  h_t$.  
From Definition \ref{Cat} we have a stratified  category  $\mathcal{C}_{(U\otimes_TA, \Lambda, H)}$.  In this section, we will deal with such categories defined by a variety of rings, groups and orderings.

We will show  that there is a Duality on the projective modules  in $\mathcal{C}_{(S\otimes_RS,  \Omega,  G)}$, namely the contravariant  functor $Hom_{S\otimes_RS}( - , 
S\otimes_RS)$ maps projectives in $\mathcal{C}_{(S\otimes_RS,  \Omega,  G)}$ to projectives in $\mathcal{C}_{(S\otimes_RS,  \Omega, G)}$.    First we verify that 
$Hom_{S\otimes_RS}( - , S\otimes_RS)$ does in fact map projective modules in 
$\mathcal{C}_{(S\otimes_RS,  \Omega,  G)}$ to modules in $\mathcal{C}_{(S\otimes_RS,  \Omega,  G)}$. 

\begin{lemma}  Let $S, R$ and $\frak{Q}$ be as defined in Section 2 and let $Q$ be a projective module in $\mathcal{C}_{(S\otimes_RS,  \Omega, G)}$.  Then $Hom_{S\otimes_RS}( Q, S\otimes_RS)$
is an $S\otimes_RS$ module in $\mathcal{C}_{(S\otimes_RS,  \Omega, G)}$.
\end{lemma}
\begin{proof} Let $M =  Hom_{S\otimes_RS}( Q, S\otimes_RS)$. M is an $S\otimes_RS$ module 
since $S\otimes_RS$ is commutative.  Letting $M^{g_i} = Hom_{S\otimes_RS}(Q, (S\otimes_RS)^{g_i})$,  $1 \leq i \leq n$, 
I claim that 
$$M = M^{g_0} \supseteq M^{g_1} \supseteq M^{g_2} \supseteq \cdots \supseteq M^{g_n} = {0}$$
is a filtration for $M$ with the required properties to ensure that $M \in \mathcal{C}_{(S\otimes_RS,  \Omega,  G)}$. 
Consider the quotient $M^{g_i}/M^{g_{i + 1}} = Hom(Q, (S\otimes_RS)^{g_i})/Hom(Q, (S\otimes_RS)^{g_{i + 1}})$. 
This is isomorphic to $M^{(g_i)}  = Hom(Q, (S\otimes_RS)^{g_i}/(S\otimes_RS)^{g_{i + 1}}) $, since $Q$ is projective in
$\mathcal{C}_{(S\otimes_RS,  \Omega, G)}$ and the sequence
\[0 \to (S\otimes_RS)^{g_{i + 1}} \to (S\otimes_RS)^{g_{i }} \to  (S\otimes_RS)^{g_i}/(S\otimes_RS)^{g_{i + 1}} \to 0,\]
is sheaf exact in the category $\mathcal{C}_{(S\otimes_RS,  \Omega, G)}$.

To see that the action of $S\otimes_RS$ on  $M^{(g_i)}$, is compatible with the 
requirements for a filtration, we must show that $(s_1 \otimes s_2)f = (1 \otimes \sigma_{g_i}(s_1)s_2)f$ for 
$f \in Hom(Q, (S\otimes_RS)^{g_i}/(S\otimes_RS)^{g_{i + 1}}) $. 
If  $q \in Q$, we have $((s_1 \otimes s_2)f)(q) = (s_1 \otimes s_2)(f(q)) = (1 \otimes \sigma_{g_i}(s_1)s_2)f(q)   = ((1\otimes \sigma_{g_i}(s_1)s_2)f)(q)$,  since 
$f(q) \in  (S\otimes_RS)^{g_i}/(S\otimes_RS)^{g_{i + 1}} $. Hence $S\otimes_RS$  acts appropriately on the quotients of our filtration for $M$. It remains to show that the quotients are finitely generated and projective as right $S$ modules. 

Since $(S\otimes_RS)^{g_i}/(S\otimes_RS)^{g_{i + 1}}$ is finitely generated and projective as a right $S$ module, and $Q$ is also finitely generated and projective as a right $S$ module, we have that $M^{(g_i)}$ is a submodule   of a finitely generated 
free right $S$ module and hence is finitely generated and projective as a right $S$-module, since $S$ is a Dedekind domain.

\end{proof}

It is not difficult to see that the proof also applies to projectives in $\mathcal{C}_{(S\otimes_RS_{\frak{Q}},  \Omega, G)}$.  

\begin{corollary}
Let $S, R$ and $\frak{Q}$ be as defined in Section 2 and let $Q$ be a projective module in $\mathcal{C}_{(S\otimes_RS_{\frak{Q}},  \Omega, G)}$.  Then $Hom_{S\otimes_RS_{\frak{Q}}}( Q, S\otimes_RS_{\frak{Q}})$
is an $S\otimes_RS_{\frak{Q}}$ module in $\mathcal{C}_{(S\otimes_RS_{\frak{Q}},  \Omega, G)}$. 
\end{corollary}

We are left with the task of showing that $Hom_{S\otimes_RS}( Q, S\otimes_RS)$ is projective in 
$\mathcal{C}_{(S\otimes_RS,  \Omega, G)}$ when $Q$ is projective in $\mathcal{C}_{(S\otimes_RS,  \Omega, G)}$. We will use a local-global argument. 
We first   prove a series of  lemmas that allow us to reduce the local case to consideration of the category
$\mathcal{C}_{(S\otimes_{S_E}S_{\frak{Q}}, \Omega_1, E)}$, where $\Omega_1$ is a poset giving an ordering on the indices of $E$.

\begin{lemma} \label{mapsxi} Let $S, R, \frak{Q}$ be as defined in section 2,  and $x_1, x_2, \dots , x_m$ be the orthogonal idempotents in $S\otimes_RS_{\frak{Q}}$  defined above. 
If $M$ and $N$ are $S\otimes_RS_{\frak{Q}}$ modules, then 

$$Hom_{S\otimes_RS_{\frak{Q}}}(x_jM, x_kN) = 0, \ \hbox{if} \ j \not= k.$$ 

and $$Hom_{S\otimes_RS_{\frak{Q}}}(M, N)  = \oplus_{i = 1}^mHom_{S\otimes_RS_{\frak{Q}}}(x_iM, x_iN). $$
Also $Hom_{S\otimes_RS_{\frak{Q}}}(x_iM, x_iN)  = Hom_{(S\otimes_RS_{\frak{Q}})x_i}(x_iM, x_iM) $. 
\end{lemma}

\begin{proof}   
Let $f \in Hom_{S\otimes_RS_{\frak{Q}}}(x_jM, x_kN) $ for $1 \leq i, j \leq n$.  Let    $m \in M$ with $f(x_jm) = x_kn$ for some $n \in N$.  Then $f(x_jm) = f(x_jx_jm) = x_jx_kn = 0$ if $j \not= k$. 
Hence $Hom_{S\otimes_RS_{\frak{Q}}}(x_jM, x_kN) = 0, \ \hbox{if} \ j \not= k$ and 
$Hom_{S\otimes_RS_{\frak{Q}}}(M, N)  = \oplus_{i = 1}^mHom_{S\otimes_RS_{\frak{Q}}}(x_iM, x_iN)$.  

If  $g \in Hom_{S\otimes_RS_{\frak{Q}}}(x_iM, x_iM)$, then  we have 
 $g \in 
  Hom_{(S\otimes_RS_{\frak{Q}})x_i}(x_iM, x_iM) $ automatically. On the other hand if 
 we start with  
  $g \in Hom_{(S\otimes_RS_{\frak{Q}})x_i}(x_iM, x_iM) )$, then if 
 $x \in S\otimes_RS_{\frak{Q}}$ and $x_im \in x_iM$, we have 
 $g(xx_im) = g(xx_ix_im) = xx_ig(x_im) = xg(x_im)$. Hence $g \in Hom_{S\otimes_RS_{\frak{Q}}}(x_iM, x_iM)$. 

\end{proof}

\begin{lemma} \label{subcat} Let Let $S, R, \frak{Q}, G = \{\sigma_{g_1}, \sigma_{g_2}, \dots , \sigma_{g_n}\}$ and $\Omega$  be as defined at the begining of this section.   Let $Y\subset G$, with $Y = \{\sigma_{y_1}, \sigma_{y_2}, \dots , \sigma_{y_k}\}$, where $\Omega|_Y$ gives
$y_1 < y_2 < \dots < y_k$. Then $\mathcal{C}_{(S\otimes_RS_{\frak{Q}}, \Omega|_Y,  Y)} $ is a full subcategory of 
$\mathcal{C}_{(S\otimes_RS_{\frak{Q}}, \Omega,  G)} $. If $0 \to N \to M \to P \to 0$ is a short exact sequence of $S\otimes_RS_{\frak{Q}}$-modules  in  
$\mathcal{C}_{(S\otimes_RS_{\frak{Q}}, \Omega|_Y,  Y)} $, then it is a sheaf exact sequence in 
$\mathcal{C}_{(S\otimes_RS_{\frak{Q}}, \Omega|_Y,  Y)} $ if and only if  it is sheaf exact in $\mathcal{C}_{(S\otimes_RS_{\frak{Q}}, \Omega,  G)} $. 
\end{lemma}

\begin{proof}  Let $M$ be a module in $\mathcal{C}_{(S\otimes_RS_{\frak{Q}}, \Omega|_Y, Y)} $ and let
$$M = M^{y_0} \supseteq M^{y_1} \supseteq M^{y_2} \supseteq \dots \supseteq M^{y_k} = 0 $$
be a filtration for $M$ in $\mathcal{C}_{(S\otimes_RS_{\frak{Q}}, \Omega|_Y, Y)} $.   If $g_t < y_1$, then we let $M^{g_t} = M$. 
 Otherwise for $t = 1, 2, \dots , n$, let  $M^{g_t} = M^{y_{i_t}}$, where $y_{i_t}$ is the largest element of $Y$ which is less than or equal to $g_t$. It is easy to see that this gives a suitable filtration on $M$ as a module in the category $\mathcal{C}_{(S\otimes_RS_{\frak{Q}}, \Omega, G)} $ satisfying the necessary conditions on quotients. It is also not difficult to see that if $0 \to N \to M \to P \to 0$ is a sheaf exact sequence in $\mathcal{C}_{(S\otimes_RS_{\frak{Q}}, \Omega|_Y, Y)} $, then it is sheaf exact in 
$\mathcal{C}_{(S\otimes_RS_{\frak{Q}}, \Omega, G)} $, since every subsequence $0 \to M^{g_t} \to N^{g_t} \to P^{g_t} \to 0$ is in fact a subsequence of the form  $0 \to M^{y_{i_t}} \to N^{y_{i_t}} \to P^{y_{i_t}} \to 0$ and hence is exact. 
\end{proof}

 \begin{lemma} Let $L, K, S,  R, \frak{Q}, E = E(\frak{Q} | \frak{P} )$ and $S_E$  be as defined in section 2. Let $\Omega_1$ be a poset ordering the indices of  $E = 
 \{\sigma_{e_1}, \sigma_{e_2}, \dots , \sigma_{e_{|E|}}\}$, such that $e_1 < e_2 < \dots < e_{|E|}$.  Then 
 $$\mathcal{C}_{(S\otimes_RS_{\frak{Q}}, \Omega_1, E)}  = \mathcal{C}_{(S\otimes_{S_E}S_{\frak{Q}}, \Omega_1, E)} .$$
 \end{lemma}

\begin{proof} Let $M \in \mathcal{C}_{(S\otimes_RS_{\frak{Q}},  \Omega_1,  E)}$ with filtration
$$M = M^{e_0} \supseteq M^{e_1} \supseteq M^{e_2} \supseteq \dots \supseteq M^{e_{|E|}} = 0.$$
By \cite{ Dyer}[Section 1.11] , we have a one to one  inclusion $M \to M\otimes_RL$, since $M$ is projective as an $R$ module. 
Also  we have $M\otimes_RL$ has a direct sum decomposition as an
$S\otimes_RL$ module of the form $M\otimes_RL = \oplus_{\omega \in \Omega_1}M^{\omega}_L$
such that each $M^{\omega}_L$ is an $L$ vector space and $sm = m\sigma_{\omega}(s)$ for each $s \in S$.  
Since $\sigma_{\omega}(s) = s$ for each $\sigma_{\omega} \in E$ and each $s \in S_E$,  by considering  the 
inclusion  of $M$ into $M\otimes_RL$, we see that $sm = ms$ for all $s \in S_E$. Hence we can view $M$ as an $S\otimes_{S_E}S_{\frak{Q}}$ module. We also see that the quotients of the filtration given above  have the appropriate properties to give $M$ the structure of a module in $\mathcal{C}_{(S\otimes_{S_E}S_{\frak{Q}},  \Omega_1,  E)}$.  It is also obvious that a sheaf  exact sequence in $\mathcal{C}_{(S\otimes_{R}S_{\frak{Q}},  \Omega_1,  E)}$ is also sheaf exact in $\mathcal{C}_{(S\otimes_{S_E}S_{\frak{Q}},  \Omega_1,  E)}$.  On the other hand any module in $\mathcal{C}_{(S\otimes_{S_E}S_{\frak{Q}},  \Omega_1,  E)}$ is automatically an $S\otimes_RS$ module with a filtration having the appropriate properties and any sheaf exact sequence in $\mathcal{C}_{(S\otimes_{S_E}S_{\frak{Q}},  \Omega_1,  E)}$ is also sheaf exact in $\mathcal{C}_{(S\otimes_{R}S_{\frak{Q}},  \Omega_1,  E)}$.  Hence the categories are equal. 
\end{proof}

\begin{lemma} \label{form}  Let $L, K, S,  R, \frak{Q}, G = \{\sigma_{g_1}, \sigma_{g_2}, \dots , \sigma_{g_n}\}, $  and $E = E(\frak{Q} | \frak{P})$ be as defined in  section 2. Let $\Omega$  be as defined at the beginning of this section. Let $\{\sigma_{x_1}, \sigma_{x_2}, \cdots , \sigma_{x_m}\}$ be a set of right coset representatives of $E$ in $G$ and let  $x_1, x_2, \dots , x_m$ be the corresponding  orthogonal idempotents of $S\otimes_RS_{\frak{Q}}$ defined at the beginning of this section.     Let $Z_i = E\sigma_{x_i} = \{\sigma_{e_1}\sigma_{x_i} = \sigma_{z_{i_1}}, 
 \sigma_{e_2}\sigma_{x_i} = \sigma_{z_{i_2}},  \dots ,  \sigma_{e_{|E|}}\sigma_{x_i} = \sigma_{z_{i_{|E|}}}\}$.  Let us assume that the indices of $E$ are labelled so that 
 $\Omega|_{Z_i}$ give the total ordering  $z_{i_1}  <  z_{i_2}  <  \dots  <  z_{i_{|E|}}$ on the indices of $Z_i$.  The category $\mathcal{C}_{(S\otimes_RS_{\frak{Q}}, \Omega|_{Z_i}, Z_i)}$ is a full subcategory of $\mathcal{C}_{(S\otimes_RS_{\frak{Q}}, \Omega, G)}$. 
Let $N$ be an $S\otimes_RS_{\frak{Q}}$ module in $\mathcal{C}_{(S\otimes_RS_{\frak{Q}}, \Omega, G)}$, then $N$ is in $\mathcal{C}_{(S\otimes_RS_{\frak{Q}}, \Omega|_{Z_i}, Z_i)}$ if and only if $N = (S\otimes_RS_{\frak{Q}})x_iM$ for some module, $M$, in $\mathcal{C}_{(S\otimes_RS_{\frak{Q}}, \Omega, G)}$ or equivalently if and only if $x_in = n$ for every $n \in N$. 
\end{lemma}

\begin{proof} 
That $\mathcal{C}_{(S\otimes_RS_{\frak{Q}}, \Omega|_{Z_i}, Z_i)}$ is a full subcategory of $\mathcal{C}_{(S\otimes_RS_{\frak{Q}}, \Omega, G)}$, follows from Lemma \ref{subcat}. 
Let $M$ be a module in $\mathcal{C}_{(S\otimes_RS_{\frak{Q}}, \Omega, G)}  $. Then
$M = x_1M \oplus x_2M \oplus \dots \oplus x_mM$ since $x_1, x_2, \dots , x_m$ are orthogonal idempotents in
$S\otimes_RS_{\frak{Q}}$.   Now $M_i = (S\otimes_RS_{\frak{Q}})x_iM = x_iM $ is a direct summand 
of a module in  $\mathcal{C}_{(S\otimes_RS_{\frak{Q}}, \Omega, G)}$ 
and it is not difficult to see (see  \cite{Dyer}[Lemma 2.6] ) that 
it has a filtration
$$M_i^{g_0} = M_i \supseteq M_i^{g_1} \supseteq M_i^{g_2} \supseteq \dots \supseteq M_i^{g_n} = 0.$$
I claim that $M_i^{g_j} = M_i^{g_{j+1}}$ if  $\sigma_{g_{i + 1}} \not\in Z_i$.  We have  that $x_im = m$ for all 
$m \in M_i$, since $x_i^2 = x_i$, hence it is enough to show that $x_i(M_i^{g_j}/M_i^{g_{j+1}}) = 0$, if $\sigma_{g_{j + 1}} \not\in Z_i$. Since $M_i^{g_j}/M_i^{g_{j+1}}$ is finitely generated and projective as a right $S_{\frak{Q}}$ module, it  is free as a right $S_{\frak{Q}}$ module, since $S_{\frak{Q}}$ is a local ring.  Hence $M_i^{g_j}/M_i^{g_{j+1}} \cong \oplus_l S_{\frak{Q}}[\sigma_{g_{ j + 1}}]_l$ and it is enough  to show that $x_iS_{\frak{Q}}[\sigma_{g_{j + 1}}] = 0$.
Let $x_i = \sum_k\alpha_k\otimes \beta_k \in S\otimes_RS_{\frak{Q}}$. If $s\in S_{\frak{Q}}[\sigma_{g_{j + 1}}]$, 
then $x_is = s(\sum_k\sigma_{g_{j + 1}}(\alpha_k)\beta_k) = s\phi_{g_{j + 1}}(x_i) = 0$, since $\sigma_{g_{j + 1}} \not\in Z_i$.  Hence we can view the filtration as a filtration of $M_i$ in $\mathcal{C}_{(S\otimes_RS_{\frak{Q}}, \Omega|_{Z_i}, Z_i)} $ :
$$M_i = M_i^{{z_0}} \supseteq M_i^{{z_1}} \supseteq \dots \supseteq M_i^{{z_{|E|}}} = 0.$$
The quotients inherit the required properties from the filtration in $\mathcal{C}_{(S\otimes_RS_{\frak{Q}}, \Omega, G)}$.  Thus if $M$ is an $S\otimes_RS_{\frak{Q}}$ module in $\mathcal{C}_{(S\otimes_RS_{\frak{Q}}, \Omega, G)}$,  then $(S\otimes_RS_{\frak{Q}})x_iM$ is an
$S\otimes_RS_{\frak{Q}}$ module in $M \in \mathcal{C}_{(S\otimes_RS_{\frak{P}}, \Omega|_{Z_i}, Z_i)} $.

On the other hand. given any $S\otimes_RS_{\frak{Q}}$ module,  $N$,  in $\mathcal{C}_{(S\otimes_RS_{\frak{Q}}, \Omega|_{Z_i}, Z_i)} $, $N$ is also in $\mathcal{C}_{(S\otimes_RS_{\frak{Q}}, \Omega, G)}$, since $\mathcal{C}_{(S\otimes_RS_{\frak{Q}}, \Omega|_{Z_i}, Z_i)} $ is a full subcategory of $\mathcal{C}_{(S\otimes_RS_{\frak{Q}}, \Omega, G)}$, by Lemma \ref{subcat}.
It has a filtration
$$N^{z_{i_0}} = N \supseteq N^{z_{i_1}} \supseteq N^{z_{i_2}} \supseteq \dots \supseteq N^{z_{i_n}} = 0.$$
For such a module, $x_jN = 0$ for $j \not= i$, since $x_j(N^{z_{i_k}}/N^{z_{i_{k+1}}}) = 0$ for each quotient. 
Hence $N = x_1N + x_2N + \dots + x_mN = x_iN$. 

It is easy to see that an $S\otimes_RS_{\frak{Q}}$ module, N,  of the form $N = x_iM$, where 
$M$ is an $S\otimes_RS_{\frak{Q}}$ module has the property that $x_in = n$ for all $n \in N$, since 
$x_i^2 = x_i$. Also it is obvious that if $N$ is an $S\otimes_RS_{\frak{Q} }$ module with the property that $x_in = n$ for all $n \in N$, then $N = x_iN$. 
 This finishes the proof of our Lemma.
\end{proof}

\begin{lemma} \label{equivalence}  Let $S, L, R, \frak{Q},  \frak{P}, \phi, \phi_{g_k},  G  = \{\sigma_{g_1}, \sigma_{g_2}, \dots , \sigma_{g_n}\}, $ and $E = E(\frak{Q} | \frak{P} )$ be as defined in section 2. let $\Omega$ be as defined at the beginning of this section. 
Let $\{\sigma_{x_1}, \sigma_{x_2}, \cdots , \sigma_{x_m}\}$ be a set of right coset representatives of $E$ in $G$, where  $\sigma_{x_1}$ is the identity in $G$. Let  $x_1, x_2, \dots , x_m$ be the  corresponding  orthogonal idempotents of $S\otimes_RS_{\frak{Q}}$ defined at the beginning of this section. 
Let $E = \{\sigma_{e_1}, \sigma_{e_2}, \dots , \sigma_{e_{|E|}}\}$  and let $Z_i$ be the right coset  $ E\sigma_{x_i} = \{\sigma_{e_1}\sigma_{x_i} = \sigma_{z_{i_1}}, 
\sigma_{e_2}\sigma_{x_i} = \sigma_{z_{i_2}}, \dots , \sigma_{e_{|E|}}\sigma_{x_i} = \sigma_{z_{i_{|E|}}}\}$.  Let us fix $i$ and assume that the indices of the homomorphisms in $E$ are labelled  such that $\Omega|_{Z_i}$ gives the total ordering on the indices of $Z_i$:  $z_{i_1} < z_{i_2} < \dots < z_{i_{|E|}}$. 
  Let $\Omega_1$ be a poset giving  the ordering, $e_1 < e_2 < \dots < e_{|E|}$  on the indices of the elements of $E$. Let $f_i$ be the homomorphism, 
$f_i : S\otimes_{S_E}S_{\frak{Q}} \to (S\otimes_RS_{\frak{Q}})x_i$
given by 
$$f_i(s_1 \otimes s_2) = (\sigma_{x_i}^{-1}(s_1)\otimes s_2)x_i$$
Then $f_i$ is an isomorphism of rings which   induces an
isomorphism  of stratified  exact  categories 
$$f_i^*:    \mathcal{C}_{(S\otimes_RS_{\frak{Q}}, \Omega|_{Z_i}, Z_i)}  \to  \mathcal{C}_{(S\otimes_{S_E}S_{\frak{Q}}, \Omega_1, E)}  .$$
\end{lemma}

\begin{proof}  From Lemma \ref{isomap}, we have an isomorphism $\gamma: S\otimes_{S_E}S_{\frak{Q}} \to (S\otimes_RS_{\frak{Q}})x_1$, with  $\gamma(s_1\otimes s_2) = 
(s_1 \otimes s_2)x_1$.  We also have a ring isomorphism   $\sigma_{x_i}^{-1} \otimes 1 : (S\otimes_{R}S_{\frak{Q}})x_1 \to (S\otimes_{R}S_{\frak{Q}})x_i$ by Lemma \ref{componentiso},
since letting $x_1 = \sum_l\alpha_l\otimes \beta_l$ we see 
  that $(\sigma_{x_i}^{-1} \otimes 1)(x_1) = x_i$ as follows; 
{\small   $$\phi_{g_k}((\sigma_{x_i}^{-1} \otimes 1)(x_1) ) =  \phi_{g_k}(\sum_l(\sigma_{x_i}^{-1}(\alpha_l)) \otimes \beta_l) )
= \Sigma_l\sigma_{g_k}(\sigma_{x_i}^{-1}(\alpha_l))\beta_l = 
\Big\{\  \begin{matrix}
&1& \hbox{if} \  \sigma_k \in E\sigma_{x_i} \\
&0&   \mbox{otherwise} 
\end{matrix}.$$}
This gives us that $f_i$ is the composition of the above  maps, $f_i = (\sigma_{x_i}^{-1}\otimes 1)\circ \gamma$,  and hence is an isomorphism of rings. 

We now show that the inverse of $f_i$, is given by   $f_i^{-1} = h_i :  (S\otimes_RS_{\frak{Q}})x_i \to S\otimes_{S_E}S_{\frak{Q}}$ where $h_i$ is the restriction of the homomorphism from $S\otimes_RS_{\frak{Q}}$ to $S\otimes_{S_E}S_{\frak{Q}}$ which sends $s_1\otimes s_2$ to 
$ \sigma_{x_i}(s_1) \otimes s_2$.  We have $h_i((f_i(s_1 \otimes s_2)) = h_i(\sigma_{x_i}^{-1}(s_1) \otimes s_2)x_i) = (s_1 \otimes s_2)h_i(x_i)$. Since $(\sigma_{x_i}^{-1} \otimes 1)(x_1) = x_i$ we have $(\sigma_{x_i} \otimes 1)(x_i) = x_1$. Now $h_i(x_i) = \psi((\sigma_{x_i} \otimes 1)(x_i)) = \psi(x_1) =  1\otimes 1$, where $\psi = \gamma^{-1}$ is the isomorphism defined in Lemma \ref{isomap}.  Thus $h_if_i = id_{S\otimes_{S_E}S_{\frak{Q}}}$.  It is easy to see that $f_ih_i = id_{S\otimes_{R}S_{\frak{Q}}}x_i$, using the fact that $x_i^2 = x_i$.

Since $f_i$ is an isomorphism of rings,  we have  that $f_i$ induces  an equivalence of categories $f_i^*: (S\otimes_RS_{\frak{Q}})x_i-\hbox{Mod} \to 
S\otimes_{S_E}S_{\frak{Q}}-\hbox{Mod}$ with inverse $h_i^*$. 
Now if $M$ is a module in  $\mathcal{C}_{(S\otimes_RS_{\frak{Q}}, \Omega|_{Z_i}, Z_i)} $, 
then, by Lemma \ref{form},  $M$  is an 
$(S\otimes_RS_{\frak{Q}})x_i$ module with a filtration 
$$M^{z_0} = M \supseteq M^{z_1} \supseteq M^{z_2} \supseteq \dots \supseteq M^{z_{|E|}} = 0.$$
Consider the filtration 
$$f_i^*(M^{z_0}) = f_i^*(M) \supseteq f_i^*(M^{z_1}) \supseteq f_i^*(M^{z_2}) \supseteq \dots \supseteq f_i^*(M^{z_{|E|}}) = 0.$$
Let $ [m] = m +  f_i^*(M^{z_k}) $ be an element of $ f_i^*(M^{z_{k- 1}}) / f_i^*(M^{z_k}) $ for some $k, 1\leq k \leq |E|$.
Then for $s\otimes1 \in S\otimes_{S_E}S_{\frak{Q}}$, we have $(s\otimes1)[m] = f_i(s\otimes 1)m + f_i^*(M^{z_k}) = 
(\sigma_{x_i}^{-1}(s)\otimes 1)m +  f_i^*(M^{z_k}) = m\sigma_{e_k}(\sigma_{x_i}(\sigma_{x_i}^{-1}(s))) +  f_i^*(M^{z_k})  = [m]\sigma_{e_k}(s) $.  The quotient  $ f_i^*(M^{z_{k - 1}}) / f_i^*(M^{z_k}) $ is finitely generated and projective as a right 
$S_{\frak{Q}}$ module since the right action of $S_{\frak{Q}}$ is that on the quotient $ M^{z_{k - 1}} / M^{z_k} $.
 Thus letting $(f_i^*(M))^{e_k} = f_i^*(M^{z_k})$, we get a filtration of $f_i^*(M)$ with the necessary  properties 
 to show that $f_i^*(M) \in  \mathcal{C}_{(S\otimes_{S_E}S_{\frak{Q}}, \Omega_1, E)} $. 

On the other hand let  $N\in  \mathcal{C}_{(S\otimes_{S_E}S_{\frak{Q}}, \Omega_1, E)}$, with filtration 
$$N^{e_0} = N \supseteq N^{e_1} \supseteq N^{e_2} \supseteq \dots \supseteq N^{e_{|E|}} = 0.$$
 We have $h_i^*(N)  \in 
(S\otimes_RS_{\frak{Q}})x_i-\hbox{Mod} $.  Consider the filtration 
$$h_i^*(N^{e_0}) = h_i^*(N) \supseteq h_i^*(N^{e_1}) \supseteq h_i^*(N^{e_2}) \supseteq \dots \supseteq h_i^*(N^{e_{|E|}})  = 0.$$
Let $[n] = n + h_i^*(N^{e_k}) \in  h_i^*(N^{e_{k - 1}}) / h_i^*(N^{e_k}) $.  For $s\otimes 1 \in S\otimes_RS_{\frak{Q}}$ we have $(s\otimes 1)[n] = h_i(s\otimes 1)n + h_i^*(N^{e_k}) =   (\sigma_{x_i}(s)\otimes 1)n + h_i^*(N^{e_k}) = n\sigma_{e_k}\sigma_{x_i}(s) + h_i^*(N^{e_k}) = n\sigma_{z_k} + h_i^*(N^{e_k})$. Since the action of $S_{\frak{Q}}$ on the quotient from the right does not change, the quotient is finitely generated and projective as a right $S_{\frak{Q}}$ module. Hence this gives the necessary filtration to show that $h_i^*(N)$ is in $\mathcal{C}_{(S\otimes_RS_{\frak{Q}}, \Omega|_{Z_i}, Z_i)} $. 

Now since $f_i^*h_i^* = Id_{\mathcal{C}_{(S\otimes_{S_E}S_{\frak{Q}}, \Omega_1, E)}}$ and 
$h_i^*f_i^* = id_{\mathcal{C}_{(S\otimes_RS_{\frak{Q}}, \Omega|_{Z_i}, Z_i)} }$, we have that 
both $f^*$ and $h^*$ preserve exact sequences of modules.  Since 
$f_i^*(M^{z_k}) = (f_i^*(M))^{e_k}$  and $h_i^*(N^{e_k}) = (f_i^*(N))^{z_k}$, for $M \in 
\mathcal{C}_{(S\otimes_RS_{\frak{Q}}, \Omega|_{Z_i}, Z_i)}$ and $N \in \mathcal{C}_{(S\otimes_{S_E}S_{\frak{Q}}, \Omega_1, E)}$, $1 \leq k \leq |E|$, we can easily see that 
$f_i^*$ and $h_i^*$ also take stratified exact sequences to stratified exact sequences. 
Hence $f^*$ is an equivalence of stratified exact categories. 
\end{proof}

\begin{lemma} 
\label{reduce} Let $L, K, S, R, \frak{Q}, \frak{P}, G, E = E(\frak{Q}|\frak{P})$ and $\Omega$ be as indicated at the beginning of this section.  
Let $\sigma_{x_i}$ be a right coset representative for the right coset $Z_i = E\sigma_{x_i}$ of $E$ in 
$G$.  Let $Q_i$ be a projective module in 
 $\mathcal{C}_{(S\otimes_R S, \Omega|_{Z_i}, Z_i)}$. Then $Q_i$  is projective in
$\mathcal{C}_{(S\otimes_R S, \Omega, G)}$.  
\end{lemma}

\begin{proof}  By Lemma \ref{form}  we have $Q_i = x_iQ_i$ and By Lemma \ref{mapsxi},  
$Hom_{S\otimes_RS_{\frak{Q}}}(Q_i, M) = Hom(Q_i, x_iM)$ for every $S\otimes_RS_{\frak{Q}}$ module in $\mathcal{C}_{(S\otimes_RS_{\frak{Q}}, \Omega, G)}$.  Now let 
$$0 \to M_1 \to M_2 \to M_3 \to 0$$
be a sheaf  exact sequence of modules in $\mathcal{C}_{(S\otimes_R S, \Omega, G)}$. 
It is not difficult to show that if $M_k, \ \ k = 1, 2, 3$ has a filtration:
$$M_k = M_k^{g_0} \supseteq M_k^{g_1} \supseteq M_k^{g_2} \supseteq \dots \supseteq M_k^{g_n} = \{0\}, $$
in $\mathcal{C}_{(S\otimes_R S, \Omega, G)}$, 
then
$$x_iM_k = x_iM_k^{g_0} \supseteq x_iM_k^{g_1} \supseteq x_iM_k^{g_2} \supseteq \dots \supseteq x_iM_k^{g_n} = \{0\}, $$
is a filtration for $x_iM_k$ in $\mathcal{C}_{(S\otimes_R S, \Omega, G)}$.   Using this filtration, 
 Lemma \ref{mapsxi} and Lemma \ref{form}, we see that 
$$0 \to x_iM_1 \to x_iM_2 \to x_iM_3 \to 0$$
is a sheaf  exact sequence of $S\otimes_RS_{\frak{Q}}$ modules, in
$\mathcal{C}_{(S\otimes_R S, \Omega, G)}$ and in 
$\mathcal{C}_{(S\otimes_R S, \Omega|_{Z_i}, Z_i)}$.
By Lemma \ref{mapsxi}, the sequence of Abelian groups 
$$0 \to Hom_{S\otimes_RS_{Q}}(Q_i, M_1) \to Hom_{S\otimes_RS_{Q}}(Q_i, M_2) \to Hom_{S\otimes_RS_{Q}}(Q_i, M_3) \to 0$$
is the same  as  the sequence
{\small $$0 \to Hom_{S\otimes_RS_{Q}}(Q_i, x_iM_1) \to Hom_{S\otimes_RS_{Q}}(Q_i, x_iM_2) \to Hom_{S\otimes_RS_{Q}}(Q_i, x_iM_3) \to 0.$$}
The latter sequence  is exact as a sequence of abelian groups  because $Q_i$ is projective in the category 
$\mathcal{C}_{(S\otimes_R S, \Omega|_{Z_i}, Z_i)}$. Hence $Q_i$ is projective in the category 
$\mathcal{C}_{(S\otimes_R S, \Omega, G)}$. 
\end{proof}

\begin{theorem} \label{local} Let $L, K, S, R, \frak{Q}, \frak{P}, G, E = E(\frak{Q}|\frak{P})$    be as defined in section 2.        
Let  $\Omega$ be as indicated at the beginning of this section.  
Let $\{\sigma_{x_i}\}_{i = 1}^{m}$ be a set  of  right coset representatives   for the right cosets $\{Z_i = E\sigma_{x_i}\}_{i = 1}^m$ of $E$ in 
$G$. Let $\{x_i\}_{i = 1}^m$ be the corresponding idempotents described at the beginning of this section.  Let $Q$ be a projective module in the category $\mathcal{C}_{(S\otimes_RS_{\frak{Q}}, \Omega, G)}$, then $Hom_{S\otimes_RS_{\frak{Q}}}(Q, S\otimes_RS_{\frak{Q}})$ is also projective in the category 
 $\mathcal{C}_{(S\otimes_RS_{\frak{Q}}, \Omega, G)}$.
\end{theorem}

\begin{proof} 
We first reduce the problem  to the category $ \mathcal{C}_{(S\otimes_{S_E}S_{\frak{Q}}, \Omega_1, E)}$. 
By Lemma \ref{mapsxi} we have  $Hom_{S\otimes_RS_{\frak{Q}}}(Q, S\otimes_RS_{\frak{Q}}) 
  = \oplus_iHom_{S\otimes_RS_{\frak{Q}}}(x_iQ, S\otimes_RS_{\frak{Q}})$.  By Lemma \ref{reduce}, 
  we have that $Q$ is projective in the category  $\mathcal{C}_{(S\otimes_RS_{\frak{Q}}, \Omega, G)}$, 
  if and only if each $x_iQ$ is projective in the corresponding category 
  $\mathcal{C}_{(S\otimes_R S_{\frak{Q}}, \Omega|_{Z_i}, Z_i)}$, for $1\leq i \leq m$.  Hence it is enough to prove   the result for the category $\mathcal{C}_{(S\otimes_R S_{\frak{Q}}, \Omega|_{Z_i}, Z_i)}$.

  Now by lemma  \ref{equivalence}  there is an isomorphism of  categories $f_i^* : \mathcal{C}_{(S\otimes_R S_{\frak{Q}}, \Omega|_{Z_i}, Z_i)} \to  \mathcal{C}_{(S\otimes_{S_E}S_{\frak{Q}}, \Omega_1, E)}$, where $\Omega_1$ is a poset giving some total ordering on the indices of $E; \{e_1, e_2, \cdots , e_{|E|}\}$. It is not difficult to see that if $Q_i$ is a module in $\mathcal{C}_{(S\otimes_R S_{\frak{Q}}, \Omega|_{Z_i}, Z_i)}$, then $Q_i$ is projective in $\mathcal{C}_{(S\otimes_R S_{\frak{Q}}, \Omega|_{Z_i}, Z_i)} $ if and only if $f^*(Q_i)$ is projective in $\mathcal{C}_{(S\otimes_{S_E}S_{\frak{Q}}, \Omega_1, E)}$. Hence it suffices to prove the theorem for 
  the category $\mathcal{C}_{(S\otimes_{S_E}S_{\frak{Q}}, \Omega_1, E)}$.

   Let $E = \{\sigma_{e_1}, \sigma_{e_2}, \cdots , \sigma_{e_{|E|}}\}$ and let $\Omega_1 $ be the poset  $ \{e_1, e_2, 
   \cdots , e_{|E|}\}$ with ordering $e_1 < e_2 < \cdots < e_{|E|}$.  Let $\phi' :  L\otimes_{L_E}L \to L_1 \oplus L_2 \oplus \cdots \oplus L_{|E|}$ and 
   $\phi'_{e_k}$ be as defined in Section 2.    Let  $(S\otimes_{S_E}S_{\frak{Q}})_{E, i} $ and $P_{E, \frak{Q}, i}$ be as in Definition \ref{EQ}. 
     From Corollary \ref{PlusEQ} and \cite{Dyer}[Theorem 1.19], we have that 
    $P_{E, \frak{Q}}  = P_{E, \frak{Q}, 1} \oplus P_{E, \frak{Q}, 2} \oplus \cdots \oplus P_{E, \frak{Q}, |E|}$ is a projective generator in
    $\mathcal{C}_{(S\otimes_{S_E}S_{\frak{Q}}, \Omega_1, E)}$.

 Let $O$ be a projective module in  $ \mathcal{C}_{(S\otimes_{S_E}S_{\frak{Q}}, \Omega_1, E)}$.
    Since $P_{E, \frak{Q}}$ is a projective generator, $O$ is a direct summand of $P_{E, \frak{Q}}^n$ for some $n \geq 1$
    and hence 
    $Hom_{S\otimes_{S_E}S_{\frak{Q}}}(O, S\otimes_{S_E}S_{\frak{Q}}) $ is a direct summand of 
    $Hom_{S\otimes_{S_E}S_{\frak{Q}}}(P_{E, \frak{Q}}^n, S\otimes_{S_E}S_{\frak{Q}})$.
     This is projective  if $Hom_{S\otimes_{S_E}S_{\frak{Q}}}(P_{E, \frak{Q}}^n, S\otimes_{S_E}S_{\frak{Q}})$ 
    is projective, which in turn is projective if\\
     $Hom_{S\otimes_{S_E}S_{\frak{Q}}}(P_{E, \frak{Q}}, 
     S\otimes_{S_E}S_{\frak{Q}})$
    is projective, since direct summands and direct sums of projective modules are projective. 
       Now $Hom_{S\otimes_{S_E}S_{\frak{Q}}}(P_{E, \frak{Q}}, S\otimes_{S_E}S_{\frak{Q}}) = \oplus_{i= 
       1}^{
       |E|}Hom_{S\otimes_{S_E}S_{\frak{Q}}}(P_{E, \frak{Q}, i}, S\otimes_{S_E}S_{\frak{Q}})$, 
    hence it is projective if and only if each $Hom_{S\otimes_{S_E}S_{\frak{Q}}}(P_{E, \frak{Q}, i}, S
    \otimes_{S_E}S_{\frak{Q}})$ is projective in $ \mathcal{C}_{(S\otimes_{S_E}S_{\frak{Q}}, \Omega_1, 
    E)}$ for each $i,  1 \leq i \leq |E|$.

  Recall from Lemma \ref{Pie} and Lemma \ref{structure}   that $S\otimes_{S_E}S_{\frak{Q}} = S_E[\Pi ]\otimes_{S_E}S_{\frak{Q}}$, where $\Pi \in \frak{Q}$ and 
  $S\subseteq (S_E)_{{\frak{Q}_E}}[\Pi]$.  Note that this implies that 
  $L_E(\Pi) = L$.  Let $I_{E, e_i} = \{\sigma_{e_{i }}, \sigma_{e_{i + 1}}, \cdots , \sigma_{e_{|E|}}\}$ and 
    for a given $i , 1 \leq i \leq k$, let $(S\otimes_{S_E}S_{\frak{Q}})_{I_{E, e_i} } =  \{x \in S\otimes_{S_E}S_{\frak{Q}} | \phi'_{e_k}(x) = 0 \ \mbox{if} \ k <  i\}$. 
    By Corollary \ref{homQ}, with the appropriate substitutions,  we have that
    \[Hom_{S\otimes_{S_E}S_{\frak{Q}}}(P_{E, \frak{Q}, i},  S\otimes_{S_E}S_{\frak{Q}}) \cong  
    (S\otimes_{S_E}S_{\frak{Q}})_{I_{E, e_i} }
\]
Also from Corollary \ref{homQ}  the isomorphism is given by $G:
Hom_{S\otimes_{S_E}S_{\frak{Q}}}(P_{E, \frak{Q}, i},  S\otimes_{S_E}S_{\frak{Q}}) \to
    (S\otimes_{S_E}S_{\frak{Q}})_{I_{E, e_i} }$, where $G(f) = f(1\otimes 1 + (S\otimes_{S_E}S_{\frak{Q}})_{E, i})$ 
and is  an isonmorphism of  $S\otimes_{S_E}S_{\frak{Q}}$ modules. 

It remains  to show that $  (S\otimes_{S_E}S_{\frak{Q}})_{I_{E, e_i} }$ is projective in 
$ \mathcal{C}_{(S\otimes_{S_E}S_{\frak{Q}}, \Omega_1, E)}$. We will in fact show that 
$  (S\otimes_{S_E}S_{\frak{Q}})_{I_{E, e_i} }$ is isomorphic to $P_{E, \frak{Q}, i}$ as an $S\otimes_RS_{\frak{Q}}$ module. Recall that $A_{e_k}(\Pi) = \Pi \otimes 1 - 1 \otimes \sigma_{e_k}(\Pi) $ for $1\leq k \leq |E|$. From Lemma 
\ref{basisP}, we see that $(S\otimes_{S_E}S_{\frak{Q}})_{I_{E, e_i}}  = \prod_{k < i}A_{e_k}(\Pi)(S\otimes_{S_E}S_{\frak{Q}})$. Now consider the $S\otimes_{S_E}S_{\frak{Q}}$ module homomorphism  $F: S\otimes_{S_E}S_{\frak{Q}} \to 
 (S\otimes_{S_E}S_{\frak{Q}})_{I_{E, e_i} }$, where   $F(x) = 
 (\prod_{k < i}A_{e_k}(\Pi)  )x$. 
  It is easy to see that $(S\otimes_{S_E}S_{\frak{Q}})_{E, i}  \subseteq  \ker F$, since $\phi'_{e_l}(\prod_{k < i}A_{e_k}(\Pi)) x) = 0 $ for all $l, \ \  1\leq l \leq |E|$, when $x \in (S\otimes_{S_E}S_{\frak{Q}})_{E, i}$. 
  On the other hand, if $x \in \ker F$, we have $\phi'_{e_l}((\prod_{k < i}A_{e_k}(\Pi)) x) = 0 $ for each $l$. 
Hence for $l \geq i$, we have $\phi'_{e_l}((\prod_{k < i}A_{e_k}(\Pi)) x) = \phi'_{e_l}(\prod_{k < i}A_{e_k}(\Pi)) \phi'_{e_l}(x) = 0 $, and since $S_{\frak{Q}}$ is a domain, we must have either
$\phi'_{e_l}(\prod_{k < i}A_{e_k}(\Pi))  = 0$ or 
  $\phi'_{e_l}(x) = 0 $.  Now if $l \geq i$ and $k < i$, then $\phi'(A_{e_k}(\Pi)) = \sigma_{e_l}(\Pi) - \sigma_{e_k}(\Pi) \not= 0$ since $\sigma_{e_l} \not= \sigma_{e_k}$ and $L = L_E(\Pi)$.   
Therefore $x \in (S\otimes_{S_E}S_{\frak{Q}})_{E, i} $.  Hence $\ker F = (S\otimes_{S_E}S_{\frak{Q}})_{E, i} $ and 
$F$ lifts to an isomorphism of $S\otimes_{S_E}S_{\frak{Q}}$ modules 
$$\bar{F} : P_{E, \frak{Q}, i} = \frac{S\otimes_{S_E}S_{\frak{Q}}}{(S\otimes_{S_E}S_{\frak{Q}})_{E, i}} \to  (S\otimes_{S_E}S_{\frak{Q}})_{I_{E, e_i} } \cong  Hom_{S\otimes_{S_E}S_{\frak{Q}}}(P_{E, \frak{Q}, i},   S\otimes_{S_E}S_{\frak{Q}}). $$
Therefore $Hom_{S\otimes_{S_E}S_{\frak{Q}}}(P_{E, \frak{Q}, i},  S\otimes_{S_E}S_{\frak{Q}}) $ is projective in the category $ \mathcal{C}_{(S\otimes_{S_E}S_{\frak{Q}}, \Omega_1, E)}$ and this completes the proof of our theorem. 
   \end{proof}

We are now ready to prove a global theorem on Duality.

\begin{theorem} Let $L, K, S, R, \frak{Q},$ and $G = \{\sigma_{g_1}, \sigma_{g_2}, \cdots , \sigma_{g_n}\}$ be as defined in Section 2. Let $\Omega$ be a poset giving a total ordering on the indices of $G$. Let 
$Q$ be a projective module in $ \mathcal{C}_{(S\otimes_RS, \ \Omega, G)}$,   then $Hom(Q, 
S\otimes_RS)$ is also projective in $ \mathcal{C}_{(S\otimes_RS, \ \Omega, G)}$. 
\end{theorem}

\begin{proof} From Lemma \ref{extproj}, we have that $Q$ is projective in $ \mathcal{C}_{(S\otimes_RS, \ \Omega, G)}$ if and only if the short exact sequence 
$$0 \to Q^{g_{i - 1}}/Q^{g_i} \to Q/Q^{g_i} \to Q/Q^{g_{i - 1}} \to 0$$
gives rise to an exact sequence
$$0 \to Hom(Q/Q^{g_{i - 1}}, S[\sigma_{g_i}]) \to Hom(Q/Q^{g_i}, S[\sigma_{g_i}])  \to Hom(Q^{g_{i - 1}}/Q^{g_i}, S[\sigma_{g_i}])$$
$
\xrightarrow{f_i}  Ext^1(Q/Q^{g_{i - 1}}, S[\sigma_{g_i}]) \to 0$\\
for each $i, 1 \leq i \leq n$, where $Hom(M, N)$ denotes $Hom_{S\otimes_RS}(M, N)$. 
Thus $Q$ is projective in $ \mathcal{C}_{(S\otimes_RS, \ \Omega, G)}$ if and only if the maps $f_i$ are surjective  for $1 \leq i \leq n$, since 
$Hom( - , S[\sigma_{g_i}])$ is a  left exact contravariant functor on $S\otimes_RS$ modules. 

For any $i, 1 \leq i \leq n$,  $f_i$ is surjective if and only if the localization 
$$(f_{i})_\frak{Q} :  Hom(Q^{g_{i - 1}}/Q^{g_i}, S[\sigma_{g_i}]) \otimes_SS_{\frak{Q}}  \to  Ext^1(Q/Q^{g_{i - 1}}, S[\sigma_{g_i}])  \otimes_SS_{\frak{Q}}$$
is surjective for each maximal ideal $\frak{Q}$ of $S$, by \cite{Reiner}[3.16(2)]. This in turn is true if and only if the map $(f_{i})_\frak{Q}'$ in the sequence below is surjective for each maximal ideal 
$\frak{Q}$ of $S$ (by \cite{Reiner}[2.39]) : 

{\small  $$ 0 \to Hom(Q_{\frak{Q}}/Q_{\frak{Q}}^{g_{i - 1}}, S_{\frak{Q}}[\sigma_{g_i}]) \to Hom(Q_{\frak{Q}}/Q_{\frak{Q}}^{g_i}, S_{\frak{Q}}[\sigma_{g_i}])  \to Hom(Q_{\frak{Q}}^{g_{i - 1}}/Q_{\frak{Q}}^{g_i}, S_{\frak{Q}}[\sigma_{g_i}])$$
$\xrightarrow{(f_{i})_\frak{Q}'}  Ext^1(Q_{\frak{Q}}/Q_{\frak{Q}}^{g_{i - 1}}, S_{\frak{Q}}[\sigma_{g_i}]) \to 0.$}

By applying Lemma \ref{extproj} again, we see that this is true if and only if $Q_{\frak{Q}}$ is projective 
in $ \mathcal{C}_{(S\otimes_RS_{\frak{Q}}, \ \Omega, G)}$. Hence $Q$ is projective in $ \mathcal{C}_{(S\otimes_RS, \ \Omega, G)}$ if and only if $Q_{\frak{Q}}$ is projective in 
$ \mathcal{C}_{(S\otimes_RS_{\frak{Q}}, \ \Omega, G)}$ for each maximal ideal 
$\frak{Q}$ of $S$. 

Thus if $Q$ is projective in the category  $ \mathcal{C}_{(S\otimes_RS, \ \Omega, G)}$, 
then $Q_{\frak{Q}}$ is projective in $ \mathcal{C}_{(S\otimes_RS_{\frak{Q}}, \ \Omega, G)}$
for each maximal ideal, $\frak{Q}$, of $S$. Hence by Theorem \ref{local}, \\
$Hom_{S\otimes_RS_{\frak{Q}}}(Q_{\frak{Q}}, S\otimes_RS_{\frak{Q}})$ is projective for each maximal ideal $\frak{Q}$ of $S$. Since $Hom_{S\otimes_RS_{\frak{Q}}}(Q_{\frak{Q}}, S\otimes_RS_{\frak{Q}})$ 
is isomorphic to  $(Hom(Q, 
S\otimes_RS))_{\frak{Q}}$, we can conclude that $Hom(Q, S\otimes_RS)$ is projective in the category
$ \mathcal{C}_{(S\otimes_RS, \ \Omega, G)}$.

\end{proof}

\section{\bf BGG Reciprocity}

Let $L, K, S, R, G = \{\sigma_{g_1}, \sigma_{g_2},  \cdots  , \sigma_{g_n}\}  , \frak{Q}, \frak{P}$ and $E = E(\frak{Q} | \frak{P})$ be as defined in Section 2. 
Let $ \mathcal{C}_{(S\otimes_RS_{\frak{Q}}, \ \Omega,  G)}$ be the associated stratified exact category 
defined in section 4, where $\Omega$ is a poset giving a  a total ordering on the indices of the elements of $G$, $\{g_1, g_2, \cdots , g_n\}$.
Let $P_{\frak{Q}, i} \ 1\leq i \leq n $ be the projective modules from Definition  \ref{Q}. 
  From Corollary \ref{PlusQ} and \cite{Dyer}[Theorem 1.19], we have that 
    $P_{ \frak{Q}}  = P_{ \frak{Q}, 1} \oplus P_{\frak{Q}, 2} \oplus \cdots \oplus P_{ \frak{Q}, n}$ is a projective generator in
    $\mathcal{C}_{(S\otimes_{R}S_{\frak{Q}}, \Omega, G)}$.
Let $\mathcal{A}_{S\otimes_RS_{\frak{Q}}} (P_{\frak{Q}}) = End_{S\otimes_RS_{\frak{Q}}}(P_{\frak{Q}})^{op}$ be the associated algebra, the structure of which has been determined in Section 8.
Let $F_{\frak{Q}} = S/\frak{Q}$ be the residue class field of $\frak{Q}$ in $S$. 
In this section, we will consider the $F_{\frak{Q}} $ algebra $\mathcal{A}_{\frak{Q}} = \mathcal{A}_{S\otimes_RS_{\frak{Q}}}  (P_{\frak{Q}})\otimes_SF_{\frak{Q}} $  In particular we will demonstrate that $\mathcal{A}_{\frak{Q}}$  exhibits a reciprocity analogous to the BGG reciprocity explored in \cite{Irving} and \cite{CPS}.  This reciprocity 
holds in general for the classes of algebras defined in \cite{CPS} and  \cite{Dyer}, however the general proof relies on much deeper and more subtle arguments than those presented here. In our case the reciprocity can be seen fairly easily with the aid of the matrix representation of $\mathcal{A}_{\frak{Q}}$ and our arguments permit the  determination of the composition   factor multiplicities of simple modules in Verma modules  which are all either 
0 or 1 here. 

We will first define BGG reciprocity for this situation. For  background information on modules, see 
\cite{CnR}.  Let $\mathcal{A}$ be a finite dimensional algebra over a finite field $F_{\frak{Q}}$. 
Let $\Lambda = \{\lambda_1, \lambda_2, \cdots , \lambda_n\}$ be a finite poset,  with a total ordering $\lambda_1 < \lambda_2 < \cdots < \lambda_n$,  in bijective correspondence with the isomorphism classes of simple  modules in the category of finitely generated left modules over $\mathcal{A}$, $\mathcal{A}$-mod. Let $L(\lambda_i)$ be the simple module corresponding to $\lambda_i \in \Lambda$ and let $P(\lambda_i)$ denote a projective cover of $L(\lambda_i)$, that is 
 $P(\lambda_i)$ is  projective  with $P(\lambda_i)/Rad(P(\lambda_i)) = L(\lambda_i)$.  $P(\lambda_i)$ is automatically indecomposable, by Nakayama's lemma \cite{CnR}[30.2].  
Every projective module in $\mathcal{A}$-mod is a direct sum of indecomposable projective modules, each of which is isomorphic to one of the $P(\lambda_i)$'s. 
 If $M$ is an $\mathcal{A}$-mod, then  $M$ has  a finite composition series:
$$M^0 = \{0\} \subset M^1 \subset M^2 \subset \cdots \subset M^K = M,$$
where each quotient, $M^k/M^{k-1}$,  is isomorphic to one of the simple modules $L(\lambda_i)$. 
 we let $[M: L(\lambda_i)]$ denote  the multiplicity of $L(\lambda_i)$ in $M$,   the number of factors $M^k/M^{k-1}$ in the composition series which are 
 isomorphic to $L(\lambda_i)$.  We let $\overline{M} = M/Rad(M)$, 
 where $Rad(M)$ denotes the radical of $M$.  
 
 \begin{definition} A collection of modules $\{M(\lambda_i) | \lambda_i \in \Lambda \}$ is called a choice of 
 Verma modules for $\mathcal{A}$-mod if  the following  conditions hold for each $\lambda_i \in \Lambda$:
 
 $\overline{M(\lambda_i)} = L(\lambda_i)$,  
 $[M(\lambda_i) : L(\lambda_j)] = 0 $ unless $\lambda_j \leq \lambda_i$,
  $[M(\lambda_i): L(\lambda_i)] = 1$, $1 \leq i, j \leq n$, and any other $\mathcal{A}$ module with these properties is a quotient of 
  $M(\lambda_i)$. 
 \end{definition}

Let $Gr(\mathcal{A})$ be the Grothendeick group of $\mathcal{A}$-mod. For any module $M$ in
$\mathcal{A}$-mod  let $[M]$ denote the class of $M$ in $Gr(\mathcal{A})$.  We have that the set $\{[L(\lambda_i)] 
| \lambda \in \Lambda \}$ is a basis for the free abelian group $Gr(\mathcal{A})$, since every module 
in $\mathcal{A}$-mod has a composition series. Indeed 
$$[M] = \sum_{\lambda_i}[M: L(\lambda_i)][L(\lambda_i)].$$

\begin{lemma}  \label{Verma} Let $\mathcal{A}$ be a finite dimensional algebra over $F_{\frak{Q}}$, and  let 
$\{M(\lambda_i) | \lambda_i \in \Lambda \}$  be  a choice of 
 Verma modules for $\mathcal{A}$-mod, then the set $\{[M(\lambda_i)] 
| \lambda_i \in \Lambda \}$ is a basis for the free abelian group $Gr(\mathcal{A})$. 
\end{lemma}

\begin{proof} We have  $[M(\lambda_i)] = \sum_{j = 1}^{n}[M(\lambda_i): L(\lambda_j)][L(\lambda_j)]$.
By the definition of Verma modules, the matrix 
$$((M(\lambda_i), l(\lambda_j))$$
is upper triangular with $1$'s on the diagonal and hence is invertible.  Hence the set $\{[M(\lambda_i)] 
| \lambda_i \in \Lambda \}$ is a basis for the free abelian group $Gr(\mathcal{A})$. 
\end{proof}

Letting $\mathcal{A}$ be a finite dimensional $F_{\frak{Q}}$ algebra as above, with a choice 
$\{M(\lambda_i) | \lambda_i \in \Lambda \}$ of Verma modules for $\mathcal{A}$-mod, we define a 
Verma flag for a module $M$ of $\mathcal{A}$-mod as a filtration
$$\{0\} = M_0 \subset M_1 \subset \cdots \subset M_R = M$$
with the property that $M_k/M_{k - 1} = M(\lambda_i)$ for some  $\lambda_i \in \Lambda$. 

If a module $M$ in $\mathcal{A}$-mod has such a Verma flag, we define $(M: M(\lambda_i))$ to be 
the number of quotients in that Verma flag which are equal to $M(\lambda_i)$. This is independent of the Verma flag chosen since the chosen  Verma modules form a basis for the Grothendeick group.
In fact if a module, $M$,  has a Verma flag, we have 
$$[M] = \sum_{\lambda_i}(M: M(\lambda_i)) [M(\lambda_i)].$$
In general, not every module has a Verma flag, however in our category $\mathcal{A}_{\frak{Q}}$-mod, with our choice of Verma modules,   every module will have a Verma flag, stemming from a filtration in the category $\mathcal{C}_{(S\otimes_{R}S_{\frak{Q}}, \Omega, G)}$. 

\begin{definition} Let $\mathcal{A}$ be a finite dimensional algebra over $F_{\frak{Q}}$, with the isomorphism classes of the simple modules indexed by the poset $\Lambda$. Let 
$\{L(\lambda_i) | \lambda_i \in \Lambda \}$  be  a set of representatives for the  isomorphism classes
of the simple modules in $\mathcal{A}$-mod, with projective covers $\{P(\lambda_i) | \lambda_i \in \Lambda \}$. We say that the algebra $\mathcal{A}$-mod has BGG reciprocity if we can make a choice 
$\{M(\lambda_i) | \lambda_i \in \Lambda \}$  of Verma modules  in $\mathcal{A}$-mod such that 
each $P(\lambda_i)$, $\lambda_i \in \Lambda$ has a Verma flag and 
$$(P(\lambda_i) : M(\lambda_j)) = [M(\lambda_j) : L(\lambda_i)]$$
for all $\lambda_i, \lambda_j \in \Lambda$. 
\end{definition}

It is not difficult to see, reasoning as in the proof of Lemma \ref{Verma} that the set $\{[P(\lambda_i)] | \lambda_i \in \Lambda \}$ forms a basis for $GR(\mathcal{A})$ if $\mathcal{A}$ has BGG reciprocity. 
In this case the Cartan matrix of $\mathcal{A}$ defined as the matrix $C = ([P(\lambda_i) : L(\lambda_j)])_{i, j}$ is a product:
$$C =  \Big((P(\lambda_i) : M(\lambda_k))\Big) \Big([M(\lambda_k) : L(\lambda_i)]\Big) = D^tD,$$
where $D =  \Big([M(\lambda_k) : L(\lambda_i)]\Big)_{k, i} $.  In particular the Cartan matrix is symmetric  when the category  $\mathcal{A}$-mod has BGG reciprocity.

We now return to the finite dimensional $F_{\frak{Q}}$ algebra, $A_{\frak{Q}} = \mathcal{A}_{S\otimes_RS_{\frak{Q}}}(P_{\frak{Q}})\otimes_{S_{\frak{Q}}}F_{\frak{Q}}$,  introduced at the beginning of this section. We will show that   $\mathcal{A}_{\frak{Q}}$ is a direct sum of subalgebras and  show that each subalgebra has BGG reciprocity.  This will  give  BGG reciprocity for  $\mathcal{A}_{\frak{Q}}$-mod. 

Let $\{E\sigma_{x_1}, E\sigma_{x_2}, \cdots , E\sigma_{x_m}\}$ be  the right cosets of $E$ in $G$. 
According to Corollary \ref{corstar},  for any given prime ideal $\frak{Q}$ of $S$, we have orthogonal idempotents $\{x_1, x_2, \cdots , x_m\}$,   in the ring $S\otimes_RS_{\frak{Q}}$,  such that 
$$\phi_k(x_i)  =  \Big\{\  \begin{matrix}
& 1 & \hbox{if}  & \sigma_k \in E\sigma_{x_i} \\
& 0 &   &\mbox{otherwise} &
\end{matrix}.
$$

 It is easy to  see that $ \mathcal{A}_{\frak{Q}}$ splits into a direct sum  of subrings since 
 $\mathcal{A}_{S\otimes_RS_{\frak{Q}}}(P_{\frak{Q}}) = 
 End_{S\otimes_RS_{\frak{Q}}}(P_{\frak{Q}})^{op}  \cong 
\oplus_{i = 1}^{m} End_{S\otimes_RS_{\frak{Q}}}(P_{\frak{Q}}x_i)^{op}$. 
   Let $\mathcal{A}_{\frak{Q}, i}  =  End_{S\otimes_RS_{\frak{Q}}}(P_{\frak{Q}}x_i)^{op}$. Since 
   $P_{\frak{Q}}$ is a projective generator for $\mathcal{C}_{(S\otimes_RS_{\frak{Q}}, \Omega, G)}$, 
we can see from Lemmas  \ref{mapsxi} and  \ref{form} that  $P_{\frak{Q}}x_i$ is a projective generator for 
$ \mathcal{C}_{(S\otimes_RS_{\frak{Q}}, \ \Omega|_{Z_i}, G)}$, where $Z_i = E\sigma_{x_i}$. 
The algebra associated to $ \mathcal{C}_{(S\otimes_RS_{\frak{Q}}, \ \Omega|_{Z_i}, G)}$ is 
$\mathcal{A}_{\frak{Q}, i}  =  End_{S\otimes_RS_{\frak{Q}}}(P_{\frak{Q}}x_i)^{op}$.
By Lemma \ref{equivalence},  we have and isomorphism of categories 
$f^*: \mathcal{C}_{(S\otimes_RS_{\frak{Q}}, \ \Omega|_{Z_i}, G)} \to  \mathcal{C}_{(S\otimes_{S_E}S_{\frak{Q}}, \ \Omega_i, E)}$, where $\Omega_i$ is a poset giving  an ordering on the indices of $E$, determined by the poset  $\Omega|_{Z_i}$. 
This isomorphism of categories induces an isomorphism of the algebras $\mathcal{A}_{\frak{Q}, i} $ and 
$\mathcal{A}_{S\otimes_{S_E}S_{\frak{Q}}}(f^*(P_{\frak{Q}}x_i))$. Hence we will focus our efforts on showing that the category 
$\mathcal{A}_{S\otimes_{S_E}S_{\frak{Q}}}(f^*(P_{\frak{Q}}x_i))$-mod exhibits BGG reciprocity. In fact we will show that a category which is Morita equivalent 
to $\mathcal{A}_{S\otimes_{S_E}S_{\frak{Q}}} (f^*(P_{\frak{Q}}x_i))$-mod exhibits BGG reciprocity. 

Let $\Omega'$ be the poset $\{e_1, e_2, \dots , e_{|E|}\}$, with  the ordering $e_1 < e_2 < \dots < e_{|E|}$. 
  Let $P_{E, \frak{Q}} = \oplus_{i = 1}^{|E|} P_{E, \frak{Q}, i}$, where $P_{E, \frak{Q}, i}$ is as defined in 
  Definition  \ref{EQ}.  Recall that we can show that $P_{E, \frak{Q}}$ is a projective generator for 
  $ \mathcal{C}_{(S\otimes_{S_E}S_{\frak{Q}}, \ \Omega', E)}$ using Corollary  \ref{PlusQ}  and \cite{Dyer}[Theorem 1.19]. 
By Lemma \ref{Newprogen} and the discussion prior to it, there is a Morita equivalence between the categories 
$\mathcal{A}_{S\otimes_{S_E}S_{\frak{Q}}}  (P_{E, \frak{Q}})$-mod  and $\mathcal{A}_{S\otimes_{S_E}S_{\frak{Q}}}  (f^*(P_{\frak{Q}}x_i))$-mod, when the indices of $E$ are labelled so that 
$\Omega' = \Omega_i$. Hence we shall show that the category $\mathcal{A}_{S\otimes_{S_E}S_{\frak{Q}}} (P_{E, \frak{Q}})$-mod has BGG reciprocity. 

Our first objective is to use the matrix structure of the algebra $\mathcal{A}_{S\otimes_{S_E}S_{\frak{Q}} } (P_{E, \frak{Q}}) $ developed in Section 7 to identify 
the simple modules $L(e_i), \ 1 \leq i \leq |E|$, in the category $\mathcal{A}_{S\otimes_{S_E}S_{\frak{Q}}} (P_{E, \frak{Q}})$-mod. 
Using Corollary \ref{AlgEQ} 
 we see that  $\mathcal{A}_{S\otimes_{S_E}S_{\frak{Q}}}(P_{E, \frak{Q}})  = $
 $$
  \frac{
\begin{pmatrix}
S\otimes_{S_E}S_{\frak{Q} }& S\otimes_{S_E}S_{\frak{Q}} & S\otimes_{S_E}S_{\frak{Q}} & \cdots& \\
(S\otimes_{S_E}S_{\frak{Q}})_{I_{E, e_2}} & S\otimes_{S_E}S_{\frak{Q}} & S\otimes_{S_E}S_{\frak{Q}} & \cdots  &\\
(S\otimes_{S_E}S_{\frak{Q}})_{I_{E, e_3}} & (S\otimes_{S_E}S_{\frak{Q}})_{I_{E, e_3} \cup I_{E, e_2}^c} & S\otimes_{S_E}S_{\frak{Q}} & \cdots &\\
(S\otimes_{S_E}S_{\frak{Q}})_{I_{E, e_4}} & (S\otimes_{S_E}S_{\frak{Q}})_{I_{E, e_4} \cup I_{E, e_2}^c} & (S\otimes_{S_E}S_{\frak{Q}})_{I_{E, e_4}\cup I_{E, e_3}^c} & \cdots  &\\
\vdots & \vdots & \vdots & \ & \\
\ &\ &\ &\ &
\end{pmatrix}}
{\begin{pmatrix}
0& (S\otimes_{S_E}S_{\frak{Q}})_{I_{E, e_2}^c}   & (S\otimes_{S_E}S_{\frak{Q}})_{I_{E, e_3}^c}  & \cdots & \\
0& (S\otimes_{S_E}S_{\frak{Q}})_{I_{E, e_2}^c}  & (S\otimes_{S_E}S_{\frak{Q}})_{I_{E, e_3}^c}  & \cdots  &\\
0& (S\otimes_{S_E}S_{\frak{Q}})_{I_{E, e_2}^c}  & (S\otimes_{S_E}S_{\frak{Q}})_{I_{E, e_3}^c}  & \cdots &\\
0& (S\otimes_{S_E}S_{\frak{Q}})_{I_{E, e_2}^c}  & (S\otimes_{S_E}S_{\frak{Q}})_{I_{E, e_3}^c}   & \cdots  &\\
\vdots & \vdots & \vdots & \ & \\
\ &\ &\ &\ &
\end{pmatrix}},
$$
where $I_{E, e_k}$  and $(S\otimes_{S_E}S_{\frak{Q}})_{I_{E, e_k}} , 1\leq k \leq |E|$ are as defined in Definition \ref{EQ}. 
By Lemma \ref{structure}, we have $S\otimes_{S_E}S_{\frak{Q}} = S_E[\Pi]\otimes_{S_E}S_{\frak{Q}}$, where $\Pi$ is an element of $S$ of order 1 in $\frak{Q}$.  Let $\Psi' : S\otimes_{S_E}S_{\frak{Q}} \to S\otimes_{S_E}S_{\frak{Q}}\otimes_{S_{\frak{Q}} }F_{\frak{Q}}$ be the map sending $s_1\otimes s_2$ to $s_1\otimes s_2 \otimes 1$.
Let $A_{e_i}(\Pi) = \Pi\otimes 1 - 1 \otimes \sigma_{e_i}(\Pi)$, then by Lemma \ref{FQ}, we have 
 $\Psi'(A_{e_i}(\Pi) )   = \Psi'(\Pi\otimes 1 - 1 \otimes \Pi) = \alpha$, for $1 \leq i \leq |E|$. 
Using  
Lemma \ref{basisP}, with $R$ replaced by $S_E$ and  $G$ replaced by $E$, we see 
 that if $I \subset E$, then $(S\otimes_{S_E}S_{\frak{Q}})_I = \prod_{\sigma_{e_k} \not\in I}A_{e_k}(\Pi)(S\otimes_{S_E}S_{\frak{Q}})$, where $(S\otimes_{S_E}S_{\frak{Q}})_I = \{ x \in S\otimes_{S_E}S_{\frak{Q}} | \phi_{e_k}'(x) = 0 \ \  \mbox{if} \ \ \sigma_{e_k}  \not\in I\}$.  
 For convenience, we let $\mathcal{A}'_{\frak{Q}} = \mathcal{A}_{S\otimes_{S_E}S_{\frak{Q}}}(P_{E, \frak{Q}})\otimes_{S_{\frak{Q}}}F_{\frak{Q}}$. We let $\mathcal{R}'_{\frak{Q}}$ denote the ring
$$\mathcal{R}'_{\frak{Q}}  = 
\begin{pmatrix}
F_{\frak{Q}}[\alpha]&F_{\frak{Q}}[\alpha]  & F_{\frak{Q}}[\alpha] &\cdots & F_{\frak{Q}}[\alpha] \\
\alpha F_{\frak{Q}}[\alpha] & F_{\frak{Q}}[\alpha] & F_{\frak{Q}}[\alpha] & \cdots  & F_{\frak{Q}}[\alpha]\\
\alpha^2 F_{\frak{Q}}[\alpha]&\alpha F_{\frak{Q}}[\alpha] & F_{\frak{Q}}[\alpha] & \dots \dots & F_{\frak{Q}}[\alpha]\\
\alpha^3 F_{\frak{Q}}[\alpha]& \alpha^2 F_{\frak{Q}}[\alpha]  & \alpha F_{\frak{Q}}[\alpha] & \dots \dots & F_{\frak{Q}}[\alpha]\\
\vdots & \vdots & \vdots & \ & \\
\alpha^{|E| -1} F_{\frak{Q}}[\alpha] &\alpha^{|E| - 2} F_{\frak{Q}}[\alpha] &\alpha^{|E| - 3} F_{\frak{Q}}[\alpha] &\ & F_{\frak{Q}}[\alpha] 
\end{pmatrix},
$$
and we let $\mathcal{I}'_{\frak{Q}}$ denote the ideal
$$\mathcal{I}'_{\frak{Q}}  = 
{\begin{pmatrix}
0&\alpha^{|E| - 1} F_{\frak{Q}}[\alpha] &\alpha^{|E| - 2} F_{\frak{Q}}[\alpha] & \dots \dots & \alpha F_{\frak{Q}}[\alpha]\\
0& \alpha^{|E| - 1} F_{\frak{Q}}[\alpha] &\alpha^{|E| - 2} F_{\frak{Q}}[\alpha] &\cdots  &\alpha F_{\frak{Q}}[\alpha]\\
0&\alpha^{|E| - 1} F_{\frak{Q}}[\alpha] &\alpha^{|E| - 2} F_{\frak{Q}}[\alpha] & \cdots &\alpha F_{\frak{Q}}[\alpha]\\
0&\alpha^{|E| - 1} F_{\frak{Q}}[\alpha] &\alpha^{|E| - 2} F_{\frak{Q}}[\alpha] & \cdots &\alpha F_{\frak{Q}}[\alpha]\\
\vdots & \vdots & \vdots & \ & \\
0&\alpha^{|E| - 1} F_{\frak{Q}}[\alpha] &\alpha^{|E| - 2} F_{\frak{Q}}[\alpha] & \cdots &\alpha F_{\frak{Q}}[\alpha]\\
\end{pmatrix}}.
$$
 By  identifying $1\otimes 1\otimes F_{\frak{Q}}$ with $F_{\frak{Q}}$, we get  
 $\mathcal{A}'_{\frak{Q }}  =    \mathcal{A}_{S\otimes_{S_E}S_{\frak{Q}}}(P_{E, \frak{Q}})\otimes_{S_{\frak{Q}}}F_{\frak{Q}}  \cong  \mathcal{R}'_{\frak{Q}}/\mathcal{I}'_{\frak{Q}} $.

By Lemma \ref{FQ} we have that   $F_{\frak{Q}}(\alpha) \cong  F_{\frak{Q}}[x]/(x^{|E|} )$, is a local ring with maximal ideal $(\alpha)$. Since $\mathcal{A}'_{\frak{Q}}$ is finite, it is Artinian and is a direct sum of the left ideals $\mathcal{A}'_{\frak{Q}}e_{i, i}, 1\leq i \leq |E|$, where $e_{i, i}$ is the $|E| \times |E|$
matrix with a $1$ in the $(i, i)$ position and zeroes elsewhere.  
 It is not difficult then to see that the Radical of the algebra $\mathcal{A}'_{\frak{Q}} $,   is the ideal 
$$
Rad(\mathcal{A}'_{\frak{Q}} ) = 
  \frac{
\begin{pmatrix}
\alpha F_{\frak{Q}}[\alpha]&F_{\frak{Q}}[\alpha]  & F_{\frak{Q}}[\alpha] &\cdots & F_{\frak{Q}}[\alpha] \\
\alpha F_{\frak{Q}}[\alpha] &\alpha F_{\frak{Q}}[\alpha] & F_{\frak{Q}}[\alpha] & \cdots  & F_{\frak{Q}}[\alpha]\\
\alpha^2 F_{\frak{Q}}[\alpha]&\alpha F_{\frak{Q}}(\alpha) &\alpha F_{\frak{Q}}[\alpha] & \dots \dots & F_{\frak{Q}}[\alpha]\\
\alpha^3 F_{\frak{Q}}[\alpha]& \alpha^2 F_{\frak{Q}}(\alpha]  & \alpha F_{\frak{Q}}[\alpha] & \dots \dots & F_{\frak{Q}}[\alpha]\\
\vdots & \vdots & \vdots & \ & \\
\alpha^{|E| -1} F_{\frak{Q}}[\alpha] &\alpha^{|E| - 2} F_{\frak{Q}}[\alpha] &\alpha^{|E| - 3} F_{\frak{Q}}[\alpha] &\ & \alpha F_{\frak{Q}}[\alpha]
\end{pmatrix}}
{\mathcal{I}'_{\frak{Q}} }$$

Hence we have 
$$\mathcal{A}'_{\frak{Q}}  /Rad(\mathcal{A}'_{\frak{Q}} ) \cong 
\begin{pmatrix}
F_{\frak{Q}} & 0 & 0 & \cdots & 0\\
0 & F_{\frak{Q}} & 0 & \cdots & 0\\
0 &0&  F_{\frak{Q}}  & \cdots & 0\\
\vdots & \vdots & \vdots  & & \vdots \\
0 & 0 & 0 & \cdots & F_{\frak{Q}}
\end{pmatrix},
$$
which is a direct sum of simple modules $L(e_i) = \mathcal{A}'_{\frak{Q}}e_{i, i}/Rad( \mathcal{A}'_{\frak{Q}}e_{i, i})$, $1\leq i \leq |E|$.  
Each $L(e_i)$    can be identified with the  copy of $F_{\frak{Q} }$ in the ith position on the diagonal of the matrix above. 
Since $L(e_i)$ is indecomposable for $1 \leq i \leq |E|$, we have $\mathcal{A}'_{\frak{Q}}e_{i, i}$ is 
also an  indecomposable $ \mathcal{A}'_{\frak{Q}}$ module, for $1 \leq i \leq |E|$, by Nakayama's lemma \cite{CnR}[(5.7)]. The set $\{L(e_i) | 1 \leq i \leq |E| \}$ is a full set of 
representatives of the isomorphism classes of the simple modules in
$\mathcal{A}'_{\frak{Q}}$-mod, \cite{CnR}[Sections 6A-6C]. 

Recall, from sections 7 and 8 that the module $Hom(P_{E, \frak{Q}}, P_{E, \frak{Q}, i})\otimes_{S_{\frak{Q}}}F_{\frak{Q}}$ in 
$End(P_{E, \frak{Q}})$-mod is identified with the  $i$ th column of the matrix ring  $\mathcal{A}'_{\frak{Q}}$, 
$ \mathcal{A}'_{\frak{Q}}e_{i, i}$.   Hence we have 
that 
$$L(e_i) \cong \frac{Hom(P_{E, \frak{Q}}, P_{E, \frak{Q}, i})\otimes_{S_{\frak{Q}}}F_{\frak{Q}} }{ Rad(Hom(P_{E, \frak{Q}}, P_{E, \frak{Q}, i})\otimes_{S_{\frak{Q}}}F_{\frak{Q}})} \cong \frac{ \mathcal{A}'_{\frak{Q}}e_{i, i}}{Rad( \mathcal{A}'_{\frak{Q}}e_{i, i})}.$$
Before we make our choice of Verma modules, we will show that under this identification, the module $Hom(P_{E, \frak{Q}}, P_{E, \frak{Q}, i}^{i + 1})\otimes_{S_{\frak{Q}}}F_{\frak{Q}}$ is identified with  the following submodule of $ \mathcal{A}'_{\frak{Q}}e_{i, i}$:
$$X_i = 
\frac{\begin{matrix}
1\ st \  row \rightarrow  \\
\\
\vdots\\
\\
  i- th \ \ row \rightarrow \\
  \\  
  \\
  \vdots\\
|E|\  th \ row \rightarrow \\
\end{matrix}
\begin{pmatrix}
0&\cdots&0& \alpha F_{\frak{Q}}[\alpha]&0&\cdots&0\\
0&\cdots&0&\alpha   F_{\frak{Q}}[\alpha]&0&\cdots&0\\
  \vdots \\
0&\cdots&0& \alpha  F_{\frak{Q}}[\alpha]&0&\cdots&0\\
 0&\cdots&0& \alpha  F_{\frak{Q}}[\alpha]&0&\cdots&0\\
0&\cdots&0&   \alpha F_{\frak{Q}}[\alpha]&0&\cdots&0\\  
0&\cdots&0&   \alpha^2 F_{\frak{Q}}[\alpha]&0&\cdots&0\\
   \vdots\\
0&\cdots&0&   \alpha^{|E| - i} F_{\frak{Q}}(\alpha)&0&\cdots&0
\end{pmatrix}  + \mathcal{I}'_{\frak{Q}}}
{\mathcal{I}'_{\frak{Q}}} .$$
\vskip .2in
\noindent
This follows from the following lemma  since $A_{e_k}(\Pi) \otimes_{S_{\frak{Q}}}F_{\frak{Q}}  = \alpha$:

\begin{lemma} Let $S, E, \Omega',  \frak{Q}$ and $ S_E$ be as defined in section 2.  Let $\Pi, P_{E, \frak{Q},}$
and $P_{E, \frak{Q}, i}$, $1 \leq i \leq |E|$ be as  above. 
Let 
$$P_{E, \frak{Q}, i} = P_{E, \frak{Q}, i}^{e_{i - 1}}  \supseteq P_{E, \frak{Q}, i}^{e_i } \supseteq \cdots \supseteq P_{E, \frak{Q}, i}^{e_{|E|}} = \{0\}$$
be a filtration for $P_{E, \frak{Q}, i}$ in the category $\mathcal{C}_{(S\otimes_{S_E}S_{\frak{Q}}, \Omega', E)}$ (as in Corollary  \ref{PlusEQ}) . 
Let $A_{e_k}(\Pi) = \Pi\otimes 1 - 1 \otimes\sigma_{e_k}(\Pi)  \in S\otimes_{S_E}S_{\frak{Q}}$. 
Then
$$Hom_{S\otimes_{S_E}S_{\frak{Q}}}(P_{E, \frak{Q}, k }, P_{E, \frak{Q}, i}^{e_i})  
= 
 \Big\{\  \begin{matrix}
& A_{e_k}(\Pi)Hom_{S\otimes_{S_E}S_{\frak{Q}}}(P_{E, \frak{Q}, k }, P_{E, \frak{Q}, i}^{e_{i - 1} })   & \hbox{if}  & k \leq i  \\
& Hom_{S\otimes_{S_E}S_{\frak{Q}}}(P_{E, \frak{Q}, k }, P_{E, \frak{Q}, i}^{e_{i - 1} })   & \hbox{if}  & k >  i
\end{matrix}.
$$

\end{lemma}

\begin{proof} Throughout the proof $Hom(-, -)$ should be interpreted as $Hom_{S\otimes_{S_E}S_{\frak{Q}}}(-, -)$.
If $k > i $, then since 
$$P_{E, \frak{Q}, k } = \frac{S\otimes_{S_E}S_{\frak{Q}}}{(S\otimes_{S_E}S_{\frak{Q}})_{I_{E, e_k}^c}}
\ \ \ \mbox{and} \ \ \ \  P_{E, \frak{Q}, i } = \frac{S\otimes_{S_E}S_{\frak{Q}}}{(S\otimes_{S_E}S_{\frak{Q}})_{I_{E, e_i}^c}},$$
by Lemma \ref{ringmaps} and Theorem \ref{ringmapsstensors},   we have $Hom(P_{E, \frak{Q}, k }, P_{E, \frak{Q}, i }) \cong \frac{ (S\otimes_{S_E}S_{\frak{Q}})_{I_{E, e_i}^c\cup I_{E, e_k}}   }{(S\otimes_{S_E}S_{\frak{Q}})_{I_{E, e_i}^c}}$, with the isomorphism from $Hom(P_{E, \frak{Q}, k }, P_{E, \frak{Q}, i}) $ to 
$ \frac{ (S\otimes_{S_E}S_{\frak{Q}})_{I_{E, e_i}^c\cup I_{E, e_k}}   }{(S\otimes_{S_E}S_{\frak{Q}})_{I_{E, e_i}^c}}$
 given by $f \to f(1\otimes 1 + (S\otimes_{S_E}S_{\frak{Q}})_{I_{E, e_i}^c})$.    By looking at  the images of the modules under  the map $\phi'$,  and by uniqueness of the filtration of   $P_{E, \frak{Q}, i }$ in  $\mathcal{C}_{(S\otimes_{S_E}S_{\frak{Q}}, \Omega', E)}$ (using Lemma  \ref{maps}  applied to the identity map), 
 we can easily see that $P_{E, \frak{Q}, i }^{e_i} =
 \frac{ (S\otimes_{S_E}S_{\frak{Q}})_{I_{E, e_i}^c\cup I_{E, e_{i + 1}}}   }{(S\otimes_{S_E}S_{\frak{Q}})_{I_{E, e_i}^c}}$.
 Therefore, for $k \geq i + 1$, we have that $f(P_{E, \frak{Q}, k } ) \subseteq P_{E, \frak{Q}, i }^{e_i}$ for all $f \in Hom(P_{E, \frak{Q}, k }, P_{E, \frak{Q}, i })$. 
 Hence $Hom(P_{E, \frak{Q}, k }, P_{E, \frak{Q}, i }) = Hom(P_{E, \frak{Q}, k }, P_{E, \frak{Q}, i }^{e_i})$ for $k \geq i + 1$.

It remains to treat the case when $k \leq i$.  If $k \leq i$, then $I_{E, e_i}^c\cup I_{E, e_k} = E$, and  $Hom(P_{E, \frak{Q}, k}, P_{E, \frak{Q}, i}) \cong \frac{ (S\otimes_{S_E}S_{\frak{Q}})  }{(S\otimes_{S_E}S_{\frak{Q}})_{I_{E, e_i}^c}}$ where the isomorphism, from 
 Lemma \ref{ringmaps},  is given by $f \to f(1\otimes 1  + (S\otimes_{S_E}S_{\frak{Q}})_{I_{E, e_i}^c})$.
Now consider the 
  short exact sequence of $S\otimes_{S_E}S_{\frak{Q}}$ modules:
 $$0 \to P_{E, \frak{Q}, i}^{e_i} \to P_{E, \frak{Q}, i} \to P_{E, \frak{Q}, i}/P_{E, \frak{Q}, i}^{e_i} \to 0.$$
 It is not difficult to see that it is a sheaf exact sequence in the category $\mathcal{C}_{(S\otimes_{S_E}S_{\frak{Q}}, \Omega', E)}$.
 Because $P_{E, \frak{Q}}$ is a projective generator of $\mathcal{C}_{(S\otimes_{S_E}S_{\frak{Q}}, \Omega', E)}$,
  the sequence 
 $$0 \to Hom(P_{E, \frak{Q}}, P_{E, \frak{Q}, i}^{e_i}) \to Hom(P_{E, \frak{Q}}, P_{E, \frak{Q}, i} ) \to Hom(P_{E, \frak{Q}},P_{E, \frak{Q}, i}/P_{E, \frak{Q}, i}^{e_i} ) \to 0$$
 is an exact sequence of $\mathcal{A}'(P_{E, \frak{Q}})$ modules, by Theorem \ref{progen}, 
   and hence exact as a sequence of 
$S\otimes_{S_E}S_{\frak{Q}}$ modules. 
 Now if $M$ is an $S\otimes_{S_E}S_{\frak{Q}}$ module, we have 
 $Hom(P_{E, \frak{Q}}, M) \cong \oplus_j Hom(P_{E, \frak{Q}, j}, M)$ as $S\otimes_{S_E}S_{\frak{Q}}$ modules 
 and therefore 
{  \[
\xymatrix @=1pc {
0\ar[r]&Hom(P_{E,\frak{Q},k }, P_{E,\frak{Q},i}^{e_i})\ar[r]^{i} & Hom(P_{E,\frak{Q},k }, P_{E,\frak{Q},i})\ar[r]^{F}& Hom(P_{E, \frak{Q},k },\frac{P_{E,\frak{Q},i}}{P_{E,\frak{Q},i}^{e_i}})\ar[r]&0
}
\]}
is an exact sequence of $S\otimes_{S_E}S_{\frak{Q}}$ modules for each $k$. 
  Hence 
 $\ker F = Hom(P_{E, \frak{Q}, k }, P_{E, \frak{Q}, i}^{e_i})$ since $i$ is the inclusion map. Certainly $A_{e_i}(\Pi)Hom(P_{E, \frak{Q}, k }, P_{E, \frak{Q}, i })
 \subseteq \ker F$, since $A_{e_i}(\Pi)(S\otimes_{S_E}S_{\frak{Q}}) = (S\otimes_{S_E}S_{\frak{Q}})_{I_{E, e_i}^c\cup I_{E, e_{i + 1}}}$, and from our discussion above, $P_{E, \frak{Q}, i }^{e_i} =
 \frac{ (S\otimes_{S_E}S_{\frak{Q}})_{I_{E, e_i}^c\cup I_{E, e_{i + 1}}}   }{(S\otimes_{S_E}S_{\frak{Q}})_{I_{E, e_i}^c}}$.
 
   On the other hand, if 
 $f \in \ker F$, then $f(1\otimes 1 + (S\otimes_{S_E}S_{\frak{Q}})_{I_{E, e_k}^c}) = 
 A_{e_i}(\Pi)x + (S\otimes_{S_E}S_{\frak{Q}})_{I_{E, e_i}^c}$, for some $x \in S\otimes_{S_E}S_{\frak{Q}}$, since 
  $f(1\otimes 1 + (S\otimes_{S_E}S_{\frak{Q}})_{I_{E, e_k}^c}) \in  A_{e_i}(\Pi)(S\otimes_{S_E}S_{\frak{Q}}) + (S\otimes_{S_E}S_{\frak{Q}})_{I_{E, e_i}^c} = (S\otimes_{S_E}S_{\frak{Q}})_{I_{E, e_i}^c\cup I_{E, e_{i + 1}}} + (S\otimes_{S_E}S_{\frak{Q}})_{I_{E, e_i}^c}$.  Since $k \leq i$, we have $Hom(P_{E, \frak{Q}, k }, P_{E, \frak{Q}, i}) 
  = \frac{S\otimes_{S_E}S_{\frak{Q}}} {(S\otimes_{S_E}S_{\frak{Q}})_{I_{E, e_k}^c} }$, hence there is a 
  homomorphism $g \in Hom(P_{E, \frak{Q}, k }, P_{E, \frak{Q}, i})$  with $g(1\otimes 1 + (S\otimes_{S_E}S_{\frak{Q}})_{I_{E, e_k}^c}) 
  = x + (S\otimes_{S_E}S_{\frak{Q}})_{I_{E, e_i}^c}) $ and $f = A_{e_i}(\Pi)g$. This completes the result.
 
 \end{proof}

We now make a choice of Verma modules in the category $\mathcal{A}'_{\frak{Q}}$-mod. 
\begin{definition}
Let 
$M(e_k) = Hom_{S\otimes_{S_E}S_{\frak{Q}}}(P_{E, \frak{Q}}, S_{\frak{Q}}[\sigma_{e_k}])\otimes_{S_{\frak{Q}}}F_{\frak{Q}}$ for $1 \leq k \leq |E|$.  
\end{definition}
\begin{lemma}The set $\{M(e_k) | 1 \leq k \leq |E| \}$ constitute a choice of Verma modules for the category $\mathcal{A}'_{\frak{Q}}$-mod.
In particular $\overline{M(e_k)} = M(e_k)/Rad (M(e_k)) =  L(e_k)$ and 
 $$[M(e_k) : L(e_j)] =  \Big\{\  \begin{matrix}
& 1 & \hbox{if}  & k \geq j \\
& 0 &   &\mbox{otherwise} 
\end{matrix}.
$$

The modules also satisfy the universal property:\\
If  $\{M'(e_k) | 1 \leq k \leq |E| \}$  is another set of modules in the category  $\mathcal{A}'_{\frak{Q}}$-mod, with the following properties:\\
$\overline{M'(e_k)} = M'(e_k)/Rad (M'(e_k)) =  L(e_k)$, $[M'(e_k) : L(e_j)] = 1$ if $k = j$ and $[M'(e_k) : L(e_j)] = 0$ if $j >   k$, \\
then for each $k, \ 1\leq k \leq |E|$, we have a surjective homomorphism $M(e_k) \to M'(e_k)$.

\end{lemma}
\begin{proof}
As discussed in the proof of the previous Lemma, 
for any $k$, $1\leq k \leq |E|$, we have an exact sequence of $End_{S\otimes_{S_E}S_{\frak{Q}}}(P_{E, \frak{Q}}) $ modules:
 $$0 \to Hom(P_{E, \frak{Q}}, P_{E, \frak{Q}, k }^{e_k}) \to Hom(P_{E, \frak{Q}}, P_{E, \frak{Q}, k }) \to Hom(P_{E, \frak{Q}},P_k/P_{E, \frak{Q}, k }^{e_k} ) \to 0,$$
 where $Hom(-,-)  = Hom_{S\otimes_{S_E}S_{\frak{Q}}}(-,-) $. By using Lemma \ref{ringmaps} and the isomorphism $\phi$, we see that 
 each module in this exact sequence is projective, and hence free,  as a right $S_{\frak{Q}}$ module. Hence the above sequence is 
 a split exact sequence of right $S_{\frak{Q}}$ modules, which remains exact when we apply the functor $\otimes_{S_{\frak{Q}} }F_{\frak{Q}}$. 
 Since $P_{E, \frak{Q}, k }/P_{E, \frak{Q}, k }^{e_k}  \cong S_{\frak{Q}}[\sigma_{e_k}]$ we have 
 $$M(e_k) \cong \frac{Hom_{S\otimes_{S_E}S_{\frak{Q}}}(P_{E, \frak{Q}}, P_{E, \frak{Q}, k })\otimes_{S_{\frak{Q}}}F_{\frak{Q}} }{Hom_{S\otimes_{S_E}S_{\frak{Q}}}(P_{E, \frak{Q}}, P_{E, \frak{Q}, k }^{e_k})\otimes_{S_{\frak{Q}}}F_{\frak{Q}} }. $$
Using the identification of $Hom_{S\otimes_{S_E}S_{\frak{Q}}}(P_{E, \frak{Q}}, P_{E, \frak{Q}, k })$ with the kth column of the matrix ring $\mathcal{R}_{\frak{Q}}' / \mathcal{I}_{\frak{Q}}' $ , we see that $M(e_k)$ is 
identified with 
 $$\mathcal{A}'_{\frak{Q}}e_{k, k}/X_k.$$
Now it is easy to  see that $\overline{M(e_k)} = L(e_k)$ and that 
 
 $$[M(e_k) : L(e_j)] =  \Big\{\  \begin{matrix}
& 1 & \hbox{if}  & k \geq j \\
& 0 &   &\mbox{otherwise} 
\end{matrix}.
$$
 Thus $\{M(e_k) | 1 \leq k \leq |E| \}$ constitutes a choice of Verma modules for the category $\mathcal{A}'_{\frak{Q}}$-mod.

 To prove the universal property for the set $\{M(e_k) | 1 \leq k \leq |E| \}$, we first show that each projective indecomposable has 
 a Verma flag for this choice of Verma modules. In the process we prove what we need to demonstrate BGG reciprocity in the category  $\mathcal{A}'_{\frak{Q}}$-mod.

Let $P(e_i) = Hom_{S\otimes_{S_E}S_{\frak{Q}}}(P_{E, \frak{Q}}, P_{E, \frak{Q}, i})\otimes_{S_{\frak{Q}}}F_{\frak{Q}}$ for $1\leq i \leq |E|$. 
Since $P_{E, \frak{Q}, i}$ is projective in  $ \mathcal{C}_{(S\otimes_{S_E}S_{\frak{Q}}, \ \Omega', E)}$, 
we have $Hom_{S\otimes_{S_E}S_{\frak{Q}}}(P_{E, \frak{Q}}, P_{E, \frak{Q}, i})$  is projective in $End_{S\otimes_{S_E}S_{\frak{Q}}}(P_{E, \frak{Q}})^{op}$ - mod. Hence $P(e_i) = Hom_{S\otimes_{S_E}S_{\frak{Q}}}(P_{E, \frak{Q}}, P_{E, \frak{Q}, i})\otimes_{S_{\frak{Q}}}F_{\frak{Q}}$ is projective and is    isomorphic to  
 $\mathcal{A}'_{\frak{Q}}e_{i, i}$ as an $\mathcal{A}'_{\frak{Q}}$ module.
 Furthermore, from the discussion above  we know  that 
$$\overline{P(e_i)} = P(e_i)/(Rad (P(e_i)) \cong \mathcal{A}'_{\frak{Q}}e_{i, i}/Rad(\mathcal{A}'_{\frak{Q}}e_{i, i})   \cong L(e_i)$$ 
 and $P(e_i)$ is an  projective cover of $L(e_i)$. Thus $\{P(e_i)|1 \leq i \leq |E|\}$ is a set of projective covers for the simple modules in
 $\mathcal{A}'_{\frak{Q}}$-mod. 
 
\begin{lemma} \label{turkey} Let $\{M(e_i)|1\leq i \leq |E|\}$ be the choice of Verma modules for the category $\mathcal{A}'_{\frak{Q}}$-mod, defined above.
Let $\{P(e_i)|1\leq i \leq |E|\}$ be the projective indecomposables described above. Then each $P(e_i)$, $1 \leq i \leq |E|$ has a Verma flag with multiplicities:
$$(P(e_i) : M(e_k) ) = 
 \Big\{\  \begin{matrix}
& 1 & \hbox{if}  & k \geq i  \\
& 0 &   &\mbox{otherwise} 
\end{matrix}.
$$
\end{lemma}
\begin{proof}
For a fixed $i$, with $1 \leq i \leq |E|$, consider the module $P_{E, \frak{Q}, i}$ in 
$\mathcal{C}_{(S\otimes_{S_E}S_{\frak{Q}}, \Omega', E)}$. 
By Corollary  \ref{PlusEQ}, we have a filtration
in $\mathcal{C}_{(S\otimes_{S_E}S_{\frak{Q}}, \Omega', E)}$  for 
$P_{E, \frak{Q}, i}$ given by 
\[(P_{E, \frak{Q}, i})^{e_0}  = P_{E, \frak{Q}, i}           = \frac{(S\otimes_{S_E}S_{\frak{Q}}) }{ (S\otimes_{S_E}S_{\frak{Q}})_{I_{E, e_i}^c}} 
 = (P_{E, \frak{Q}, i})^{e_{i - 1}} \ \ \mbox{ and} \ \  \]
\[ (P_{E, \frak{Q}, i})^{e_k} = \frac{(S\otimes_{S_E}S_{\frak{Q}})_{I_{E, e_i}^c\cup I_{E, e_{k + 1}}} }{ (S\otimes_{S_E}S_{\frak{Q}})_{I_{E, e_i}^c}}.\]
For $k \geq i$ we  have that $ (P_{E, \frak{Q}, i})^{e_{k - 1}} / (P_{E, \frak{Q}, i})^{e_k} \cong 
\phi'_k(S\otimes_{S_E}S_{\frak{Q}})_{I^c_{E, e_i} \cup I_{E, e_k}}  = s_kS_{\frak{Q}}[\sigma_k]$, for some $s_k \in S_{\frak{Q}}$. 
Hence 
$ (P_{E, \frak{Q}, i})^{e_{k - 1}} / (P_{E, \frak{Q}, i})^{e_k}   \cong S_{\frak{Q}}[\sigma_k]$ as $S\otimes_{S_E}S_{\frak{Q}}$ modules. 
It is not difficult to see that the sequence 
$$0 \to  (P_{E, \frak{Q}, i})^{e_k} \to  (P_{E, \frak{Q}, i})^{e_{k - 1}}  \to  (P_{E, \frak{Q}, i})^{e_{k - 1}}/ (P_{E, \frak{Q}, i})^{e_k} 
\to 0$$
is sheaf exact in  $\mathcal{C}_{(S\otimes_{S_E}S_{\frak{Q}}, \Omega', E)}$  for $1 \leq k \leq |E|$. 
Since $Hom_{S\otimes_{S_E}S_{\frak{Q}}}(P_{E, \frak{Q}}, -)$ takes sheaf exact sequences to exact sequences, 
we see that we have isomorphisms of $End_{S\otimes_{S_E}S_{\frak{Q}}}(P_{E, \frak{Q}})$ modules for $k\geq 1$:
$$\frac{ Hom_{S\otimes_{S_E}S_{\frak{Q}}}(P_{E, \frak{Q}}, (P_{E, \frak{Q}, i})^{e_{k - 1}} )}{       Hom_{S\otimes_{S_E}S_{\frak{Q}}}(P_{E, \frak{Q}}, (P_{E, \frak{Q}, i})^{e_k}  ) } \cong    
 Hom_{S\otimes_{S_E}S_{\frak{Q}}}\Bigg(P_{E, \frak{Q}},    \frac{  (P_{E, \frak{Q}, i})^{e_{k -1}}}{ (P_{E, \frak{Q}, i})^{e_k}}\Bigg) $$
 $$
 \cong  Hom_{S\otimes_{S_E}S_{\frak{Q}}}(P_{E, \frak{Q}}, S_{\frak{Q}}[\sigma_{e_k}])
 $$
Now as in the previous Lemma,  since the Hom-spaces above are projective as right $S_{\frak{Q}}$ modules, tensoring by $F_{\frak{Q}}$ is exact, giving  isomorphisms   of $\mathcal{A}'_{\frak{Q}}$  modules:
{\small $$
\frac{ Hom_{S\otimes_{S_E}S_{\frak{Q}}}(P_{E, \frak{Q}}, (P_{E, \frak{Q}, i})^{e_{k - 1}}) \otimes_{S_{\frak{Q}}}F_{\frak{Q}}  }{       Hom_{S\otimes_{S_E}S_{\frak{Q}}}(P_{E, \frak{Q}}, (P_{E, \frak{Q}, i})^{e_k}   )\otimes_{S_{\frak{Q}}}F_{\frak{Q}}  } 
\cong 
\frac{ Hom_{S\otimes_{S_E}S_{\frak{Q}}}(P_{E, \frak{Q}}, (P_{E, \frak{Q}, i})^{e_{k - 1}} )}{       Hom_{S\otimes_{S_E}S_{\frak{Q}}}(P_{E, \frak{Q}}, (P_{E, \frak{Q}, i})^{e_k}   )} \otimes_{S_{\frak{Q}}}F_{\frak{Q}}
$$}

$$\cong   Hom_{S\otimes_{S_E}S_{\frak{Q}}}(P_{E, \frak{Q}}, S_{\frak{Q}}[\sigma_{e_k}])  \otimes_{S_{\frak{Q}}}F_{\frak{Q}}  = M(e_k)
$$
if $k \geq i$.  Letting 
$P(e_i)_k = Hom_{S\otimes_{S_E}S_{\frak{Q}}}(P_{E, \frak{Q}}, (P_{E, \frak{Q}, i})^k )\otimes_{S_{\frak{Q}}}F_{\frak{Q}}$ for $k \geq i - 1$ , we get a Verma flag 
$$P(e_i) = P(e_i)_{i - 1}  \supset P(e_i)_{i } \supset \cdots \supset P(e_i)_{|E|} = 0$$
for $P(e_i)$, giving us the multiplicities 
$$(P(e_i) : M(e_k) ) = 
 \Big\{\  \begin{matrix}
& 1 & \hbox{if}  & k \geq i  \\
& 0 &   &\mbox{otherwise} 
\end{matrix}.
$$
\end{proof}
We are now ready to demonstrate the universal property of the set $\{M(e_k) | 1 \leq k \leq |E| \}$. Let 
$\{M'(e_k) | 1 \leq k \leq |E| \}$  is another set of modules in the category  $\mathcal{A}'_{\frak{Q}}$-mod, with the following properties:\\
$\overline{M'(e_k)} = M'(e_k)/Rad (M'(e_k)) =  L(e_k)$, $[M'(e_k) : L(e_j)] = 1$ if $k = j$ and $[M'(e_k) : L(e_j)] = 0$ if $j  >   k$. \\
Then for each $k, 1 \leq k \leq |E|$, we have an exact sequence
$$0 \to Rad(M'(e_k)) \to M'(e_k) \to L(e_k) \to 0.$$
Because $P(e_k)$ is projective, with projection   $\pi: P(e_k) \to L(e_k)$,  we have a map $\tau : P(e_k) \to M'(e_k)$ making the following diagram commute: 
\[\xymatrix{
&&
&P(e_k)\ar[d]^{\pi}\ar[dl]_{\tau}&\\
0\ar[r]&Rad(M'(e_k)) \ar[r] & M'(e_k) \ar[r]^p & L(e_k) \ar[r] & 0
}\]

Let $M_1 = \tau(P(e_k))$ denote the image of the map $\tau$. For any $m \in M'(e_k)$, we have $P(m) = P(\tau(p_1))$ for some $p_1 \in P(e_k)$, 
since $\pi$ is surjective.  Hence $M_1 + Rad(M'(e_k)) = M'(e_k)$. By Nakayama's lemma,  \cite{CnR}[30.2], we have $M_1 = M'(e_k)$ and the map $\tau$ is surjective.

Let $P(e_k) = P(e_k)_{k - 1} \supset P(e_k)_{k } \supset \cdots \supset P(e_k)_{|E|} = 0$ be the  Verma flag for $P(e_k)$ with respect to the set 
 $\{M(e_k) | 1 \leq k \leq |E| \}$ described  in the proof of Lemma \ref{turkey} above. We have 
 $$(P(e_k) : M(e_j) ) = 
 \Big\{\  \begin{matrix}
& 1 & \hbox{if}  & j \geq k  \\
& 0 &   &\mbox{otherwise} 
\end{matrix}
$$
and $P(e_k)_{k + i - 1}/P(e_k)_{k + i } \cong M(e_{k + i})$. 

We will show by induction that $\tau(P(e_k)_{k}) = 0$ and
this gives the desired surjection $\bar{\tau} : M(e_k) \cong P(e_k)/P(e_k)_k \to M'(e_k)$. 

If $k = |E|$, then $P(e_k)_k = P(e_k)_{|e|} = 0$ and the result is automatically true.

If $k < |E|$, then $P(e_k)_{|E| - 1} \cong M(e_{|E|})$. Let $\tau(P(e_k)_{|E| - 1}) = M_1 \subseteq  M'(e_k)$.
Since $P(e_k)_{|E| - 1}/Rad (P(e_k)_{|E| - 1}) \cong L(e_{|E|})$, we have $M_1/\tau(Rad (P(e_k)_{|E| - 1}) )$ is either isomorphic to 
$L_{|E|}$ or trivial. Since $L_{|E|}$ does not appear in the composition series of $M'(e_k)$, $M_1$ cannot have a quotient isomorphic to 
$L_{|E|}$.  Hence $M_1 = \tau(Rad (P(e_k)_{|E| - 1})  \subseteq Rad(M_1)$. Hence $Rad(M_1) = M_1$ and $M_1 = 0$ by Nakayama's lemma 
 \cite{CnR}[(5.7)]. 
 
By the same argument, we can show that if $\tau(P(e_k)_s) = 0$ and $k \leq s - 1$, then  $\tau(P(e_k)_{s - 1}) = 0$. Hence , by induction,
$\tau(P(e_k)_k) = 0$ giving us the universal property of $M(e_k)$ and our chosen set of Verma modules. 
\end{proof}

This gives us BGG reciprocity for each category $ \mathcal{C}_{(S\otimes_RS_{\frak{Q}}, \ \Omega|_{E\sigma_{x_i}}, 
G)}$, for each such category associated to the coset $E\sigma_{x_i}$ of $E$ in $G$. Now we can simply take the union of the representatives of the isomorphism classes of the simple modules, their projective covers and their choice of Verma modules for each category to get a set of representatives of the isomorphism classes of simple modules, their projective covers and a choice of Verma modules for the category $ \mathcal{C}_{(S\otimes_RS_{\frak{Q}}, \ \Omega, 
G)}$, since $ \mathcal{C}_{(S\otimes_RS_{\frak{Q}}, \ \Omega, 
G)}$ is the direct product of the categories  $ \mathcal{C}_{(S\otimes_RS_{\frak{Q}}, \ \Omega|_{E\sigma_{x_i}}, 
G)}$.  Putting the results together we get 

\begin{theorem}  Let $L, K, S, R, \frak{Q}, \frak{P}, G = \{\sigma_{g_1}, \sigma_{g_2}, \cdots , \sigma_{g_n}\}$, and  $E = E(\frak{Q}|\frak{P})$ be as in Section 2. 
Let   $\Omega$ be the poset giving the ordering $g_1 < g_2 < g_3 < \dots < g_n$ on the indices of $G$.  Let $P = P_1 \oplus P_2 \oplus \dots \oplus P_n$, where $P_i, 1 \leq i \leq n$ are as defined in Definition \ref{P}. 
Let $A_{S\otimes_R S}(P)$  be the algebra associated to 
$ \mathcal{C}_{(S\otimes_RS, \ \Omega,  G) }$. Let $A_{\frak{Q}} = A_{S\otimes_R S}(P)\otimes_S F_{\frak{Q}}$, where $F_{\frak{Q}}$ is the residue class field 
$S/{\frak{Q}}$. 
We can choose a set of representatives of the isomorphism classes of the simple modules in $A_{\frak{Q}}$ - mod, $\{L(\sigma_{g_i})\}$ which are in one to one correspondence with the elements of 
$G$. We can also have a set of corresponding projective covers $\{P(\sigma_{g_i})\}$, and a chice of Verma modules $\{M(\sigma_{g_i})\}$ such that the following reciprocity law holds:
$$[M(g_k): L(g_i)] = (P(g_i) : M(g_k) ) = 
 \Big\{\  \begin{matrix}
& 1 & \hbox{if}  & k \geq i \ \ \hbox{and} \ \ E\sigma_{g_i} = E\sigma_{g_k} \\
& 0 &   &\mbox{otherwise} 
\end{matrix}.
$$

\end{theorem}


\begin{thebibliography}{1}

\bibitem{Brown}
Brown, Karen S.
 \newblock Extensions of ``thickened" Verma modules of the Virasora Algebra\\
\newblock { J. Algebra }, 269(2003), no. 1, 160-168.

\bibitem{CPS}
Cline, E. ; Parshall, B. ; Scott, L. 
 \newblock Finite Dimensional Algebras and highest weight categories.
\newblock { J. Reine Angew. Math. }, 391(1988), 85-99.

\bibitem{CnR}
 Curtis, C. W., Reiner, I.
\newblock {\em Methods of Representation Theory, Volume 1}. 
\newblock Wiley Classics Library, 1990 edition, John Wiley and Sons.

\bibitem{Dyer}
 Dyer, M.
\newblock Stratified Exact Categories and Highest Weight Representations. 
\newblock Preprint, http://www.nd.edu/~dyer/papers/stratcat.pdf

\bibitem{Dyer2}
 Dyer, M.
 \newblock Representation theories from Coxeter groups
\newblock { Canadian Mathematical Society}, 19(1985), 247--270.

\bibitem{HandOM}
Hahn, A., O'Meara, O. T.,  
\newblock {\em The Classical Groups and K-Theory}.
\newblock Springer-Verlag

\bibitem{Irving}
Irving, Ronald S.
 \newblock BGG algebras and the BGG reciprocity principle.
\newblock { J. Algebra }, 135(1990), no. 2, 363-380.

\bibitem{Jac1}
Jacobson, N.
\newblock {\em Basic Algebra I}.
\newblock W. H. Freeman and Company.

\bibitem{Jac2}
Jacobson, N.
\newblock {\em Basic Algebra II}. 
\newblock W. H. Freeman and Company.

\bibitem{Lang}
Lang, S.
\newblock {\em Algebraic Number Theory}.
\newblock  Graduate Texts in Mathematics. Springer-Verlag.

\bibitem{Marcus}
Marcus, D. A.
\newblock {\em Number Fields}.
\newblock Springer, Universitext.

\bibitem{Quillen}
Quillen, D. 
 \newblock Higher Algebraic K-Theory I. 
\newblock { Lect. Notes in Math.  },  Vol 341(1973), Springer, 77-139.

\bibitem{Reiner}
Reiner, I.
\newblock {\em Maximal Orders}.
\newblock 1975 edition, Academic press.

\end{thebibliography}
\end{document}